\documentclass[11pt,letterpaper]{amsart}
\usepackage[foot]{amsaddr}

\usepackage{mathtools}
\usepackage{amsmath}
\usepackage[T1]{fontenc}
\usepackage[latin9]{inputenc}
\usepackage{geometry}
\usepackage{wrapfig}
\usepackage{multicol}
\usepackage{graphicx}
\usepackage{soul}
\usepackage{xcolor}
\usepackage{amssymb}
\usepackage{placeins}
\usepackage{bbm}
\usepackage{cite}
\usepackage{caption}
\usepackage{enumerate}
\usepackage{afterpage}
\usepackage{enumitem}
\usepackage{bmpsize}
\usepackage{hyperref}
\usepackage{tabu}
\usepackage{enumitem}
\numberwithin{equation}{section}
\usepackage{stmaryrd}
\usepackage{tikz}
\usetikzlibrary{matrix,graphs,arrows,positioning,calc,decorations.markings,shapes.symbols}
\usetikzlibrary{calc,intersections,angles,quotes,arrows.meta}

\makeatletter
\renewcommand{\email}[2][]{%
  \ifx\emails\@empty\relax\else{\g@addto@macro\emails{,\space}}\fi%
  \@ifnotempty{#1}{\g@addto@macro\emails{\textrm{(#1)}\space}}%
  \g@addto@macro\emails{#2}%
}
\makeatother

\newtheorem{theorem}{Theorem}[section]
\newtheorem{lemma}[theorem]{Lemma}
\newtheorem{proposition}[theorem]{Proposition}

{ \theoremstyle{definition}
\newtheorem{definition}[theorem]{Definition}}
{ \theoremstyle{remark}
\newtheorem{remark}[theorem]{Remark}}

\newcommand{\ice}{\mathsf{Inter}}
\newcommand{\im}{\mathsf{i}}
\newcommand{\Real}{\mathrm{Re}\hspace{0.5mm}}
\newcommand{\Imag}{\mathrm{Im}\hspace{0.5mm}}

\newcommand{\kgeo}{K^{\mathrm{geo}}}
\newcommand{\kbm}{K^{\mathrm{BM}}}

\newcommand{\Leb}{\operatorname{Leb}}
\newcommand{\SFt}{S_2}
\newcommand{\SFb}{S_1}
\newcommand{\GFt}{G_2}
\newcommand{\GFb}{G_1}

\newcommand{\sigmaq}{\sigma_1}
\newcommand{\fq}{f_1}
\newcommand{\hq}{h_1}
\newcommand{\zc}{z_c}
\newcommand{\pq}{p_1}

\newcommand{\sigmap}{\sigma_2}

\newcommand{\hp}{h_2}
\newcommand{\pp}{p_2}

\title{Curve separation in supercritical half-space last passage percolation}
\author{Evgeni Dimitrov} 
\email{edimitro@usc.edu (corresponding author), zhengyez@usc.edu}
\address{Department of Mathematics, University of Southern California,
Los Angeles, CA 90089, USA}
\author{Zhengye Zhou}

\begin{document}

\begin{abstract} We study line ensembles arising naturally in symmetrized/half-space geometric last passage percolation (LPP) on the $N \times N$ square. The weights of the model are geometrically distributed with parameter $q^2$ off the diagonal and $cq$ on the diagonal, where $q \in (0,1)$ and $c \in [0, q^{-1})$. In the supercritical regime $c > 1$, we show that the ensembles undergo a phase transition: the top curve separates from the rest and converges to a Brownian motion under $N^{1/2}$ fluctuations and $N$ spatial scaling, while the remaining curves converge to the Airy line ensemble under $N^{1/3}$ fluctuations and $N^{2/3}$ spatial scaling. Our analysis relies on a distributional identity between half-space LPP and the Pfaffian Schur process. The latter exhibits two key structures: (1) a Pfaffian point process, which we use to establish finite-dimensional convergence of the ensembles, and (2) a Gibbsian line ensemble, which we use to extend convergence uniformly over compact sets.
\end{abstract}

\maketitle

\tableofcontents

%
%
\section{Introduction and main results}\label{Section1}

%
%
\subsection{Half-space geometric last passage percolation}\label{Section1.1} We begin by introducing {\em symmetrized} or {\em half-space} geometric last passage percolation (LPP). The model depends on two parameters $q \in (0,1)$ and $c \in [0, q^{-1})$. Let $W = (w_{i,j}: i,j \geq 1)$ be an array, where $(w_{i,j}: 1 \leq i \leq j)$ are independent geometric variables with $w_{i,j} \sim \mathrm{Geom}(q^2)$ when $ i \neq j$ and $w_{i,i} \sim \mathrm{Geom}(cq)$, and $w_{i,j} = w_{j,i}$ for all $i,j \geq 1$. Here, we write $X \sim \mathrm{Geom}(\alpha)$ to mean $\mathbb{P}(X = k ) = \alpha^k (1- \alpha)$ for $k \in \mathbb{Z}_{\geq 0}$. 

We visualize the weight $w_{i,j}$ as being associated with the point $(i,j)$ on the lattice $\mathbb{Z}^2$, see the left side of Figure \ref{Fig.Grid}. An {\em up-right path} $\pi$ in $\mathbb{Z}^2$ is a (possibly empty) sequence of vertices $\pi = (v_1, \dots, v_r)$ with $v_i \in \mathbb{Z}^2$, and $v_i - v_{i-1} \in \{(0,1), (1,0)\}$. For an up-right path $\pi$ in $\mathbb{Z}_{\geq 1}^2$, we define its {\em weight} by
\begin{equation}\label{Eq.PathWeight}
W(\pi) = \sum_{v \in \pi} w_v,
\end{equation}
and for any $(m,n) \in \mathbb{Z}_{\geq 1}^2$, we define the {\em last passage time} $G_1(m,n)$ by
\begin{equation}\label{Eq.LPT}
G_1(m,n) = \max_{\pi} W(\pi),
\end{equation} 
where the maximum is over all up-right paths from $(1,1)$ to $(m,n)$. 

\begin{figure}[ht]
\centering
\begin{tikzpicture}[scale=0.9]

\def\m{7} 
\def\n{5} 
\def\k{3} 

\definecolor{Bg}{gray}{1.0}        
\definecolor{C1}{gray}{0.9} 
\definecolor{C2}{gray}{0.6} 
\definecolor{C3}{gray}{0.3} 

\begin{scope}[shift={(0,0)}]


  \foreach \j in {1,...,\m}{
      \node at (\j+0.5, -0.5) {\(\j\)};
  }
  \foreach \j in {1,...,\n}{
      \node at (0.35,-0.5 + \j) {\(\j\)};
  }

  \foreach \x/\y in {1/1, 1/2, 1/3, 2/3, 3/3, 3/4, 3/5, 4/5, 5/5, 6/5, 7/5} {
  \fill[C2] (\x,\y) rectangle ++(1,-1);
  
}

\foreach \i in {1,...,\n}{
  \foreach \j in {1,...,\m}{
    \draw[black] (\j,\i) rectangle ++(1,-1);
    \node at (\j+0.5, \i-0.5) {$w_{\j, \i}$};
  }
}

\node at (\m+1.5,-0.5) {$i$};
\node at (0.35,-0.5 + \n + 1) {$j$};

\end{scope}

\begin{scope}[shift={(9,0)}]

  \foreach \x/\y in {1/1, 2/1, 3/1, 4/1, 5/1, 6/1, 7/1} {
  \fill[C1] (\x/2,\y/2) rectangle ++(1/2,-1/2);
  
}

  \foreach \x/\y in {1/2, 2/2, 3/2, 3/2, 4/2, 5/2, 6/2, 7/2} {
  \fill[C2] (\x/2,\y/2) rectangle ++(1/2,-1/2);
  
}

  \foreach \x/\y in {1/3, 2/3, 3/3, 3/3, 4/3, 5/3, 6/3, 7/3} {
  \fill[C3] (\x/2,\y/2) rectangle ++(1/2,-1/2);
  
}

\foreach \x/\y in {1/1, 2/1, 3/1, 3/2, 3/3, 3/4, 3/5} {
  \fill[C1] (5+ \x/2,\y/2) rectangle ++(1/2,-1/2);
  
}

  \foreach \x/\y in {1/2, 2/2, 2/3, 2/4, 2/5, 2/6, 3/6} {
  \fill[C2] (5+ \x/2,\y/2) rectangle ++(1/2,-1/2);
  
}

  \foreach \x/\y in {1/3, 1/4, 1/5, 1/6, 1/7, 2/7, 3/7} {
  \fill[C3] (5+ \x/2,\y/2) rectangle ++(1/2,-1/2);
  
}

\foreach \i in {1,...,3}{
  \foreach \j in {1,...,7}{
    \draw[black] (\j/2,\i/2) rectangle ++(1/2,-1/2);
    \draw[black] (\i/2 +5, \j/2) rectangle ++(1/2,-1/2);
  }
}

\begin{scope}[shift={(1,-1)}]
  \draw[fill=C1] (1,5.5) rectangle ++(0.5,-0.5);
  \node[anchor=west] at (1.5,5.25) {$\pi_1$};
  \draw[fill=C2] (2.5,5.5) rectangle ++(0.5,-0.5);
  \node[anchor=west] at (3,5.25) {$\pi_2$};
  \draw[fill=C3] (4,5.5) rectangle ++(0.5,-0.5);
  \node[anchor=west] at (4.5, 5.25) {$\pi_3$};
\end{scope}
\end{scope}

\end{tikzpicture}
\caption{The left side depicts the array $W = (w_{i,j}: i,j \geq 1)$ and an up-right path $\pi$ (in gray) that connects $(1,1)$ to $(7,5)$. The right side depicts $\min(m,n)$ pairwise disjoint up-right paths, with $\pi_i$ connecting $(1,i)$ to $(m, n-k+i)$ that cover the whole $n \times m$ rectangle.} \label{Fig.Grid}
\end{figure}
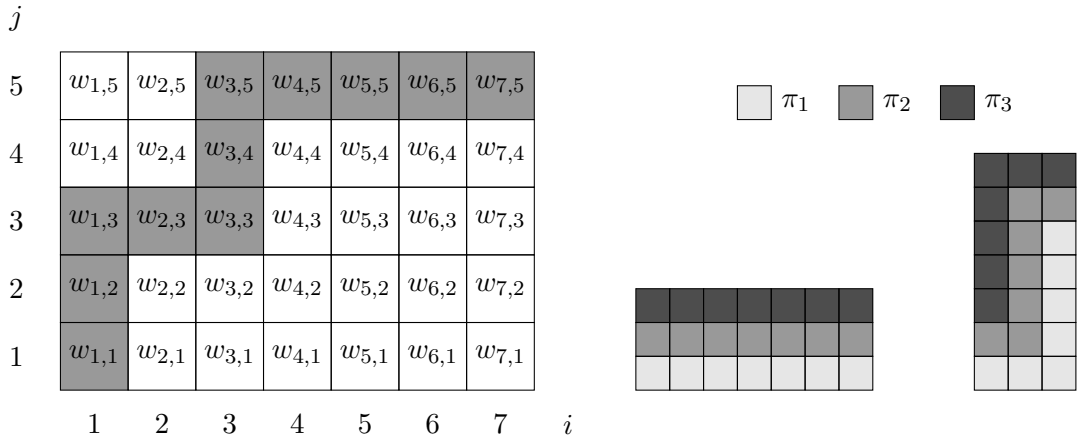

In \cite{BR01c}, building on the previous works \cite{BR01a,BR01b}, Baik and Rains showed that the fluctuations of $G_1(n,n)$ as $n \rightarrow \infty$ are asymptotically described by the GOE Tracy-Widom distribution $F_{\mathrm{GOE}}$ (when $c = 1$), and by the GSE Tracy-Widom distribution $F_{\mathrm{GSE}}$ (when $c \in (0,1)$). We refer the reader to \cite{TW05} for the definitions of $F_{\mathrm{GOE}}$ and $F_{\mathrm{GSE}}$. In addition, \cite{BR01c} showed that if one scales $c = 1 - \varpi \alpha_q n^{-1/3}$ together with $n$, one obtains a one-parameter family of cross-over distributions $F_{\mathrm{cross}}(\cdot; \varpi)$ that interpolate between $F_{\mathrm{GOE}}$ when $\varpi = 0$, $F_{\mathrm{GSE}}$ when $\varpi \rightarrow \infty$, and the normal distribution $\Phi$ when $\varpi \rightarrow -\infty$. 

In \cite{SI04}, Imamura and Sasamoto considered the related {\em polynuclear growth model (PNG)} for the special parameter choices $c = 0$ and $c = 1$. They showed for a fixed $\kappa \in (0,1)$ that the finite-dimensional distributions of $G_1(s,t)$ around the point $(\kappa n, n)$ (under suitable scaling) converge to those of the Airy process. They also found the limits of the finite-dimensional distributions of $G_1(s,t)$ around the point $(n, n)$, and expressed them as certain Fredholm Pfaffians that are different when $c = 0$ and $c = 1$. In \cite{BBNV18}, Betea, Bouttier, Nejjar and Vuleti{\'c}
computed the limits of the finite-dimensional distributions of $G_1(s,t)$ around the point $(n, n)$ when one scales $c = 1 - \varpi \alpha_q n^{-1/3}$ together with $n$, and obtained a one-parameter family of Fredholm Pfaffians (indexed by $\varpi$) -- the ones discovered in \cite{SI04} for $c = 0$ and $c = 1$ correspond to $\varpi = 0$ and $\varpi \rightarrow \infty$, respectively. We mention that the same Fredholm Pfaffians were obtained by Baik, Barraquand, Corwin and Suidan in the context of LPP with exponential weights in \cite{BBCS18}. \\

It turns out that $G_1(m,n)$ can naturally be embedded into a sequence $G(m,n) = (G_k(m,n): k \geq 1)$ of higher-rank last passage times that we describe next. For $k = 1, \dots, \min(m,n)$, we define
\begin{equation}\label{Eq.HRLPT}
G_k(m,n) = \max_{\pi_1, \dots, \pi_k} \left[ W(\pi_1) + \cdots + W(\pi_k) \right],
\end{equation}
where the maximum is over $k$-tuples of pairwise disjoint up-right paths $(\pi_1, \dots, \pi_k)$ with $\pi_i$ connecting the points $(1,i)$ with $(m, n-k + i)$. Note that $G_{\min(m,n)}(m,n) = \sum_{i = 1}^m \sum_{j = 1}^n w_{i,j}$, as one can find $\min(m,n)$ pairwise disjoint paths as above that contain all the vertices in the $n \times m$ rectangle, see the right side of Figure \ref{Fig.Grid}. When $k \geq \min(m,n) + 1$, we cannot find $k$ disjoint paths as above, and so the convention is to set
\begin{equation}\label{Eq.HRLPT2}
G_k(m,n) = \sum_{i = 1}^m \sum_{j = 1}^n w_{i,j} \mbox{ for } k \geq \min(m,n) + 1.
\end{equation}

Our main object of interest is not the $G_{k}(m,n)$'s themselves but rather their successive differences, defined by
\begin{equation}\label{Eq.LPTLambdas}
\lambda_1(m,n) = G_1(m,n) \mbox{ and } \lambda_k(m,n) = G_k(m,n) - G_{k-1}(m,n) \mbox{ for }k\geq 2.
\end{equation}
From \cite[(2.12)]{DY25b}, we know that $\lambda(m,n) = (\lambda_k(m,n): k \geq 1)$ is a {\em partition} (i.e. a decreasing sequence of non-negative integers that is eventually zero). In addition, if we fix $n$ and let $m$ vary from $1$ to $n$, these partitions {\em interlace}, meaning that 
\begin{equation}\label{Eq.InterlaceIntro}
\lambda_1(m,n) \geq \lambda_{1}(m-1,n) \geq \lambda_{2}(m,n) \geq  \lambda_{2}(m-1,n) \geq \cdots, 
\end{equation}
where $\lambda(0,n)$ is the empty partition. By linearly interpolating the points $(m, \lambda_i(m,n))$, we can view $\{\lambda_i(m,n): i \geq 1, 0 \leq m \leq n\}$ as a sequence of random continuous functions, or equivalently as a {\em line ensemble} (as in Definition \ref{CLEDef}), see Figure \ref{Fig.DiscreteLE}.
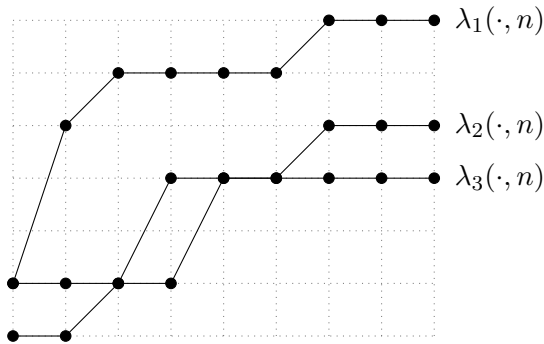
\begin{figure}[ht]
	\begin{center}
		\begin{tikzpicture}[scale=0.7]
		\begin{scope}
        \def\r{0.1}
		\draw[dotted, gray] (0,0) grid (8,6);

        \draw[fill = black] (0,1) circle [radius=\r];
        \draw[fill = black] (1,4) circle [radius=\r];
        \draw[fill = black] (2,5) circle [radius=\r];
        \draw[fill = black] (3,5) circle [radius=\r];
        \draw[fill = black] (4,5) circle [radius=\r];
        \draw[fill = black] (5,5) circle [radius=\r];
        \draw[fill = black] (6,6) circle [radius=\r];
        \draw[fill = black] (7,6) circle [radius=\r];
        \draw[fill = black] (8,6) circle [radius=\r];
        \draw[-][black] (0,1) -- (1,4);
        \draw[-][black] (1,4) -- (2,5);
        \draw[-][black] (2,5) -- (3,5);
        \draw[-][black] (3,5) -- (4,5);
        \draw[-][black] (4,5) -- (5,5);
        \draw[-][black] (5,5) -- (6,6);
        \draw[-][black] (6,6) -- (7,6);
        \draw[-][black] (7,6) -- (8,6);

        \draw[fill = black] (0,1) circle [radius=\r];
        \draw[fill = black] (1,1) circle [radius=\r];
        \draw[fill = black] (2,1) circle [radius=\r];
        \draw[fill = black] (3,3) circle [radius=\r];
        \draw[fill = black] (4,3) circle [radius=\r];
        \draw[fill = black] (5,3) circle [radius=\r];
        \draw[fill = black] (6,4) circle [radius=\r];
        \draw[fill = black] (7,4) circle [radius=\r];
        \draw[fill = black] (8,4) circle [radius=\r];
        \draw[-][black] (0,1) -- (1,1);
        \draw[-][black] (1,1) -- (2,1);
        \draw[-][black] (2,1) -- (3,3);
        \draw[-][black] (3,3) -- (4,3);
        \draw[-][black] (4,3) -- (5,3);
        \draw[-][black] (5,3) -- (6,4);
        \draw[-][black] (6,4) -- (7,4);
        \draw[-][black] (7,4) -- (8,4);

        \draw[fill = black] (0,0) circle [radius=\r];
        \draw[fill = black] (1,0) circle [radius=\r];
        \draw[fill = black] (2,1) circle [radius=\r];
        \draw[fill = black] (3,1) circle [radius=\r];
        \draw[fill = black] (4,3) circle [radius=\r];
        \draw[fill = black] (5,3) circle [radius=\r];
        \draw[fill = black] (6,3) circle [radius=\r];
        \draw[fill = black] (7,3) circle [radius=\r];
        \draw[fill = black] (8,3) circle [radius=\r];
        \draw[-][black] (0,0) -- (1,0);
        \draw[-][black] (1,0) -- (2,1);
        \draw[-][black] (2,1) -- (3,1);
        \draw[-][black] (3,1) -- (4,3);
        \draw[-][black] (4,3) -- (5,3);
        \draw[-][black] (5,3) -- (6,3);
        \draw[-][black] (6,3) -- (7,3);
        \draw[-][black] (7,3) -- (8,3);

        \draw (9.25, 6) node{$\lambda_1(\cdot,n)$};
        \draw (9.25, 4) node{$\lambda_2(\cdot,n)$};
        \draw (9.25, 3) node{$\lambda_3(\cdot,n)$};

		\end{scope}

		\end{tikzpicture}
	\end{center}
	\caption{The figure depicts the top three curves in $\{\lambda_i(m,n): i \geq 1, 0 \leq m \leq n\}$.}
	\label{Fig.DiscreteLE}
\end{figure}

In the joint work with Yang \cite{DY25}, the first author constructed a one-parameter family (indexed by $\varpi$) of {\em half-space Airy line ensembles} as weak limits of the line ensembles $\{\lambda_i(m,n): i \geq 1, m \geq n\}$ as above with $c = 1 - \varpi \alpha_q n^{-1/3}$. The latter are half-space analogues of the Airy line ensemble constructed in \cite{CorHamA}, see also Definition \ref{Def.AiryLE}. In \cite{Z25}, the second author showed that when $c \in (0,1]$ and $\kappa \in (0,1)$ the line ensembles $\{\lambda_{i}(t,n)\}_{i \geq 1}$ for $t$ near $\kappa n$ converge (after a suitable shift and scaling) to the Airy line ensemble. This result continues to hold when $c > 1$, as long as $\kappa \in (0, \kappa_0)$, where $\kappa_0 = \left(\frac{1 - qc}{c - q}\right)^2$. In addition, \cite{Z25} showed that when $c = z_c - \varpi \alpha_q n^{-1/3}$ (for a cerain constant $z_c$), the line ensembles $\{\lambda_{i}(t,n)\}_{i \geq 1}$ for $t$ near $\kappa_0 n$ instead converge (after a suitable shift and scaling) to an Airy wanderer line ensemble from \cite{AFM10,CorHamA,ED24a}. \\

The discussion in the previous paragraph suggests that when $c > 1$, which we call the {\em supercritical regime}, the line ensembles $\{\lambda_{i}(\kappa n ,n)\}_{i \geq 1}$ have one type of behavior for $ \kappa \in (0,\kappa_0)$, which agrees with the {\em subcritical regime} $c \in (0,1)$, and a different one on $(\kappa_0, 1)$. The goal of this paper is to obtain a detailed description of the behavior of the line ensembles on $(\kappa_0, 1)$. As the results in the next section show, what happens on $(\kappa_0, 1)$ is that the top curve $\lambda_1(t,n)$ separates from the rest and behaves like a Brownian motion when appropriately shifted, and scaled horizontally by $n$ and vertically by $n^{1/2}$. On the other hand, the remaining curves $\{\lambda_{i}(t,n)\}_{i \geq 2}$ converge for $t$ near $\kappa n$ for each $\kappa \in (\kappa_0,1)$, under an appropriate shift, and $n^{2/3}$ horizontal and $n^{1/3}$ vertical scaling, to the Airy line ensemble. In Figure \ref{Fig.Simulations} we show simulations of our line ensembles in the subcritical and supercritical regimes. 
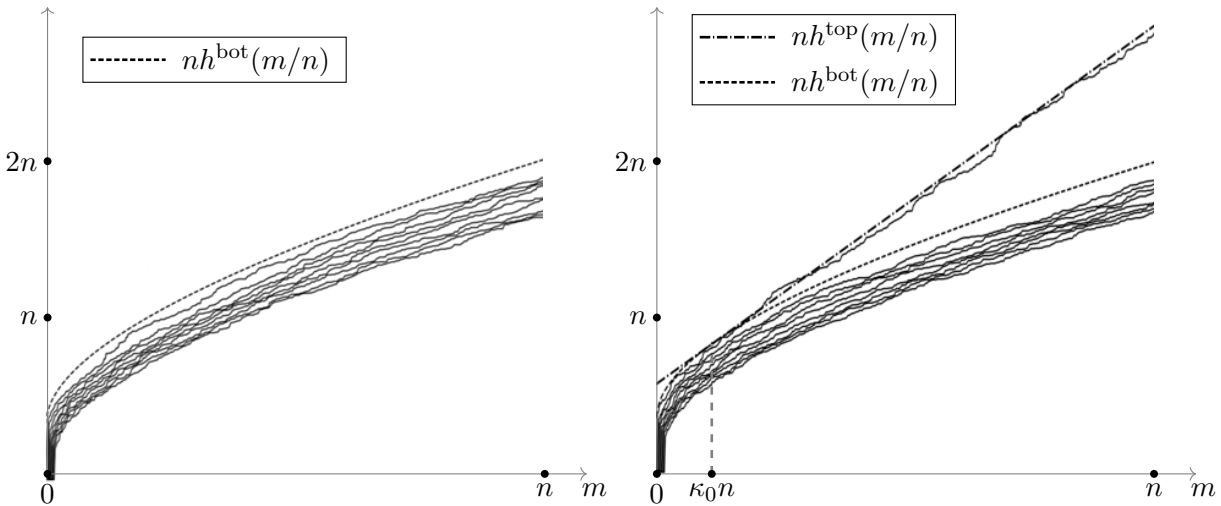
\begin{figure}[ht]
	\begin{center}
      \begin{tikzpicture}
        \node[anchor=south west, inner sep=0] (img1) at (0,0) 
            {\includegraphics[width=0.4\textwidth]{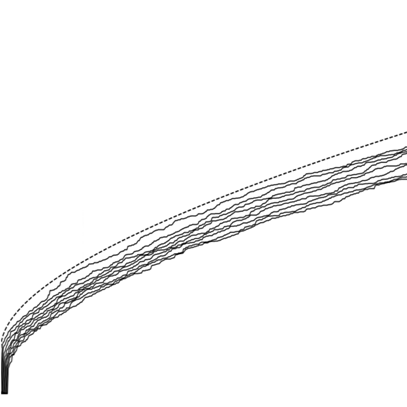}};
    
        \node[anchor=south west, inner sep=0] (img2) at (8.12,0) 
            {\includegraphics[width=0.4\textwidth]{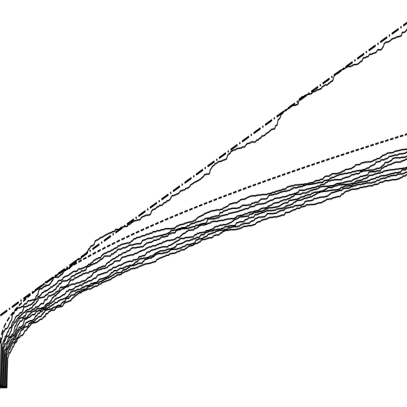}};


        \draw[->][gray] (0.03,0.29) -- (7.2,0.29);
        \draw[->][gray] (0.03,0.29) -- (0.03,6.5);
        \draw[black] (0.5,5.5) rectangle ++(3.5,0.6);
        \draw (7.3,0.29) node[below = 0pt]{$m$};
        \draw (2.8, 5.8) node{$nh^{\mathrm{bot}}(m/n)$};
        \draw[-][line width = 1pt, dash pattern=on 2pt off 1pt][black] (0.6,5.8) -- (1.6,5.8);
        \fill (6.64, 0.29) circle (1.5pt) node[below=0pt] {$n$};
        \fill (0.03, 0.29) circle (1.5pt) node[below=0pt] {$0$};

        \fill (0.03, 2.37) circle (1.5pt) node[left=0pt] {$n$};
        \fill (0.03, 4.45) circle (1.5pt) node[left=0pt] {$2n$};


        \begin{scope}[shift={(8.1,0)}]

        \draw[->][gray] (0.03,0.29) -- (7.2,0.29);
        \draw[->][gray] (0.03,0.29) -- (0.03,6.5);
        \draw (7.3,0.29) node[below = 0pt]{$m$};

        \draw[black] (0.5,5.2) rectangle ++(3.5,1.2);
        \draw (2.8, 6.1) node{$nh^{\mathrm{top}}(m/n)$};
        \draw (2.8, 5.5) node{$nh^{\mathrm{bot}}(m/n)$};

        \draw[line width = 1pt, dash pattern=on 4pt off 1pt on 0.5pt off 1pt][black] (0.6,6.1) -- (1.6,6.1);
        \draw[-][line width = 1pt, dash pattern=on 2pt off 1pt][black] (0.6,5.5) -- (1.6,5.5);
        
        \fill (6.64, 0.29) circle (1.5pt) node[below=0pt] {$n$};
        \fill (0.03, 0.29) circle (1.5pt) node[below=0pt] {$0$};

        \draw[-][line width = 1 pt, dashed][gray] (0.76,0.29) -- (0.76,2);
        \fill (0.76, 0.29) circle (1.5pt) node[below=0pt] {$\kappa_0 n$};

        \fill (0.03, 2.37) circle (1.5pt) node[left=0pt] {$n$};
        \fill (0.03, 4.45) circle (1.5pt) node[left=0pt] {$2n$};

        \end{scope}

\end{tikzpicture}

	\end{center}
	\caption{The figure depicts the top several curves of $\{\lambda_i(m,n): i \geq 1, 0 \leq m \leq n\}$ for $q = 0.5$, $n = 500$, and $c = 0.8$ (left) and $c = 1.4$ (right). In the subcritical regime (left side), the curves concentrate around the function $n h^{\mathrm{bot}}(m/n)$ on $[0,n]$, where $h^{\mathrm{bot}}(\kappa) = \frac{ q (q + 2 \sqrt{\kappa} + q \kappa)}{1-q^2}$. In the supercritical regime (right side), they concentrate around the same curve on $[0, \kappa_0 n]$. On the interval $[\kappa_0 n, n]$ the top curve separates from the rest and follows the straight line $n h^{\mathrm{top}}(m/n)$, where $h^{\mathrm{top}}(\kappa) = \frac{q}{c-q} + \frac{qc \kappa}{1-qc}$, while the remaining curves continue to follow $n h^{\mathrm{bot}}(m/n)$.}
	\label{Fig.Simulations}
\end{figure}

The convergence of the top curve to a Brownian motion is in agreement with the expected Gaussian asymptotic behavior for $\lambda_1(n,n)$, see \cite[Remark 4.3]{BBNV18} and \cite[Theorem 1.3]{BBCS18} for related results in the context of exponential LPP. In addition, top curve separation has recently been established for the half-space log-gamma polymer in \cite{DZ24}, which may be viewed as a positive-temperature analogue of the LPP model studied here. The main contribution of the present paper is to provide an essentially complete description of the mechanism of curve separation, with explicit functional limit theorems not only for the top curve, but also for the remaining ones. The convergence to the Airy line ensemble admits the following heuristic explanation: in the supercritical phase, the top curve drifts to positive infinity, after which the second curve assumes the role of the new top curve, the third becomes the new second, and so on. We expect this phenomenon to persist across a broad class of half-space models in the KPZ universality class, with prototypical examples including exponential LPP and the log-gamma polymer.

A key advantage of the present work is the integrable structure of the model: it admits a Pfaffian point process representation with an explicit correlation kernel given by a double contour integral. This structure provides access to exact formulas which, although still requiring substantial analysis, enable us to obtain a detailed description of the curve separation phenomenon. We anticipate that our techniques can be extended to other integrable half-space models, such as exponential, Bernoulli, Poisson, or Brownian LPP. Although positive-temperature models like the log-gamma polymer are not known to possess a Pfaffian point process structure, we expect that some of the methods developed here --- particularly those related to line ensembles --- can be adapted to that setting. More broadly, the development of techniques capable of rigorously establishing this type of curve separation for general directed polymer models would benefit from the clear description of the underlying physical mechanism that emerges in the present integrable framework.

%
%
\subsection{Main results}\label{Section1.2} In this section we present the main results of the paper. We introduce the scaling of our line ensembles that captures the top curve's behavior in the following definition.
\begin{definition}\label{Def.TopCurveScaledLPP}
Fix $q \in (0,1)$, $c \in (1, q^{-1})$, and set
\begin{equation}\label{Eq.ConstTopIntro}
\kappa_0 = \left( \frac{1 - qc}{c-q}\right)^2 \in (0,1), \hspace{2mm} p^{\mathrm{top}} = \frac{qc}{1- qc}.
\end{equation}
Fix $N \in \mathbb{N}$, and let $\lambda(m,N)$ be as in (\ref{Eq.LPTLambdas}). For $i \geq 1$, define 
\begin{equation}\label{Eq.ExtPaths}
U^{N}_i(s) = \begin{cases} \lambda_i(s,N)  &\mbox{ if } s= 1,\dots, N, \\  0 &\mbox{ if  } s \leq 0, \\ \lambda_i(N,N) &\mbox{ if } s \geq N+1. \end{cases}
\end{equation}
By linearly interpolating $(s,U^{N}_i(s))$, we may view $U_i^{N}$ as random continuous functions on $\mathbb{R}$. Set $C^{\mathrm{top}}_N = \lfloor (c-q)^{-1} qN \rfloor$, and define the rescaled process $\mathcal{U}_1^{\mathrm{top},N}$ on $(\kappa_0, 1)$ by
$$\mathcal{U}_1^{\mathrm{top},N}(t) = [p^{\mathrm{top}} (1 + p^{\mathrm{top}})]^{-1/2} N^{-1/2} \left( U_1^{N}(tN) - C^{\mathrm{top}}_N -  p^{\mathrm{top}} t N   \right) \mbox{ for } t \in (\kappa_0, 1).$$
\end{definition}

The following is the main result we prove for the top curve of our ensembles. Its proof is given in Section \ref{Section7.3}.
\begin{theorem}\label{Thm.Main1} Assume the same notation as in Definition \ref{Def.TopCurveScaledLPP}, and let $(B_t: t \geq 0)$ be a standard Brownian motion. Then, $\mathcal{U}_1^{\mathrm{top},N} \Rightarrow W$ in $C\left((\kappa_0, 1)\right)$ (with the topology of uniform convergence over compact sets), where $W_t = B_{t - \kappa_0}$ for $t \in (\kappa_0, 1)$.
\end{theorem}

To state the second main result of the paper we need to introduce the Airy line ensemble, which we accomplish in the following definition.
\begin{definition}\label{Def.AiryLE} The {\em Airy line ensemble} $\mathcal{A} = \{\mathcal{A}_i\}_{i \geq 1}$ is a sequence of real-valued random continuous functions, defined on $\mathbb{R}$, and are strictly ordered in the sense that $\mathcal{A}_i(t) > \mathcal{A}_{i+1}(t)$ for all $i \geq 1$ and $t \in \mathbb{R}$. It is uniquely specified by the following conditions. If one fixes a finite set $\mathsf{S} = \{s_1, \dots, s_m\} \subset \mathbb{R}$ with $s_1 < \cdots < s_m$, then the random measure on $\mathbb{R}^2$, defined by
\begin{equation}\label{Eq.RMS1}
M(A) = \sum_{i \geq 1} \sum_{j = 1}^m {\bf 1} \left\{\left(s_j, \mathcal{A}_{i}(s_j) \right) \in A \right\},
\end{equation}
is a determinantal point process on $\mathbb{R}^2$ with reference measure $\mu_{\mathsf{S}} \times \mathrm{Leb}$, where $\mu_{\mathsf{S}}$ is the counting measure on $\mathsf{S}$, and $\mathrm{Leb}$ is the usual Lebesgue measure on $\mathbb{R}$, and whose correlation kernel is given by the {\em extended Airy kernel}, defined for $x_1, x_2 \in \mathbb{R}$ and $t_1, t_2 \in \mathsf{S}$ by 
\begin{equation}\label{Eq.S1AiryKer}
\begin{split}
K^{\mathrm{Airy}}(t_1,x_1; t_2,x_2) = & -  \frac{{\bf 1}\{ t_2 > t_1\} }{\sqrt{4\pi (t_2 - t_1)}} \cdot e^{ - \frac{(x_2 - x_1)^2}{4(t_2 - t_1)} - \frac{(t_2 - t_1)(x_2 + x_1)}{2} + \frac{(t_2 - t_1)^3}{12} } \\
& + \frac{1}{(2\pi \im)^2} \int_{\mathcal{C}_{\alpha}^{\pi/3}} d z \int_{\mathcal{C}_{\beta}^{2\pi/3}} dw \frac{e^{z^3/3 -x_1z - w^3/3 + x_2w}}{z + t_1 - w - t_2}.
\end{split}
\end{equation}
In (\ref{Eq.S1AiryKer}) we have that $\alpha, \beta \in \mathbb{R}$ are arbitrary subject to $\alpha + t_1 > \beta + t_2$, and $\mathcal{C}_{z}^{\varphi}=\{z+|s|e^{\mathrm{sgn}(s)\im\varphi}, s\in \mathbb{R}\}$ is the infinite curve oriented from $z+\infty e^{-\im\varphi}$ to $z+\infty e^{\im\varphi}$. 
\end{definition}
\begin{remark}\label{Rem.AiryLE1}
There are various formulas for the extended Airy kernel, and the one in (\ref{Eq.S1AiryKer}) comes from \cite[Proposition 4.7 and (11)]{BK08} under the change of variables $u \rightarrow z + t_1$ and $w \rightarrow w + t_2$. 
\end{remark}
\begin{remark}\label{Rem.AiryLE2} The introduction of the extended Airy kernel is frequently attributed to \cite{Spohn} where it arises in the context of the polynuclear growth model, although it appeared earlier in \cite{Mac94} and \cite{FNH99}. A continuous version of the Airy line ensemble was formally constructed in \cite{CorHamA} by taking a weak limit of Brownian watermelons. 
\end{remark}
\begin{remark}\label{Rem.AiryLE3} We mention that the conditions in Definition \ref{Def.AiryLE} uniquely specify the law of $\mathcal{A}$ in view of \cite[Proposition 2.13(3)]{ED24a}, \cite[Corollary 2.20]{ED24a} and \cite[Lemma 3.1]{DM21}.
\end{remark}

We next introduce the scaling of our line ensembles that captures the bottom curves' behavior in the following definition.
\begin{definition}\label{Def.BotCurveScaledLPP}
Fix $q \in (0,1)$, $c \in (1, q^{-1})$, and $\kappa \in (\kappa_0, 1)$, where $\kappa_0$ is as in (\ref{Eq.ConstTopIntro}), and set
\begin{equation}\label{Eq.ConstBotIntro}
\begin{split}
&h^{\mathrm{bot}}(\kappa)  = \frac{ q (q + 2 \sqrt{\kappa} + q \kappa)}{1-q^2}, \hspace{2mm} p^{\mathrm{bot}} = \frac{d}{d\kappa} h^{\mathrm{bot}}(\kappa) = \frac{q (1 + q \sqrt{\kappa})  }{(1-q^2) \sqrt{\kappa} }, \hspace{2mm} \\
& \fq = \frac{q^{1/3}}{2 \kappa^{2/3}(1-q^2)^{2/3} (q + \sqrt{\kappa})^{1/3} (1 + q \sqrt{\kappa})^{1/3}}.
\end{split}
\end{equation}
Set $C^{\mathrm{bot}}_N = \lfloor h^{\mathrm{bot}}(\kappa) N \rfloor$, and with $U^{N}_i$ as in (\ref{Eq.ExtPaths}) define the rescaled processes $\{\mathcal{U}_i^{\mathrm{bot},N} \}_{i \geq 1}$ on $\mathbb{R}$ by
$$\mathcal{U}_i^{\mathrm{bot},N}(t) = [p^{\mathrm{bot}} (1 + p^{\mathrm{bot}})]^{-1/2} N^{-1/3} \left( U_{i+1}^{N}(\lfloor \kappa N \rfloor + t N^{2/3}) - C^{\mathrm{bot}}_N -  p^{\mathrm{bot}} t N^{2/3}   \right) \mbox{ for } t \in \mathbb{R}.$$
We denote by $\mathcal{U}^{\mathrm{bot},N} = \{\mathcal{U}^{\mathrm{bot},N}_i\}_{i \geq 1}$ the corresponding line ensemble as in Definition \ref{CLEDef}.
\end{definition}

The following is the main result we prove for the bottom curves of our ensembles. Its proof is given in Section \ref{Section8.2}.
\begin{theorem}\label{Thm.Main2} Assume the same notation as in Definition \ref{Def.BotCurveScaledLPP}. Then, $\mathcal{U}^{\mathrm{bot},N}  \Rightarrow \mathcal{U}^{\mathrm{bot},\infty}$ in $C\left(\mathbb{N} \times \mathbb{R} \right)$ (as in Definition \ref{CLEDef}). Here, $\mathcal{U}^{\mathrm{bot},\infty} = \{\mathcal{U}^{\mathrm{bot},\infty}_i\}_{i \geq 1}$ is the line ensemble, defined by
$$\mathcal{U}_i^{\mathrm{bot},\infty}(t) =  (2\fq)^{-1/2} \cdot \left( \mathcal{A}_i(\fq t) - \fq^2 t^2 \right) \mbox{ for } i \geq 1, t \in \mathbb{R},$$
where $\mathcal{A} = \{\mathcal{A}_i\}_{i \geq 1}$ is the Airy line ensemble from Definition \ref{Def.AiryLE}.
\end{theorem}

Let us give a (necessarily somewhat reductive) physical interpretation of the two theorems above. When $c \in (1, q^{-1})$, the weights along the main diagonal $\{w_{i,i}\}_{i \geq 1}$ are sufficiently large that they dominate the variational problem defining $G_{1}(m,n)$. Although we do not investigate the maximizing path geometry in this paper, we expect that the maximizing path makes an essentially linear beeline along the ``rich'' diagonal before departing to its final location. It spends enough time near the diagonal to experience a central limit theorem that ultimately leads to its Gaussian fluctuations.

Once this leading path has effectively harvested the diagonal's excess weight, it creates a form of energetic barrier for the remaining paths. These lower curves are forced slightly away from the diagonal and no longer benefit from its enhanced weights. In this displaced environment, they evolve in a regime that is largely unaffected by the diagonal's dominance and therefore develop Airy line ensemble statistics. In this sense, the top curve extracts the diagonal advantage, while the remaining curves revert to the universal KPZ fluctuation structure characteristic of the subleading ensemble. It would be very interesting to shed more light on the geometry of maximizing paths and to determine how much of the above heuristic description can be rigorously verified. We hope to address this and other questions in future work.

%
%
\subsection{Main ideas and paper outline}\label{Section1.3} The starting point of our analysis is a distributional equality between half-space LPP and the Pfaffian Schur processes from \cite{BR05}. The latter are recalled in Section \ref{Section2.1} and their relationship to LPP is detailed in Proposition \ref{Prop.LPPandSchur}. One key advantage of the Pfaffian Schur process is that it is a Pfaffian point process with an explicit correlation kernel, expressible as a double contour integral. The general formula for the kernel is presented in Proposition \ref{Prop.CorrKernel1}. In addition, we derive two alternative formulas for the correlation kernel in Lemmas \ref{Lem.PrelimitKernelBulk} and \ref{Lem.PrelimitKernelEdge}, which are suitable for asymptotic analysis in the two scaling regimes for the bottom curves and the top curves of our ensembles, respectively.

In Sections \ref{Section4} and \ref{Section5} we establish the uniform over compact sets convergence of our correlation kernels from Lemmas \ref{Lem.PrelimitKernelBulk} and \ref{Lem.PrelimitKernelEdge}. The precise statements are given in Propositions \ref{Prop.KernelConvBottom} and \ref{Prop.KernelConvTop}, and proved using the method of steepest descent. In order to carry out our arguments in these sections, we require suitable descent contours, and estimates for various functions along them, which is the content of Section \ref{Section3}. One of the challenges here was to find contours that work for the full range of parameters $q,c,\kappa$, and also for both of the scalings under consideration. 

The second key feature of the Pfaffian Schur process is that it has the structure of a {\em geometric line ensemble} that satisfies the {\em interlacing Gibbs property}. These terms are introduced in Section \ref{Section6.1} and the line ensemble structure of the Pfaffian Schur processes is detailed in Section \ref{Section6.3}. The main advantage of the interlacing Gibbs property for the present paper is that it effectively reduces the proofs of Theorems \ref{Thm.Main1} and \ref{Thm.Main2} to establishing finite-dimensional convergence, using the framework introduced in \cite{dimitrov2024tightness}. One of the reasons Theorem \ref{Thm.Main1} is established for $(\kappa_0,1)$ and not $[\kappa_0,1]$ is to make the results from \cite{dimitrov2024tightness} directly applicable. With some (modest) effort one could extend the convergence to the endpoints, but we do not pursue this here. 

The usual approach of showing the finite-dimensional convergence of the top curves of our line ensembles goes through computing the limit of a certain Fredholm Pfaffian built from our correlation kernel. The latter is given by an infinite series where each summand converges due to the kernel convergence from Proposition \ref{Prop.KernelConvTop}. The fact that the series converges is then established using a combination of upper-tail kernel estimates and Hadamard's inequality to control the summands, and the dominated convergence theorem to control the full series.

In this paper, however, we take a different approach. The kernel convergence from Proposition \ref{Prop.KernelConvTop} implies vague convergence of point processes, and in Lemma \ref{Lem.FDC} we show that the latter can be improved to finite-dimensional convergence in the presence of one-point tightness. In addition, Lemma \ref{Lem.tightcri} shows that one point tightness essentially follows from tightness from above when vague convergence of point processes is present. To prove tightness from above for our ensembles we simply estimate the expected values of the first moments of the corresponding point processes. At a high level, our approach requires that we only work with the first term (as opposed to the full series) of the Fredholm Pfaffian, thereby reducing technical computations and estimates significantly. The same approach is also employed for the bottom curves, and the advantages there are even more considerable, as for multiple curves one generally needs to deal with the partial derivatives of certain multivariate Fredholm-Pfaffian power series. We mention that our framework should be quite general, and applicable to a wide range of determinantal and Pfaffian point processes.

The finite-dimensional convergence of the top curve is established in Proposition \ref{Prop.FinitedimEdge}, and from it Theorem \ref{Thm.Main1} is quickly deduced in Section \ref{Section7.3} using the interlacing Gibbs property and the tightness criterion from \cite{dimitrov2024tightness}. The finite-dimensional convergence for the bottom curves is established in Proposition \ref{Prop.FinitedimBulk}; however, unlike the top curve case we cannot readily apply the results from \cite{dimitrov2024tightness}. The obstruction is that the interlacing Gibbs property is lost when we project our ensembles to curves of index $i \geq 2$. 

To overcome this, we introduce auxiliary ensembles that do satisfy the Gibbs property, and which are close to the ones we actually want to study. The precise definition of these ensembles and what we mean by ``close'' is detailed in Section \ref{Section8.3}. Roughly speaking these ensembles are built from the original ones by pushing the top curve to $+\infty$, which does not affect the curve distribution too much as the top curve is already macroscopically separated from the others due to Theorem \ref{Thm.Main1}. In Section \ref{Section8.4} we apply the tightness criterion from \cite{dimitrov2024tightness} to our auxiliary ensembles, and deduce tightness of the bottom curves. The latter and the finite-dimensional convergence from Proposition \ref{Prop.FinitedimBulk} then complete the proof of Proposition \ref{Prop.ConvBottom}, which is essentially a restatement of Theorem \ref{Thm.Main2} in terms of Pfaffian Schur processes.

%
%
\subsection{Related works and future directions}\label{Section1.4} 
Most previous works on symmetrized LPP have focused on the behavior of the model near the main diagonal, where the influence of the boundary parameter $c$ is most pronounced. While the asymptotic Gaussianity of $\lambda_1(n,n)$ has been investigated in \cite[Remark 4.3]{BBNV18} and \cite[Theorem 1.3]{BBCS18}, the full Brownian motion limit established in Theorem \ref{Thm.Main1} appears, to the best of our knowledge, to be new. There are, however, related results for positive-temperature analogues of our model, which we now discuss.

In the seminal work \cite{BCD25}, Barraquand, Corwin, and Das initiated a systematic study of general half-space Gibbsian line ensembles, a class that includes our LPP model as a special case. Although \cite{BCD25} focuses primarily on the half-space log-gamma polymer, it introduced a number of key ideas that have since inspired further developments in the study of directed half-space polymers and line ensembles; see \cite{DS25,DZ24,Gins} and the references therein. In particular, \cite{DZ24} investigates the half-space log-gamma polymer in the analogous supercritical regime, referred to there as the {\em bound phase}. The associated half-space log-gamma line ensemble exhibits a curve separation phenomenon between the first and second curves, in the same spirit as our results. 

The analysis in \cite{DZ24} relies almost exclusively on combinatorial arguments and properties of Gibbsian line ensembles, reflecting the fact that the log-gamma model does not enjoy the exact solvability available in our setting. As a consequence, their methods do not yield the functional limit theorems of the type proved here. While the authors obtain a local description of the top curve as a random walk --- an input that could, in principle, be used to establish convergence as in Theorem \ref{Thm.Main1} --- they do not address the distributional limits of the second, third, and subsequent curves. One possible strategy for proving an analogue of Theorem \ref{Thm.Main2} for the log-gamma line ensemble would be to show that the ensemble formed by the second, third, and subsequent curves is tight, globally parabolic, and satisfies the Brownian Gibbs property from \cite{CorHamA}. Establishing tightness of this ensemble would likely require one-point tightness of the second curve, potentially achievable using the arguments of \cite{DS25}, and then applying the tightness criterion in \cite{DW21}. Once tightness is obtained, one could invoke the strong characterization of the Airy line ensemble from \cite{AH23} to complete the argument.

Taken together with \cite{Z25}, the present work provides an essentially complete description of the line ensembles $\{\lambda_k(m,n)\}_{k \geq 1}$ away from the main diagonal $m = n$. What remains is to understand their behavior near the diagonal, where the limiting picture depends sensitively on the value of the boundary parameter $c$: whether $c \in (0,1)$, $c$ is critically scaled near $1$, or $c \in (1, q^{-1})$. So far, only the critical case has been analyzed. In \cite{DY25}, it was shown that under critical scaling, these ensembles converge to a one-parameter family of (critical) half-space Airy line ensembles. The subcritical regime $c \in (0,1)$ and supercritical regime $c \in (1, q^{-1})$ are currently under investigation in \cite{DDY26}. The prevailing expectation is as follows. When $c \in (0,1)$, the ensembles $\{\lambda_k(m,n)\}_{k \geq 1}$, when observed near the diagonal, converge to the pinned half-space Airy line ensemble that was recently constructed in \cite{DSY26}. When $c \in (1, q^{-1})$, the top curve separates from the rest as demonstrated in the present paper, while the remaining curves again converge to the pinned half-space Airy line ensemble, which by now is expected to represent the universal scaling limit for non-critical KPZ half-space models.

%
%
\section{Pfaffian Schur processes}\label{Section2} In this section we introduce a family of {\em Pfaffian Schur processes} and discuss some of their properties. Our models are special cases of the ones introduced in \cite{BR05}, which are themselves Pfaffian analogs of the determinantal Schur processes in \cite{OR03}. We refer interested readers to \cite{BR05, BBNV18, BBCS18} for more background on Pfaffian Schur processes.

%
%
\subsection{Definitions}\label{Section2.1} A {\em partition} is a non-increasing sequence of non-negative integers $\lambda = (\lambda_1 \geq \lambda_2 \geq \cdots)$ that is eventually zero. The {\em weight} of a partition $\lambda$ is given by $|\lambda| = \sum_{i \geq 1} \lambda_i$. There is a single partition of weight $0$, which we denote by $\emptyset$. Given two partitions $\lambda$ and $\mu$, we say that they {\em interlace}, denoted by $\lambda \succeq \mu$ or $\mu \preceq \lambda$, if $\lambda_1 \geq \mu_1 \geq \lambda_2 \geq \mu_2 \geq \cdots$.

For two partitions $\lambda, \mu$ we define the {\em skew Schur polynomial} in a single variable $x$ by
\begin{equation}\label{Eq.SkewSchur}
s_{\lambda/\mu}(x) = {\bf 1}\{ \mu \preceq \lambda\} \cdot x^{|\lambda| - |\mu|}.
\end{equation}
For a partition $\lambda$, we also define the {\em boundary monomial} in a single variable $c$ by
\begin{equation}\label{Eq.Tau}
\tau_{\lambda}(c) = c^{\sum_{j = 1}^{\infty} (-1)^{j-1}\lambda_j} = c^{\lambda_1 - \lambda_2 + \lambda_3 - \lambda_4 + \cdots}.
\end{equation}

With the above notation in place we can define our main object of interest.
\begin{definition}\label{Def.SchurProcess} Fix $N \in \mathbb{N}$ and $c, a_1, \dots, a_N \in [0,\infty)$, such that $ca_i, a_j a_k < 1$ for $i,j,k \in \{1, \dots, N\}$ with $j \neq k$. With these parameters we define the {\em Pfaffian Schur process} to be the probability distribution on sequences of partitions $\lambda^1, \lambda^2, \dots, \lambda^N$, given by
\begin{equation}\label{Eq.SchurProcess}
\mathbb{P}(\lambda^1, \dots, \lambda^N) = \frac{1}{Z} \cdot \tau_{\lambda^1}(c) s_{\lambda^1/\lambda^2}(a_{N}) s_{\lambda^2/\lambda^3}(a_{N-1}) \cdots s_{\lambda^{N}/\lambda^{N+1}}(a_1),
\end{equation}
where $\lambda^{N+1} = \emptyset$, and $Z$ is a normalization constant, computed explicitly in \cite[Proposition 3.2]{BR05}:
$$Z = \prod_{i = 1}^N \frac{1}{1 - ca_i} \times \prod_{1 \leq i < j \leq N} \frac{1}{1 - a_i a_j}.$$
In this paper we almost exclusively work with the above measures when 
\begin{equation}\label{Eq.HomogeneousParameters}
a_1 = a_2 = \cdots = a_N = q \in (0,1), \mbox{ and } c \in (0, q^{-1}).
\end{equation}
\end{definition}
\begin{remark}\label{Rem.BRSpecial} The measures in Definition \ref{Def.SchurProcess} are special cases of the Pfaffian Schur processes in \cite[Section 3]{BR05}, and correspond to setting $\rho_0^+ = c$, $\rho_i^+ = 0$ for $i = 1, \dots, N-1$, and $\rho_i^- = a_{N-i+1}$ for $i = 1, \dots, N$. 
\end{remark}

The following statement explains how the Schur processes are related to the half-space LPP model from Section \ref{Section1.1}. It is a special case of \cite[Theorem 2.7]{DY25b}. 
\begin{proposition}\label{Prop.LPPandSchur} Let $(\lambda^1, \dots,\lambda^N)$ be distributed as in (\ref{Eq.SchurProcess}) with parameters as in (\ref{Eq.HomogeneousParameters}). Then, $(\lambda^1, \dots, \lambda^N)$ has the same distribution as $(\lambda(N,N), \lambda(N-1,N), \dots, \lambda(1,N))$ from (\ref{Eq.LPTLambdas}).
\end{proposition}
\begin{proof} From \cite[(2.10)]{DY25b} we have that almost surely $(\lambda(N,N), \lambda(N-1,N), \dots, \lambda(1,N)) = (\lambda(N,N), \lambda(N,N-1), \dots, \lambda(N,1))$. The statement now follows from \cite[Theorem 2.7]{DY25b} applied to $M = 0$ and the down-right path $\gamma = (v_0, \dots, v_N)$ with $v_i = (N, N-i)$.
\end{proof}

%
%
\subsection{Pfaffian point process structure}\label{Section2.2} In this section we explain how one can interpret the Pfaffian Schur processes as a {\em Pfaffian point process} (this is the origin of their name). In what follows we freely use the definitions and notation regarding Pfaffian point processes from \cite[Section 5.2]{DY25}. We also refer the interested reader to \cite[Appendix B]{OQR17} and \cite{R00} for additional background on Pfaffian point processes.

Given two integers $a \leq b$, we let $\llbracket a, b \rrbracket$ denote the set $\{a, a+1, \dots, b\}$. We also set $\llbracket a,b \rrbracket = \emptyset$ when $a > b$, $\llbracket a, \infty \rrbracket = \{a, a+1, a+2 , \dots \}$, $\llbracket - \infty, b\rrbracket = \{b, b-1, b-2, \dots\}$ and $\llbracket - \infty, \infty \rrbracket = \mathbb{Z}$. We fix $m \in \mathbb{N}$ and $1 \leq M_1 < M_2 < \cdots < M_m \leq N$. For a random sequence of partitions $(\lambda^1, \dots, \lambda^N)$, we define the point process $\mathfrak{S}(\lambda)$ on $ \mathbb{R}^2$ (which is supported on $\llbracket 1, m \rrbracket \times \mathbb{Z}$) by
\begin{equation}\label{Eq.PointProcessSchur}
\mathfrak{S}(\lambda)(A) = \sum_{i \geq 1} \sum_{j = 1}^m {\bf 1}\{(j, \lambda_i^{N - M_j + 1} - i) \in A\}.
\end{equation}
Below, we let $C_r$ be the positively oriented zero-centered circle of radius $r > 0$, and let $\operatorname{Mat}_2(\mathbb{C})$ be the set of $2 \times 2$ matrices with complex entries.
\begin{proposition}\label{Prop.CorrKernel1} Assume the same notation as in Definition \ref{Def.SchurProcess} with parameters as in (\ref{Eq.HomogeneousParameters}), and let $\mathfrak{S}(\lambda)$ be as in (\ref{Eq.PointProcessSchur}). Then, $\mathfrak{S}(\lambda)$ is a Pfaffian point process on $\mathbb{R}^2$ with reference measure given by the counting measure on $\llbracket 1, m \rrbracket \times \mathbb{Z}$ and with correlation kernel $\kgeo: (\llbracket 1, m \rrbracket \times \mathbb{Z})^2 \rightarrow \operatorname{Mat}_2(\mathbb{C})$, given as follows. For each $u,v \in \llbracket 1, m \rrbracket$ and $x,y \in \mathbb{Z}$
\begin{equation}\label{Eq.K11Geo}
\begin{split}
\kgeo_{11}(u,x;v,y) = &\frac{1}{(2\pi \im)^2} \oint_{C_{r_1}} dz \oint_{C_{r_1}} dw \frac{z-w}{(z^2 -1)(w^2 - 1)(zw - 1)} \cdot (1 - c/z) (1-c/w) \\
& \times z^{-x} w^{-y} \cdot (1- q/z)^N (1 - q/w)^N(1-qz)^{-M_u} (1- qw)^{-M_v},
\end{split}
\end{equation}
where $r_1 \in (1, q^{-1})$. In addition,
\begin{equation}\label{Eq.K12Geo}
\begin{split}
\kgeo_{12}(u,x;v,y) &= -\kgeo_{21}(v,y; u,x) = \frac{1}{(2\pi \im)^2} \oint_{C_{r^z_{12}}} dz \oint_{C_{r^w_{12}}} dw \frac{zw-1}{z(z-w)(z^2-1)} \cdot \frac{z-c}{w-c} \\
& \times z^{-x} w^{y} \cdot (1- q/z)^N (1 - q/w)^{-N}(1-qz)^{-M_u} (1- qw)^{M_v},
\end{split}
\end{equation}
where $r_{12}^z \in (1, q^{-1})$, $r^w_{12} > \max(c,q)$, and $r_{12}^w < r_{12}^z$ when $u \leq v$, while $r_{12}^z < r_{12}^w$ when $u > v$. Finally, 
\begin{equation}\label{Eq.K22Geo}
\begin{split}
\kgeo_{22}(u,x;v,y) = &\frac{1}{(2\pi \im)^2} \oint_{C_{r_2}}  dz \oint_{C_{r_2}} dw \frac{z-w}{zw -1} \cdot \frac{1}{(z-c)(w-c)}\cdot z^{x} w^{y} \\
& \times (1- q/z)^{-N} (1 - q/w)^{-N}(1-qz)^{M_u} (1- qw)^{M_v},
\end{split}
\end{equation}
where $r_2 > \max(c,q,1)$. 
\end{proposition}
\begin{remark}\label{Rem.CorrKernel} Proposition \ref{Prop.CorrKernel1} is a special case of \cite[Theorem 3.3]{BR05} with the convention $i = m - u + 1$, $j = m - v + 1$, $u = x$ and $v = y$, corresponding to the choice of specializations as in Remark \ref{Rem.BRSpecial} with all $a_i = q$. The only differences between our formulas and those in \cite[Theorem 3.3]{BR05}, see also \cite[Section 4.2]{BBCS18}, are that we changed variables $w \rightarrow 1/w$ within $K_{12}$ and $w \rightarrow 1/w, z \rightarrow 1/z$ within $K_{22}$. We also mention that the formula for $K_{22}$ in \cite[Theorem 3.3]{BR05} has a small typo and the factor $(1-zw)$ in the denominator of the integrand should be replaced with $(zw-1)$, cf. \cite[Remark 4.1]{BBCS18} and \cite[Remark 2.6]{G19}.
\end{remark}

%
%
\subsection{Parameter scaling for the bottom curves}\label{Section2.3} In this section we introduce a rescaling of the Pfaffian Schur processes from Definition \ref{Def.SchurProcess}, which captures the asymptotic behavior of $\{ \lambda_i^j: i \geq 2, j \in \llbracket 1, N \rrbracket\}$. Subsequently, we derive an alternative formula for the correlation kernel from Proposition \ref{Prop.CorrKernel1} that is suitable for asymptotic analysis in this scaling.

We summarize various parameters that appear throughout the paper in the following definition. 
\begin{definition}\label{Def.ParametersBulk} Fix $q \in (0,1)$, $c \in (1, q^{-1})$, and set
\begin{equation}\label{Eq.Kappa0}
\kappa_0 = \left( \frac{1 - qc}{c-q}\right)^2 \in (0,1).
\end{equation}
For $\kappa \in (\kappa_0, 1)$, we define
\begin{equation}\label{Eq.ParBulk1}
\begin{split}
&\sigmaq = \frac{q^{1/3} (q + \sqrt{\kappa})^{5/3} }{\kappa^{1/6} (1 - q^2)^{2/3}(1+q \sqrt{\kappa})^{1/3}}, \hspace{2mm} \fq = \frac{q^{1/3}}{2 \kappa^{2/3}(1-q^2)^{2/3} (q + \sqrt{\kappa})^{1/3} (1 + q \sqrt{\kappa})^{1/3}}, \hspace{2mm} \\
&\zc(\kappa) = \frac{1 + q \sqrt{\kappa}}{q + \sqrt{\kappa}}, \hspace{2mm} \pq= \frac{q (1 + q \sqrt{\kappa})  }{(1-q^2) \sqrt{\kappa} }, \hspace{2mm} \hq  = \frac{ q (q + 2 \sqrt{\kappa} + q \kappa)}{1-q^2}.
\end{split}
\end{equation}
We further set
\begin{equation}\label{Eq.ParEdge2}
\begin{split}
&\pp = \frac{qc}{1- qc}, \hspace{1mm}\sigmap = \sqrt{\pp (1 + \pp)} =  \frac{q^{1/2} c^{1/2}}{1 - qc}, \hspace{1mm} \hp(\kappa) =    \kappa \pp + \frac{q}{c-q} = \frac{qc \kappa ( c- q) + q (1-qc) }{(1- qc)(c -q)}.
\end{split}
\end{equation}
We mention that the above parameters satisfy $q^{-1} > c > \zc > 1 > q$. When $\kappa$ is clear from the context, we drop it from the notation and simply write $\zc$ and $\hp$. 
\end{definition}

We next introduce how we rescale our random partitions in the following definition.
\begin{definition}\label{Def.ScalingBulk} Assume the same parameters as in Definition \ref{Def.ParametersBulk}. Fix $m \in \mathbb{N}$, $t_1, \dots, t_m \in \mathbb{R}$ with $t_1 < t_2 < \cdots < t_m$, and set $\mathcal{T} = \{t_1, \dots, t_m\}$. We also define for $t \in \mathcal{T}$ the quantity $T_{t} = T_t(N) =  \lfloor t N^{2/3} \rfloor$ and the lattice $\Lambda_t(N) = a_t(N) \cdot \mathbb{Z} + b_t(N)$, where 
\begin{equation}\label{Eq.LatticeBulk}
a_t(N) = (\sigmaq\zc)^{-1} N^{-1/3}, \mbox{ and } b_t(N) = (\sigmaq\zc)^{-1} N^{-1/3} \cdot \left( - \hq N - \pq T_{t} \right).
\end{equation}

We let $\mathbb{P}_N$ be the Pfaffian Schur process from Definition \ref{Def.SchurProcess} with parameters as in (\ref{Eq.HomogeneousParameters}). Here, we assume that $N$ is sufficiently large so that 
\begin{equation}\label{Eq.LargeNBulk}
N \geq \lfloor \kappa N \rfloor +  T_{t_m}(N) >  \lfloor \kappa N \rfloor +  T_{t_{m-1}}(N) > \cdots >  \lfloor \kappa N \rfloor +  T_{t_1}(N) \geq 1.
\end{equation}
If $(\lambda^1, \dots, \lambda^N)$ have law $\mathbb{P}_N$, we define the random variables
\begin{equation}\label{Eq.XsBulk}
X_i^{j,N} = (\sigmaq\zc)^{-1} N^{-1/3} \cdot \left( \lambda_i^{N - \lfloor \kappa N \rfloor   -  T_{t_j} + 1} - \hq N - \pq T_{t_j}  - i\right) \mbox{ for } i \in \mathbb{N} \mbox{ and } j \in \llbracket 1, m \rrbracket.
\end{equation}
\end{definition}

We next introduce certain functions, contours and measures that will be used to define our alternative correlation kernel.
\begin{definition}\label{Def.SGBulk} Assume the same parameters as in Definition \ref{Def.ParametersBulk}. For $z \in \mathbb{C} \setminus \{0, q, q^{-1}\}$, we introduce the functions
\begin{equation}\label{Eq.SGBulk}
    \begin{split}
        &\SFb(z) = \log(1-q/z) - \kappa \cdot \log(1-qz) - \hq \cdot \log(z), \hspace{2mm} \bar{\SFb}(z) = \SFb(z) - \SFb(\zc), \\
        &\GFb(z) = -\log(1-qz) - \pq \cdot \log(z), \hspace{2mm} \bar{\GFb}(z) = \GFb(z) - \GFb(\zc).
    \end{split}
\end{equation}
In equation (\ref{Eq.SGBulk}) and in the rest of the paper we always take the principal branch of the logarithm.
\end{definition}

\begin{definition}\label{Def.ContoursBulk} Assume the same parameters as in Definition \ref{Def.ParametersBulk}. Suppose $x \in [\zc, c]$, $\theta \in (0, \pi)$, $R > q^{-1}$ and $r \in [0, q^{-1} - c]$. With this data we define the contour $C(x, \theta, R,r)$ as follows. Let $z^{\pm}$ be the points where the rays $\{x + te^{\pm \im \theta}: t \geq 0\}$ intersect the $x$-centered circle of radius $r$, and let $\zeta^{\pm}$ be the points where they intersect the $0$-centered circle of radius $R$. The contour $C(x, \theta, R,r)$ consists of three oriented segments that connect $\zeta^-$ to $z^-$, $z^-$ to $z^+$, and $z^+$ to $\zeta^+$, as well as the counterclockwise oriented circular arc of the $0$-centered circle that connects $\zeta^+$ to $\zeta^-$. See the left side of Figure \ref{Fig.ContoursBulk}. We also recall that $C_r$ is the positively oriented zero-centered circle of radius $r > 0$. 
\end{definition}

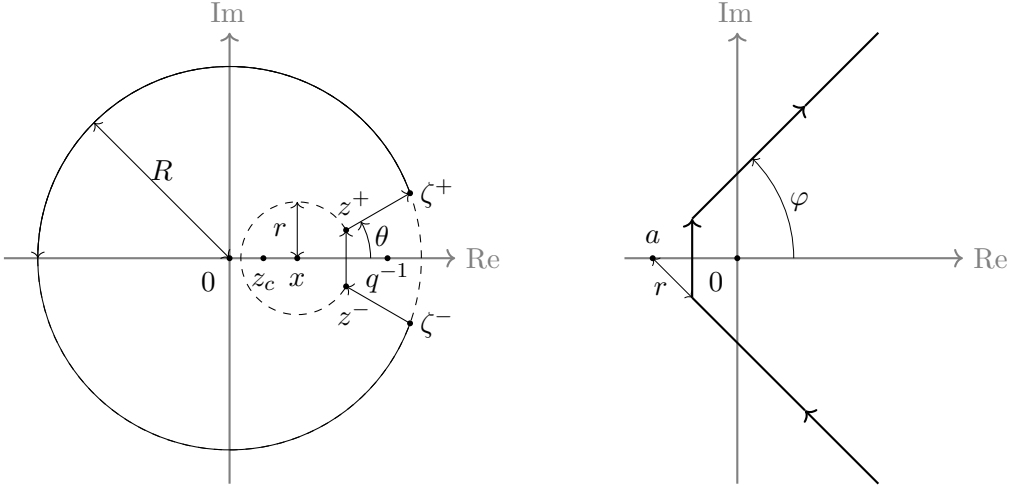
\begin{figure}[h]
    \centering
     \begin{tikzpicture}[scale=0.75]

        \def\tra{9} 
        \draw[->,thick,gray] (-4,0)--(4.0,0) node[right] {$\Real$};
  \draw[->,thick,gray] (0,-4)--(0,4) node[above] {$\Imag$};

  \def\x{1.2}   
  \def\R{3.4}   
  \def\r{1.0}   
  \def\th{30}   

  \node (O) at (0,0) {};
  \node (X) at (\x,0) {};
  
  \path (X) ++(\th:1) coordinate (Xp);
  \path (X) ++(-\th:1) coordinate (Xm);

  \path[name path=BigCirc]   (O) circle[radius=\R];
  \path[name path=SmallCirc] (X) circle[radius=\r];
  \path[name path=RayPlus]   (X) -- ($(X)+10*(\th:1)$);
  \path[name path=RayMinus]  (X) -- ($(X)+10*(-\th:1)$);

  \path[name intersections={of=RayPlus and BigCirc,by=ZetaPlus}];
  \path[name intersections={of=RayMinus and BigCirc,by=ZetaMinus}];
  \path[name intersections={of=RayPlus and SmallCirc,by=zPlus}];
  \path[name intersections={of=RayMinus and SmallCirc,by=zMinus}];


  \draw[dashed] (O) circle[radius=\R];
  \draw[dashed] (zPlus) arc (30:330:\r);

  \draw[->] (ZetaMinus) -- (zMinus);
  \draw[->] (zMinus) -- (zPlus);
  \draw[->] (zPlus) -- (ZetaPlus);
 \draw[<->,thin] (0,0) -- ++(135:{\R})node[pos=0.5, above] {$R$};
 \draw[<->,thin] (\x,0) -- ++(90:{\r}) node[pos=0.5, left] {$r$};
  \draw let
      \p1 = (ZetaPlus),
      \p2 = (ZetaMinus),
      \n1 = {atan2(\y1,\x1)},                    
      \n2 = {atan2(\y2,\x2)},                    
      \n3 = {ifthenelse(\n2<\n1, \n2+360, \n2)}, 
      \n4 = {(\n1+\n3)/2}                        
    in
      [->] (ZetaPlus) arc[start angle=\n1, end angle=\n4, radius=\R];

  \draw let
      \p1 = (ZetaPlus),
      \p2 = (ZetaMinus),
      \n1 = {atan2(\y1,\x1)},
      \n2 = {atan2(\y2,\x2)},
      \n3 = {ifthenelse(\n2<\n1, \n2+360, \n2)}
    in
      (ZetaPlus) arc[start angle=\n1, end angle=\n3, radius=\R];

  
  \fill (O) circle (1.5pt) node[below left=2pt] {$0$};
  \fill (X) circle (1.5pt) node[below=2pt] {$x$};
   \fill (0.6,0) circle (1.5pt) node[below=2pt]{$z_c$};
  \fill (2.8,0) circle (1.5pt) node[below=0pt, yshift = 2pt] {$q^{-1}$};
  \fill (ZetaPlus)  circle (1.5pt) node[right=0pt] {$\zeta^{+}$};
  \fill (ZetaMinus) circle (1.5pt) node[right=0pt] {$\zeta^{-}$};
  \fill (zPlus)     circle (1.5pt) node[above = 2pt, xshift = 3pt]  {$z^{+}$};
  \fill (zMinus)    circle (1.5pt) node[below =2pt, xshift = 3pt] {$z^{-}$};

  \draw[->] (\x+1.3,0) arc (0:30:1.3);
  \node at (\x+1.5,0.4) {$\theta$};

        \draw[->, thick, gray] ({\tra -2},0)--({\tra + 4},0) node[right]{$\Real$};
        \draw[->, thick, gray] ({\tra + 0},-4)--({\tra + 0},4) node[above]{$\Imag$};
        \fill (\tra, 0) circle (1.5pt) node[below left=2pt] {$0$};

        \fill (\tra - 1.5,0) circle (1.5pt) node[above =2pt] {$a$};
        \draw[->, thick] ({\tra -0.8},-0.7)--({\tra -0.8},0.7);
        \draw[->, thick] ({\tra -0.8},0.7)--({\tra + -0.8 + 2},0.7 + 2);
        \draw[-, thick] ({\tra + -0.8 + 2},0.7 + 2)--({\tra -0.8 + 3.3},4);
         \draw[-, thick] ({\tra -0.8 + 2},-2 - 0.7)--({\tra -0.8},-0.7);
        \draw[->, thick] ({\tra + -0.8 + 3.3},-4)--({\tra -0.8 + 2},-2 - 0.7);

         \draw[<->, very thin] (\tra -1.5,0)--(\tra - 0.8, -0.7) node[pos=0.5, below left, xshift = 2pt, yshift = 2pt] {$r$};
        \draw[->, very thin] (\tra + 1,0) arc (0:45:2.5);
        \node at (\tra + 1.1,1) {$\varphi$};

    \end{tikzpicture} 
    \caption{The left side depicts the contours $C(x,\theta, R,r)$ from Definition \ref{Def.ContoursBulk}. The right side depicts the contours $\mathcal{C}_a^{\varphi}[r]$, defined above (\ref{Eq.I11Vanish}).}
    \label{Fig.ContoursBulk}
\end{figure}

\begin{definition} \label{Def:ScaledLatticeMeasures}
Fix $m \in \mathbb{N}$,  $t_1 < \cdots < t_m$, and set $\mathcal{T} = \{t_1, \dots, t_m\}$. If $\nu = (\nu_{t_1}, \dots, \nu_{t_m})$ is an $m$-tuple of locally finite measures on $\mathbb{R}$, we define the (locally finite) measure $\mu_{\mathcal{T},\nu}$ on $\mathbb{R}^2$ by
\begin{equation}\label{Eq.MuToNu}
\mu_{\mathcal{T},\nu}(A) = \sum_{t \in \mathcal{T}} \nu_t(A_{t}), \mbox{ where } A_{t} = \{ y \in \mathbb{R}: (t,y) \in A\}.
\end{equation}
\end{definition}

With the above notation in place we can state the main result of this section.

\begin{lemma}\label{Lem.PrelimitKernelBulk} Assume the same notation as in Definitions \ref{Def.ParametersBulk}, \ref{Def.ScalingBulk}, and \ref{Def.SGBulk}. In addition, fix any $\theta \in (\pi/4, \pi/2)$ and $R > q^{-1}$, and let 
\begin{equation}\label{Eq.PrelimitContours}
    \begin{split}
        &\Gamma_N = C(\zc,\theta, R, \sec(\theta) N^{-1/3}), \hspace{2mm} \gamma_N = C\left(\zc, 2\pi/3, \sqrt{\zc^2 + N^{-1/6} - \zc N^{-1/12}}, 0  \right), \\
        &\tilde{\gamma}_N = C\left(\zc, \pi/2, \sqrt{\zc^2 + N^{-1/6}}, 0  \right)
    \end{split}
\end{equation}
be as in Definition \ref{Def.ContoursBulk}. Let $M^N$ be the point process on $\mathbb{R}^2$, formed by $\{(t_j, X_i^{j,N}): i \geq 1, j \in \llbracket 1, m\rrbracket \}$. Then, for all large $N$ (depending on $q,\kappa, c, \mathcal{T}$ and $\theta$) the $M^N$ is a Pfaffian point process with reference measure $\mu_{\mathcal{T},\nu(N)}$ and correlation kernel $K^N$ that are defined as follows. 

The measure $\mu_{\mathcal{T},\nu(N)}$ is as in Definition \ref{Def:ScaledLatticeMeasures} for $\nu(N) = (\nu_{t_1}(N), \dots, \nu_{t_m}(N))$, where $\nu_{t}(N)$ is $(\sigmaq\zc)^{-1} N^{-1/3}$ times the counting measure on $\Lambda_{t}(N)$. 

The correlation kernel $K^N: (\mathcal{T} \times \mathbb{R}) \times (\mathcal{T} \times \mathbb{R}) \rightarrow\operatorname{Mat}_2(\mathbb{C})$ takes the form
\begin{equation}\label{Eq:BulkKerDecomp}
\begin{split}
&K^N(s,x; t,y) = \begin{bmatrix}
    K^N_{11}(s,x;t,y) & K^N_{12}(s,x;t,y)\\
    K^N_{21}(s,x;t,y) & K^N_{22}(s,x;t,y) 
\end{bmatrix} \\
&= \begin{bmatrix}
    I^N_{11}(s,x;t,y) & I^N_{12}(s,x;t,y) + R^N_{12}(s,x;t,y) \\
    -I^N_{12}(t,y;s,x) - R^N_{12}(t,y;s,x) & I^N_{22}(s,x;t,y) + R^N_{22}(s,x;t,y) 
\end{bmatrix},
\end{split}
\end{equation}
where $I^N_{ij}(s,x;t,y), R^N_{ij}(s,x;t,y)$ are defined as follows. The kernels $I^N_{ij}$ are given by
\begin{equation}\label{Eq.DefIN11Bulk}
\begin{split}
&I^N_{11}(s,x;t,y) = \frac{1}{(2\pi \im)^{2}}\oint_{\Gamma_N} dz \oint_{\Gamma_N} dw F_{11}^N(z,w) H_{11}^N(z,w) \mbox{, where }\\
& F^N_{11}(z,w) = e^{N\bar{\SFb}(z) + N\bar{\SFb}(w)} \cdot e^{T_s \bar{\GFb}(z) + T_t \bar{\GFb}(w)} \cdot e^{- \sigmaq \zc x N^{1/3} \log (z/\zc) - \sigmaq \zc y N^{1/3} \log(w/\zc)  }, \\
&H^N_{11}(z,w) = \sigmaq \zc N^{1/3} \cdot  \frac{(z-w)( 1 - c/z) (1 - c/w)(1-qz)^{\kappa N - \lfloor \kappa N \rfloor}(1-qw)^{\kappa N - \lfloor \kappa N \rfloor} }{(z^{2}-1)(w^{2}-1)(zw-1)};
 \end{split}
\end{equation}
\begin{equation}\label{Eq.DefIN12Bulk}
\begin{split}
&I^N_{12}(s,x;t,y) = \frac{1}{(2\pi \im)^{2}}\oint_{\Gamma_N} dz \oint_{\gamma_N} dw F_{12}^N(z,w) H_{12}^N(z,w) \mbox{, where }\\
& F^N_{12}(z,w) = e^{N\bar{\SFb}(z) - N\bar{\SFb}(w)} \cdot e^{T_s \bar{\GFb}(z) - T_t \bar{\GFb}(w)} \cdot e^{- \sigmaq \zc x N^{1/3} \log (z/\zc) + \sigmaq \zc y N^{1/3} \log(w/\zc)  }, \\
&H^N_{12}(z,w) =  \sigmaq \zc N^{1/3} \cdot \frac{(zw - 1)(z-c)(1-qz)^{\kappa N - \lfloor \kappa N \rfloor}}{z (z-w)(z^2 - 1) (w-c)(1-qw)^{\kappa N - \lfloor \kappa N \rfloor}} ;
\end{split}
\end{equation}
\begin{equation}\label{Eq.DefIN22Bulk}
\begin{split}
&I^N_{22}(s,x;t,y) = \frac{1}{(2\pi \im)^{2}}\oint_{\gamma_N} dz \oint_{\gamma_N} dw F_{22}^N(z,w) H_{22}^N(z,w) \mbox{, where }\\
& F^N_{22}(z,w) = e^{-N\bar{\SFb}(z) - N\bar{\SFb}(w)} \cdot e^{-T_s \bar{\GFb}(z) - T_t \bar{\GFb}(w)} \cdot e^{ \sigmaq\zc x N^{1/3} \log (z/\zc) + \sigmaq \zc y N^{1/3} \log(w/\zc)  }, \\
&H^N_{22}(z,w) =   \sigmaq \zc N^{1/3}\cdot \frac{(z-w)}{(zw - 1)(z- c)(w - c)(1-qz)^{\kappa N - \lfloor \kappa N \rfloor}(1-qw)^{\kappa N - \lfloor \kappa N \rfloor}}.
\end{split}
\end{equation}
The kernels $R^N_{ij}$ are given by
\begin{equation}\label{Eq.DefRN12Bulk}
\begin{split}
R^N_{12}(s,x;t,y) = &\frac{-{\bf 1}\{s > t \} \cdot \sigmaq \zc N^{1/3} }{2 \pi \im} \oint_{\tilde{\gamma}_{N}}\frac{dz}{z} e^{(T_s - T_t) \bar{\GFb}(z)} \cdot e^{ \sigmaq \zc N^{1/3}(y-x)  \log (z/\zc)} \\
& + \frac{\sigmaq \zc N^{1/3}}{2\pi \im} \oint_{\Gamma_N} dz \frac{F_{12}^N(z,c) (zc-1) (1-qz)^{\kappa N - \lfloor \kappa N \rfloor}}{z(z^2-1)(1-qc)^{\kappa N - \lfloor \kappa N \rfloor}};
\end{split}
\end{equation}
\begin{equation}\label{Eq.DefRN22Bulk}
\begin{split}
R^N_{22}(s,x;t,y) = & \frac{\sigmaq \zc N^{1/3}}{2\pi \im} \oint_{\gamma_N} dz \frac{F_{22}^N(z,c)}{(c z - 1)(1-qz)^{\kappa N - \lfloor \kappa N \rfloor}(1-qc)^{\kappa N - \lfloor \kappa N \rfloor}}\\
&- \frac{\sigmaq \zc N^{1/3}}{2\pi \im} \oint_{\gamma_N} dw \frac{F_{22}^N(c,w)}{(c  w - 1)(1-qc)^{\kappa N - \lfloor \kappa N \rfloor}(1-qw)^{\kappa N - \lfloor \kappa N \rfloor}}.
\end{split}
\end{equation}
\end{lemma}
\begin{proof} From Definition \ref{Def.ParametersBulk} we have $q^{-1} > c > \zc > 1$, and so we can find $N_0$, depending on $q, \kappa, c$ and $\theta$, such that for $N \geq N_0$ we have $q^{-1} - c \geq  \sec(\theta) N^{-1/3}$ and also for $z \in \Gamma_N$ and $w \in \gamma_N$
\begin{equation}\label{Eq.ContoursNestedBulk}
|z| \geq \zc + N^{-1/3} > \zc \geq |w| > \zc - N^{-1/12} > 1.   
\end{equation}
Throughout the proof we assume that $N$ is sufficiently large so that $N \geq N_0$ and (\ref{Eq.LargeNBulk}) holds. Let $f: \mathbb{R} \rightarrow \mathbb{R}$ be a piece-wise linear increasing bijection such that $f(i) = t_i$ for $i \in \llbracket 1, m \rrbracket$. Define $\phi_N: \mathbb{R}^2 \rightarrow \mathbb{R}^2$ through 
$$\phi_N(s, x) = \left(f(s), (\sigmaq\zc)^{-1} N^{-1/3} \cdot \left( x- \hq N - \pq \lfloor f(s) N^{2/3} \rfloor \right) \right),$$   
and observe that $M^N = \mathfrak{S}(\lambda) \phi_N^{-1}$, where $\mathfrak{S}(\lambda)$ is as in (\ref{Eq.PointProcessSchur}) for $M_j = \lfloor \kappa N \rfloor + T_{t_j}$. It follows from Proposition \ref{Prop.CorrKernel1} and the change of variables formula \cite[Proposition 5.8(5)]{DY25} with the above $\phi_N$, \cite[Proposition 5.8(4)]{DY25} with 
\begin{equation}
f(s,x) =  \exp \left(\sigmaq \zc x N^{1/3} \cdot \log (\zc) - T_s \GFb(\zc) - N \SFb(\zc) \right),
\end{equation}
and \cite[Proposition 5.8(4)]{DY25} with $c_1 = c_2 = (\sigmaq\zc)^{1/2} N^{1/6}$ that $M^N$ is a Pfaffian point process with reference measure $\mu_{\mathcal{T},\nu(N)}$ and correlation kernel $\tilde{K}^N: (\mathcal{T} \times \mathbb{R})^2 \rightarrow\operatorname{Mat}_2(\mathbb{C})$, given by
\begin{equation*}
\tilde{K}^N(s,x;t,y) = \begin{bmatrix} \sigmaq \zc N^{1/3} f(s,x)f(t,y) \kgeo_{11}(\tilde{s},\tilde{x}; \tilde{t},\tilde{y}) &  \sigmaq \zc N^{1/3} \frac{f(s,x)}{f(t,y)} \kgeo_{12}(\tilde{s},\tilde{x}; \tilde{t},\tilde{y}) \\ \sigmaq \zc N^{1/3} \frac{f(t,y)}{f(s,x)} \kgeo_{21}(\tilde{s},\tilde{x}; \tilde{t},\tilde{y}) & \sigmaq \zc N^{1/3} \frac{1}{f(s,x)f(t,y)} \kgeo_{22}(\tilde{s},\tilde{x}; \tilde{t},\tilde{y}) \end{bmatrix},
\end{equation*}
where $\kgeo$ is as in Proposition \ref{Prop.CorrKernel1}, $\tilde{s} = f^{-1}(s)$, $\tilde{t} = f^{-1}(t)$ and
\begin{equation*}
\tilde{x} = \hq N + \pq T_s + \sigmaq \zc N^{1/3} x, \hspace{2mm}\tilde{y} = \hq N + \pq T_t + \sigmaq \zc N^{1/3} y.
\end{equation*}
All that remains is to show that $\tilde{K}^N$ agrees with $K^N$ as in the statement of the lemma.\\

We note that the following identities hold:
\begin{equation}\label{Eq.ChangeVarsBulk}
\begin{split}
&z^{\mp \tilde{x}} (1-q/z)^{\pm N}(1-qz)^{\mp \kappa N \mp T_s} f(s,x)^{\pm 1} = e^{\pm N \bar{\SFb}(z) \pm T_s \bar{\GFb}(z) \mp \sigmaq \zc x N^{1/3} \log(z/\zc) }, \\
&w^{\mp \tilde{y}} (1-q/w)^{\pm N}(1-qw)^{\mp \kappa N \mp T_t} f(t,y)^{\pm 1} = e^{ \pm N \bar{\SFb}(w) \pm T_t \bar{\GFb}(w) \mp \sigmaq \zc y  N^{1/3} \log(w/\zc)}.
\end{split}
\end{equation}

{\bf \raggedleft Matching $K^N_{11}$.} We may deform both contours $C_{r_1}$ in the definition of $\kgeo_{11}$ in Proposition \ref{Prop.CorrKernel1} to $\Gamma_N$ without crossing any of the poles of the integrand and hence without affecting the value of the integral by Cauchy's theorem. The reason we do not cross any poles is because $\Gamma_N$ encloses the unit circle $C_1$, see (\ref{Eq.ContoursNestedBulk}). After we perform the deformation, multiply by $\sigmaq \zc N^{1/3} f(s,x)f(t,y)$ and apply (\ref{Eq.ChangeVarsBulk})  we obtain $\tilde{K}^N_{11}( s,x; t,y) = I^N_{11}(s,x; t,y)$.\\

{\bf \raggedleft Matching  $K^N_{12}$ and $K^N_{21}$.} Since $\tilde{K}^N$ and $K^N$ are both skew-symmetric, it suffices to match $K^N_{12}$. We proceed to deform $C_{r_{12}^w}$ to $\gamma_N$, and $C_{r_{12}^z}$ to $\Gamma_N$. If $s \leq t$, then in the process of deformation we cross only the simple pole at $w = c$. If $s > t$, we further cross the simple pole at $z = w$. By the Residue theorem we obtain
\begin{equation}\label{eq:K12Res}
\begin{split}
&\kgeo_{12}(\tilde{s},\tilde{x}; \tilde{t},\tilde{y}) = \frac{1}{(2\pi \im)^{2}}\oint_{\Gamma_N} d  z \oint_{\gamma_N}  dw  \frac{zw-1}{z(z-w)(z^{2}-1)} \cdot \frac{z-c}{w-c} \cdot  z^{-\tilde{x} } w^{\tilde{y} }  \\
&  \times     (1-q/z)^{N} (1-q/w)^{-N}(1-qz)^{- \lfloor \kappa N \rfloor - T_s}   (1-qw)^{\lfloor \kappa N \rfloor + T_t} \\
& + \frac{1}{2\pi \im}\oint_{\Gamma_N} d  z  \frac{zc-1}{z(z^{2}-1)} \cdot  z^{-\tilde{x} } c^{\tilde{y} }  \cdot  (1-q/z)^{N} (1-q/c)^{-N}(1-qz)^{- \lfloor \kappa N \rfloor - T_s}   (1-qc)^{\lfloor \kappa N \rfloor + T_t} \\
& + \frac{-{\bf 1}\{s > t\} }{2\pi \im} \oint_{\Gamma_N} \frac{dz}{z} (1-q/z)^{T_t-T_s} \cdot z^{\tilde{y}-\tilde{x}}.
\end{split}
\end{equation}
Using (\ref{Eq.ChangeVarsBulk}), one readily verifies that 
\begin{equation}\label{eq:K12Match1}
\begin{split}
&\sigmaq \zc N^{1/3}\cdot \frac{f(s,x)}{f(t,y)} \times [\mbox{lines 1 and 2 in (\ref{eq:K12Res})}] =  I_{12}^N(s,x; t,y) , \mbox{ and } \\
&\sigmaq \zc N^{1/3} \cdot \frac{f(s,x)}{f(t,y)} \times [\mbox{lines 3 and 4 in (\ref{eq:K12Res})}] = R_{12}^N(s,x;t,y).
\end{split}
\end{equation}
We mention that in the second line in (\ref{eq:K12Match1}), when matching the fourth line in (\ref{eq:K12Res}) with the first line in (\ref{Eq.DefRN12Bulk}) we deformed $\Gamma_N$ to $\tilde{\gamma}_N$, which can be done without crossing any poles. Equations (\ref{eq:K12Res}) and (\ref{eq:K12Match1}) show that $\tilde{K}^N_{12}( s,x; t,y) = I^N_{12}(s,x; t,y) + R^N_{12}(s,x;t,y)$.\\

{\bf \raggedleft Matching  $K^N_{22}$.} Starting from the formula for $\kgeo_{22}$ in Proposition \ref{Prop.CorrKernel1} with $r_2 $ large (say $r_2 \geq 1 + c$), we deform both contours $C_{r_2}$ to $\gamma_N$. In the process of deformation we pick up two residues from the simple poles at $w = c$ and $z = c$, see (\ref{Eq.ContoursNestedBulk}). We thus obtain the formula
\begin{equation}\label{eq:K22Res1}
\begin{split}
&\kgeo_{22}(\tilde{s},\tilde{x}; \tilde{t},\tilde{y}) = \frac{1}{(2\pi \im)^{2}}\oint_{\gamma_N} d  z \oint_{\gamma_N}  dw \frac{z-w}{zw-1} \cdot \frac{1}{(z-c)(w-c)} \cdot z^{\tilde{x}}w^{\tilde{y}} \\
&\times  (1-q/z)^{-N}(1-q/w)^{-N} (1-qz)^{\lfloor \kappa N \rfloor + T_s} (1-qw)^{\lfloor \kappa N \rfloor + T_t} \\
& + \frac{1}{2\pi \im} \oint_{\gamma_N} d z  \frac{1}{zc-1} \cdot  z^{\tilde{x}}c^{\tilde{y}}  \cdot (1-q/z)^{-N}(1-q/c)^{-N} (1-qz)^{\lfloor \kappa N \rfloor + T_s} (1-qc)^{\lfloor \kappa N \rfloor + T_t} \\
& - \frac{1}{2\pi \im} \oint_{\gamma_N} d w  \frac{1}{wc-1} \cdot  c^{\tilde{x}}w^{\tilde{y}}  \cdot (1-q/c)^{-N}(1-q/w)^{-N} (1-qc)^{\lfloor \kappa N \rfloor + T_s} (1-qw)^{\lfloor \kappa N \rfloor + T_t}.
\end{split}
\end{equation}
Using (\ref{Eq.ChangeVarsBulk}), one readily verifies that 
\begin{equation}\label{eq:K22Match1}
\begin{split}
&\sigmaq \zc N^{1/3} \cdot \frac{1}{f(s,x)f(t,y)} \times [\mbox{lines 1 and 2 in (\ref{eq:K22Res1})}] = I_{22}^N(s,x; t,y) \\
&\sigmaq \zc N^{1/3} \cdot  \frac{1}{f(s,x)f(t,y)} \times [\mbox{lines 3 and 4 in (\ref{eq:K22Res1})}] = R_{22}^N(s,x; t,y).
\end{split}
\end{equation}
Equations (\ref{eq:K22Res1}) and (\ref{eq:K22Match1}) show that $\tilde{K}^N_{22}( s,x; t,y) = I^N_{22}(s,x; t,y) + R^N_{22}(s,x;t,y)$.
\end{proof}

%
%
\subsection{Parameter scaling for the top curve}\label{Section2.4} 
In this section we introduce a rescaling of the Schur processes from Definition \ref{Def.SchurProcess}, which captures the asymptotic behavior of $\{ \lambda_1^j: j \in \llbracket 1, N \rrbracket\}$. Subsequently we derive an alternative formula for the correlation kernel from Proposition \ref{Prop.CorrKernel1} that is suitable for the asymptotic analysis in this scaling.

We introduce how we rescale our random partitions in the following definition.
\begin{definition}\label{Def.ScalingEdge} Assume the same parameters as in Definition \ref{Def.ParametersBulk}. We further fix $m \in \mathbb{N}$, $\kappa_1, \dots, \kappa_m$, such that 
\begin{equation}\label{IneqKappa}
\kappa_0 < \kappa_1 < \kappa_2 < \cdots < \kappa_m < 1,
\end{equation}
and set $\mathcal{T} = \{\kappa_1, \dots, \kappa_m\}$. For $\kappa \in \mathcal{T}$, we define the lattice $\Lambda_{\kappa}(N) = a_{\kappa}(N) \cdot \mathbb{Z} + b_{\kappa}(N)$, where 
\begin{equation}\label{Eq.LatticeEdge}
a_{\kappa}(N) = \sigmap^{-1} N^{-1/2}, \mbox{ and } b_{\kappa}(N) =  - \sigmap^{-1}N^{1/2} \cdot \hp(\kappa).
\end{equation}

We let $\mathbb{P}_N$ be the Pfaffian Schur process from Definition \ref{Def.SchurProcess} with parameters as in (\ref{Eq.HomogeneousParameters}). Here, we assume that $N$ is sufficiently large so that 
\begin{equation}\label{Eq.LargeNEdge}
N \geq \lfloor \kappa_m N \rfloor >  \lfloor \kappa_{m-1} N \rfloor > \cdots > \lfloor \kappa_1 N \rfloor \geq 1.
\end{equation}
If $(\lambda^1, \dots, \lambda^N)$ have law $\mathbb{P}_N$, we define the random variables
\begin{equation}\label{Eq.YEdge}
Y_i^{j,N} = \sigmap^{-1} N^{-1/2} \cdot \left( \lambda_i^{N - \lfloor \kappa_j N \rfloor + 1} - \hp(\kappa_j)  N  - i\right) \mbox{ for } i \in \mathbb{N} \mbox{ and } j \in \llbracket 1, m \rrbracket.
\end{equation}
\end{definition}

We next introduce certain functions that will be used to define our alternative correlation kernel.
\begin{definition}\label{Def.SGEdge} Assume the same parameters as in Definition \ref{Def.ParametersBulk}. For $z \in \mathbb{C} \setminus \{0, q, q^{-1}\}$, we introduce the functions
\begin{equation}\label{Eq.SGEdge}
    \begin{split}
        &\SFt(\kappa,z) = \log(1-q/z) - \kappa \cdot \log(1-qz) - \hp (\kappa)\cdot \log(z), \hspace{2mm} \bar{\SFt}(\kappa,z) = \SFt(\kappa,z) - \SFt(\kappa, c), \\
       &\GFt(z) = -\log(1-qz) - \pp \cdot \log(z), \hspace{2mm} \bar{\GFt}(z) = \GFt(z) - \GFt(c).
    \end{split}
\end{equation}
When $\kappa$ is clear from the context, we drop it from the notation and simply write $\SFt(z)$. 
\end{definition}

With the above notation in place we can state the main result of this section.
\begin{lemma}\label{Lem.PrelimitKernelEdge} Assume the same notation as in Definitions \ref{Def.ParametersBulk}, \ref{Def.ScalingEdge}, and \ref{Def.SGEdge}. In addition, fix parameters $\theta_{\kappa_1}, \dots, \theta_{\kappa_m} \in (\pi/4, \pi/2)$, $R_{\kappa_1}, \dots, R_{\kappa_m} > q^{-1}$, and the contours
\begin{equation}\label{Eq.ContoursEdge}
\begin{split}
&\Gamma_{\kappa, N} = C\left(c,\theta_{\kappa}, R_\kappa, \sec(\theta_{\kappa}) N^{-1/2}\right), \hspace{2mm} \gamma_{\kappa} = C_{\zc(\kappa)}\mbox{ for } \kappa \in \mathcal{T}, \mbox{ and } \\
& \tilde{\gamma}_N = C\left(c, \pi/2, \sqrt{c^2 + N^{-1/6}}, 0 \right),
\end{split}
\end{equation}
as in Definition \ref{Def.ContoursBulk}. Let $M^N$ be the point process on $\mathbb{R}^2$, formed by $\{(\kappa_j, Y_i^{j,N}): i \geq 1, j \in \llbracket 1, m\rrbracket \}$. Then, for all large $N$ (depending on $q, c, \mathcal{T}$ and $\{\theta_{\kappa}: \kappa \in \mathcal{T}\}$) the $M^N$ is a Pfaffian point process with reference measure $\mu_{\mathcal{T},\nu(N)}$ and correlation kernel $K^N$ that are defined as follows. 

The measure $\mu_{\mathcal{T},\nu(N)}$ is as in Definition \ref{Def:ScaledLatticeMeasures} for $\nu(N) = (\nu_{\kappa_1}(N), \dots, \nu_{\kappa_m}(N))$, where $\nu_{\kappa}(N)$ is $\sigmap^{-1} N^{-1/2}$ times the counting measure on $\Lambda_{\kappa}(N)$. 

The correlation kernel $K^N: (\mathcal{T} \times \mathbb{R}) \times (\mathcal{T} \times \mathbb{R}) \rightarrow\operatorname{Mat}_2(\mathbb{C})$ takes the form
\begin{equation}\label{Eq:EdgeKerDecomp}
\begin{split}
&K^N(s,x; t,y) = \begin{bmatrix}
    K^N_{11}(s,x;t,y) & K^N_{12}(s,x;t,y)\\
    K^N_{21}(s,x;t,y) & K^N_{22}(s,x;t,y) 
\end{bmatrix} \\
&= \begin{bmatrix}
    I^N_{11}(s,x;t,y) & I^N_{12}(s,x;t,y) + R^N_{12}(s,x;t,y) \\
    -I^N_{12}(t,y;s,x) - R^N_{12}(t,y;s,x) & I^N_{22}(s,x;t,y) + R^N_{22}(s,x;t,y)
\end{bmatrix},
\end{split}
\end{equation}
where $I^N_{ij}(s,x;t,y), R^N_{ij}(s,x;t,y)$ are defined as follows. The kernels $I^N_{ij}$ are given by
\begin{equation}\label{Eq.DefIN11Edge}
\begin{split}
&I^N_{11}(s,x;t,y) = \frac{1}{(2\pi \im)^{2}}\oint_{\Gamma_{s, N}} dz \oint_{\Gamma_{t, N}} dw F_{11}^N(z,w) H_{11}^N(z,w) \mbox{, where }\\
& F^N_{11}(z,w) = e^{N\bar{\SFt}(s, z) + N\bar{\SFt}(t, w)} \cdot e^{-  \sigmap x N^{1/2} \log (z/c) -  \sigmap y N^{1/2} \log(w/c)  }, \\
&H^N_{11}(z,w) = \sigmap N^{1/2} \cdot  \frac{(z-w)( 1 - c/z) (1 - c/w)(1-qz)^{s N - \lfloor s N \rfloor}(1-qw)^{t N - \lfloor t N \rfloor} }{(z^{2}-1)(w^{2}-1)(zw-1)(1-qc)^{s N - \lfloor s N \rfloor}(1-qc)^{t N - \lfloor t N \rfloor} };
\end{split}
\end{equation}
\begin{equation}\label{Eq.DefIN12Edge}
\begin{split}
&I^N_{12}(s,x;t,y) = \frac{1}{(2\pi \im)^{2}}\oint_{\Gamma_{s, N}} dz \oint_{\gamma_{t}} dw F_{12}^N(z,w) H_{12}^N(z,w) \mbox{, where }\\
& F^N_{12}(z,w) = e^{N\bar{\SFt}(s,z) - N\bar{\SFt}(t,w)}\cdot e^{-  \sigmap x N^{1/2} \log (z/c) +  \sigmap y N^{1/2} \log(w/c)  }, \\
&H^N_{12}(z,w) =   \sigmap N^{1/2} \cdot \frac{(zw - 1)(z-c)(1-qz)^{s N - \lfloor s N \rfloor}(1-qc)^{t N - \lfloor t N \rfloor}}{z (z-w)(z^2 - 1) (w-c)(1-qw)^{t N - \lfloor t N \rfloor}(1-qc)^{s N - \lfloor s N \rfloor}} ;
\end{split}
\end{equation}
\begin{equation}\label{Eq.DefIN22Edge}
\begin{split}
&I^N_{22}(s,x;t,y) = \frac{1}{(2\pi \im)^{2}}\oint_{\gamma_s} dz \oint_{\gamma_{t}} dw F_{22}^N(z,w) H_{22}^N(z,w) \mbox{, where }\\
& F^N_{22}(z,w) = e^{-N\bar{\SFt}(s,z) - N\bar{\SFt}(t, w)}\cdot e^{  \sigmap x N^{1/2} \log (z/c) +  \sigmap y N^{1/2} \log(w/c)  }, \\
&H^N_{22}(z,w) =  \sigmap N^{1/2}\cdot \frac{(z-w)(1-qc)^{s N - \lfloor s N \rfloor}(1-qc)^{t N - \lfloor t N \rfloor}}{(zw - 1)(z- c)(w - c)(1-qz)^{s N - \lfloor s N \rfloor}(1-qw)^{t N - \lfloor t N \rfloor}}.
\end{split}
\end{equation}
The kernels $R^N_{ij}$ are given by
\begin{equation}\label{Eq.DefRN12Edge}
\begin{split}
&R^N_{12}(s,x;t,y) = \frac{-{\bf 1}\{s > t \}   \sigmap N^{1/2} }{2 \pi \im}  \oint_{\tilde{\gamma}_N} dz e^{(s-t)N \bar{\GFt}(z)} \cdot e^{\sigmap (-  x +  y) N^{1/2} \log(z/c)  } \\
&\times \frac{(1-qz)^{sN - \lfloor sN \rfloor}(1-qc)^{tN - \lfloor tN \rfloor}}{z(1-qz)^{tN - \lfloor tN \rfloor}(1-qc)^{sN - \lfloor sN \rfloor}} + \frac{\sigmap N^{1/2}}{2\pi \im} \oint_{\Gamma_{s,N}} dz \frac{F_{12}^N(z,c) (zc-1) (1-qz)^{s N - \lfloor s N \rfloor}}{z(z^2-1)(1-qc)^{s N - \lfloor s N \rfloor}};
\end{split}
\end{equation}
\begin{equation}\label{Eq.DefRN22Edge}
\begin{split}
R^N_{22}(s,x;t,y) = & \frac{\sigmap N^{1/2}}{2\pi \im} \oint_{\gamma_s} dz \frac{F_{22}^N(z,c)(1-qc)^{s N - \lfloor s N \rfloor}}{(c z - 1)(1-qz)^{s N - \lfloor s N \rfloor}}\\
&- \frac{\sigmap N^{1/2}}{2\pi \im} \oint_{\gamma_t} dw \frac{F_{22}^N(c,w)(1-qc)^{t N - \lfloor t N \rfloor}}{(c  w - 1)(1-qw)^{t N - \lfloor t N \rfloor}}.
\end{split}
\end{equation}
\end{lemma}
\begin{proof} From Definition \ref{Def.ParametersBulk} we can find $N_0$, depending on $q, c, \mathcal{T}$ and $\{\theta_{\kappa}: \kappa \in \mathcal{T}\}$, such that for $N \geq N_0$ and $\kappa \in \mathcal{T}$, we have $q^{-1} - c \geq \sec(\theta_{\kappa}) N^{-1/2}$, and also for $z \in \Gamma_{\kappa, N}$ and $w \in \gamma_{\kappa}$, we have 
\begin{equation}\label{Eq.ContoursNestedEdge}
|z| \geq c + N^{-1/2} > \zc(\kappa) = |w| > 1.   
\end{equation}
Throughout the proof we assume that $N$ is sufficiently large so that $N \geq N_0$ and (\ref{Eq.LargeNEdge}) holds.

Let $f: \mathbb{R} \rightarrow \mathbb{R}$ be a piece-wise linear increasing bijection, such that $f(i) = \kappa_i$ for $i \in \llbracket 1, m \rrbracket$. Define $\phi_N: \mathbb{R}^2 \rightarrow \mathbb{R}^2$ through 
$$\phi_N(s, x) = \left(f(s),  \sigmap^{-1}N^{-1/2} \cdot \left( x- \hp(f(s)) N  \right) \right),$$   
and observe that $M^N = \mathfrak{S}(\lambda) \phi_N^{-1}$, where $\mathfrak{S}(\lambda)$ is as in (\ref{Eq.PointProcessSchur}) for $M_j = \lfloor \kappa_j N \rfloor$. It follows from Proposition \ref{Prop.CorrKernel1} and the change of variables formula \cite[Proposition 5.8(5)]{DY25} with the above $\phi_N$, \cite[Proposition 5.8(4)]{DY25} with 
\begin{equation}
f(s,x) =  \exp \left( \sigmap x N^{1/2} \cdot \log (c)  - N \SFt(s,c) \right) \cdot \frac{1}{(1-qc)^{sN - \lfloor sN \rfloor}},
\end{equation}
and \cite[Proposition 5.8(4)]{DY25} with $c_1 = c_2 =  \sigmap^{1/2} N^{1/4}$ that $M^N$ is a Pfaffian point process with reference measure $\mu_{\mathcal{T},\nu(N)}$ and correlation kernel $\tilde{K}^N: (\mathcal{T} \times \mathbb{R}) \times (\mathcal{T} \times \mathbb{R}) \rightarrow\operatorname{Mat}_2(\mathbb{C})$, given by
\begin{equation*}
\tilde{K}^N(s,x; t,y) = \begin{bmatrix}  \sigmap N^{1/2} f(s,x)f(t,y) \kgeo_{11}(\tilde{s},\tilde{x}; \tilde{t},\tilde{y}) & \sigmap N^{1/2} \frac{f(s,x)}{f(t,y)} \kgeo_{12}(\tilde{s},\tilde{x}; \tilde{t},\tilde{y}) \\ \sigmap N^{1/2} \frac{f(t,y)}{f(s,x)} \kgeo_{21}(\tilde{s},\tilde{x}; \tilde{t},\tilde{y}) & \sigmap N^{1/2}\frac{1}{f(s,x)f(t,y)} \kgeo_{22}(\tilde{s},\tilde{x}; \tilde{t},\tilde{y}) \end{bmatrix},
\end{equation*}
where $\kgeo$ is as in Proposition \ref{Prop.CorrKernel1}, $\tilde{s}  = f^{-1}(s)$, $\tilde{t} =  f^{-1}(t)$ and
\begin{equation*}
\tilde{x} = \hp(s) N  +  \sigmap N^{1/2} x, \hspace{2mm}\tilde{y} = \hp(t) N  +  \sigmap N^{1/2} y.
\end{equation*}
All that remains is to show that $\tilde{K}^N$ agrees with $K^N$ as in the statement of the lemma.

We note that we have the following identities
\begin{equation}\label{Eq.ChangeVarsEdge}
\begin{split}
&z^{\mp \tilde{x}} (1-q/z)^{\pm N}(1-qz)^{\mp \kappa N } f(s,x)^{\pm 1} = e^{\pm N \bar{\SFt}(\kappa, z)  \mp \sigmap x N^{1/2} \log(z/c) } \cdot \frac{1}{(1-qc)^{\pm sN \mp \lfloor sN \rfloor}}, \\
&w^{\mp \tilde{y}} (1-q/w)^{\pm N}(1-qw)^{\mp \kappa N } f(t,y)^{\pm 1} = e^{ \pm N \bar{\SFt}(\kappa, w) \mp \sigmap y  N^{1/2} \log(w/c)} \cdot\frac{1}{(1-qc)^{\pm tN \mp \lfloor tN \rfloor}}.
\end{split}
\end{equation}
From this point on the proof is essentially the same as that of Lemma \ref{Lem.PrelimitKernelBulk}, where instead of using (\ref{Eq.ChangeVarsBulk}) we use (\ref{Eq.ChangeVarsEdge}). We omit the details.
\end{proof}

%
\section{Preliminary estimates}\label{Section3} In this section we establish various properties of the functions $\SFb, \GFb$ from Definition \ref{Def.SGBulk} and $\SFt, \GFt$ from Definition \ref{Def.SGEdge}. In particular, we find estimates for these functions near their critical points in Section \ref{Section3.1}, estimates for them along circular contours in Section \ref{Section3.2} and along the contours $C(x,\theta, R,0)$ from Definition \ref{Def.ContoursBulk} in Section \ref{Section3.3}.

%
\subsection{Power series expansions}\label{Section3.1} In this section we establish a few basic statements about the power series expansions of the functions $\SFb, \GFb, \SFt, \GFt$.

\begin{lemma}\label{Lem.PowerSeriesSG} Assume the notation from Definitions \ref{Def.ParametersBulk}, \ref{Def.SGBulk} and \ref{Def.SGEdge}. There exist constants $\delta_0 \in (0,1)$ and $C_0 > 0$, depending on $q, \kappa$ and $c$, such that $\SFb(z), \GFb(z)$ are analytic in the disk $\{|z- \zc| < 2\delta_0\}$, the functions $\SFt(z), \GFt(z)$ are analytic in the region $\{|z- c| < 2\delta_0\}$, and the following statements hold. If $|z - \zc| \leq \delta_0$, then
\begin{equation}\label{Eq.TaylorS1}
\left| \SFb(z) - \SFb(\zc) - (\sigmaq^3/3)(z- \zc)^3  \right| \leq C_0 |z - \zc|^4,
\end{equation}
\begin{equation}\label{Eq.TaylorG1}
\left| \GFb(z) - \GFb(\zc) - \fq\sigmaq^2 (z- \zc)^2  \right| \leq C_0 |z - \zc|^3. 
\end{equation}
If $|z-c| \leq \delta_0$, then
\begin{equation}\label{Eq.TaylorS2}
\left| \SFt(z) - \SFt(c) - [\sigmap^2(\kappa - \kappa_0)/2c^2] (z-c)^2 \right| \leq C_0 |z - c|^3,
\end{equation}
\begin{equation}\label{Eq.TaylorG2}
\left| \GFt(z) - \GFt(c) - (\sigmap^2/2c^2) (z-c)^2 \right| \leq C_0 |z - c|^3.
\end{equation}
\end{lemma}
\begin{proof} Note that as $q < 1 < z_c < c < q^{-1}$, the functions $\SFb(z), \GFb(z)$ and $\SFt(z), \GFt(z)$ are analytic in their respective region with $\delta_0 = (1/2) \cdot \min(1, z_c - q, q^{-1} - c)$. By a direct computation we have $\SFb'(\zc) = \SFb''(\zc) = \GFb'(\zc) = 0$, and 
$$\SFb'''(\zc) = \frac{2q (q + \sqrt{\kappa})^{5}}{\kappa^{1/2} (1+q \sqrt{\kappa}) (1 - q^2)^2}= 2 \sigmaq^3, \hspace{2mm} \GFb''(\zc) = \frac{q (q + \sqrt{\kappa})^3}{\kappa (1 - q^2)^2 (1 + q \sqrt{\kappa})} = 2\fq\sigmaq^2,$$
which imply (\ref{Eq.TaylorS1}) and (\ref{Eq.TaylorG1}). Another direct computation gives $\SFt'(c) = \GFt'(c) = 0$, and 
$$\SFt''(c) = \frac{q (q c + \sqrt{\kappa} c - 1 - q \sqrt{\kappa})(1 + \sqrt{\kappa} c - q \sqrt{\kappa} - qc) }{c (c-q)^2 (1 - qc)^2} = \frac{\sigmap^2(\kappa - \kappa_0)}{c^2} \mbox{, } \GFt''(c) = \frac{q}{c(1-qc)^2} = \frac{\sigmap^2}{c^2},$$
which imply (\ref{Eq.TaylorS2}) and (\ref{Eq.TaylorG2}).
\end{proof}

\begin{lemma}\label{Lem.DecayNearCritTheta} Assume the notation in Lemma \ref{Lem.PowerSeriesSG}, and fix $\theta \in (\pi/4, \pi/2)$. There exist constants $\delta_1 \in (0,\delta_0]$ and $\epsilon_1 > 0$, depending on $q, \kappa, c, \theta$, such that for $r \in [0, \delta_1]$
\begin{equation}\label{Eq.CritDecayS1}
\Real \left [\SFb(z) - \SFb(\zc) \right] \leq -\epsilon_1 \cdot r^3, \mbox{ if } z = \zc + r e^{\pm \im \theta}, 
\end{equation}
\begin{equation}\label{Eq.CritDecayS2}
\Real \left [\SFt(z) - \SFt(c) \right] \leq -\epsilon_1 \cdot r^2, \mbox{ if } z = c+r e^{\pm \im \theta}.
\end{equation}
\end{lemma}
\begin{proof} Set $A = \min (\sigmaq^3/6, \sigmap^2(\kappa - \kappa_0)/4c^2)$ and define $\epsilon_1 = \min \left( -A \cos(3\theta), - A \cos(2\theta) \right)$. We pick $\delta_1 \in (0, \delta_0]$ small enough so that $C_0 \delta_1 \leq \epsilon_1$.
If $z = \zc + r e^{\pm \im \theta}$, we have $\Real \left[ (z- \zc)^3 \right] = r^3 \cos(3 \theta)$, and so by (\ref{Eq.TaylorS1}) we conclude for $r \in [0, \delta_1]$ 
$$ \Real \left [\SFb(z) - \SFb(\zc) \right] \leq r^3 \left[(\sigmaq^3/3)\cos(3 \theta) + C_0 \delta_1\right] \leq r^3[-2\epsilon_1 + C_0\delta_1] \leq - \epsilon_1 \cdot r^3 .$$
If $z = c + r e^{\pm\im \theta}$, we have $\Real \left[ (z- c)^2 \right] = r^2 \cos(2 \theta)$, and so by (\ref{Eq.TaylorS2}) we conclude for $r \in [0, \delta_1]$ 
$$ \Real \left [\SFt(z) - \SFt(c) \right] \leq r^2 \left[\sigmap^2(\kappa - \kappa_0)/2c^2 \cos(2 \theta) + C_0 \delta_1\right] \leq r^2 \left[-2\epsilon_1 + C_0 \delta_1\right] \leq -\epsilon_1 \cdot r^2.$$
\end{proof}

\begin{lemma}\label{Lem.DecayNearCritGen} Assume the notation in Lemma \ref{Lem.PowerSeriesSG}. There exist constants $\delta_2 \in (0,\delta_0]$ and $\epsilon_2 > 0$, depending on $q, \kappa, c$, that satisfy the following statements for $r \in [0, \delta_2]$:
\begin{equation}\label{Eq.CritGrowS1}
\Real \left [\SFb(z) - \SFb(\zc) \right] \geq \epsilon_2 \cdot r^3, \mbox{ if } z = \zc + r e^{\pm \im 2\pi/3};
\end{equation}
\begin{equation}\label{Eq.CritDecayG1}
\Real \left [\GFb(z) - \GFb(\zc) \right] \leq -\epsilon_2 \cdot r^2, \mbox{ if } z = \zc + r e^{\pm \im \pi/2};  
\end{equation}
\begin{equation}\label{Eq.CritDecayG2}
\Real \left [\GFt(z) - \GFt(c) \right] \leq -\epsilon_2 \cdot r^2, \mbox{ if } z = c + r e^{\pm \im \pi/2}.
\end{equation}
\end{lemma}
\begin{proof} Set $\epsilon_2 = (1/2) \min (\sigmaq^3/3, \fq \sigmaq^2, \sigmap^2/2c^2 )$ and pick $\delta_2 \in (0, \delta_0]$ so that $C_0 \delta_2 \leq \epsilon_2$. If $z = \zc + r e^{\pm \im 2\pi/3}$, we have $\Real \left[ (z- \zc)^3 \right] = r^3$, and so by (\ref{Eq.TaylorS1}) we conclude for $r \in [0, \delta_2]$ 
$$ \Real \left [\SFb(z) - \SFb(\zc) \right] \geq r^3 \left[(\sigmaq^3/3) - C_0 \delta_2\right] \geq r^3 \left[2\epsilon_2 - \epsilon_2\right] \geq \epsilon_2 \cdot r^3.$$
If $z = \zc + r e^{\pm \im \pi/2}$, we have $\Real \left[ (z- \zc)^2 \right] = -r^2$, and so by (\ref{Eq.TaylorG1}) we conclude for $r \in [0, \delta_2]$ 
$$ \Real \left [\GFb(z) - \GFb(\zc) \right] \leq r^2 \left[-\fq \sigmaq^2 + C_0 \delta_2\right] \leq r^2 \left[ -2 \epsilon_2 + \epsilon_2 \right] \leq - \epsilon_2 \cdot r^2.$$
If $z = c + r e^{\pm\im \pi/2}$, we have $\Real \left[ (z- c)^2 \right] = -r^2$, and so by (\ref{Eq.TaylorG2}) we conclude for $r \in [0, \delta_2]$ 
$$ \Real \left [\GFt(z) - \GFt(c) \right] \leq r^2 \left[-(\sigmap^2/2c^2) + C_0 \delta_2\right] \leq r^2 \left[ -2 \epsilon_2 + \epsilon_2 \right] \leq - \epsilon_2 \cdot r^2.$$
\end{proof}

We also require the following two simple results.
\begin{lemma}\label{Lem.DiffS} Assume the notation from Definitions \ref{Def.ParametersBulk} and \ref{Def.SGBulk}. Then,
\begin{equation}\label{Eq.DiffS}
\SFb(\zc) - \SFb(c) < 0. 
\end{equation}
\end{lemma}
\begin{proof} Recall that $q^{-1} > c> \zc > 1 >q$. By a direct computation we get for $x \in (\zc, c)$
$$\SFb'(x) = \frac{q (1 - qx - \sqrt{\kappa}x+\sqrt{\kappa} q)^2}{x(1-q^2)(1-qx)(x - q)} > 0,$$
which implies (\ref{Eq.DiffS}).
\end{proof}

\begin{lemma}\label{Lem.DiffS2} Assume the notation from Definitions \ref{Def.ParametersBulk} and \ref{Def.SGEdge}. Then,
\begin{equation}\label{Eq.DiffS2}
\SFt(\zc) - \SFt(c) > 0. 
\end{equation}
\end{lemma}
\begin{proof}  By a direct computation we get for $x \in [\zc, c]$
$$\SFt'(x) = \frac{q(c-x) f(x)}{x (x- q) (1 - qx)(1 - qc)(c-q)}, $$
where 
$$f(x) = (-\kappa c + \kappa q + q^2c - q)x + (\kappa q c - q^2 \kappa - qc + 1). $$
One directly computes
$$f(c) = (c-q)^2 (\kappa_0 - \kappa) < 0,\mbox{ and } f(\zc) = \sqrt{\kappa} (1- q^2)( \zc -  c)< 0,$$
where we used that $q^{-1} > c> \zc > 1 >q$ and $\kappa > \kappa_0$. Since $f(x)$ is linear, we conclude that $f(x) < 0$ for all $x \in [\zc,c]$, and so $\SFt'(x) < 0$, implying (\ref{Eq.DiffS2}).
\end{proof}

%
\subsection{Estimates along circles}\label{Section3.2} Throughout this and the next section we denote $b = \sqrt{\kappa}$ to ease the notation. In order to not repeat ourselves too much in our arguments, we introduce the following auxiliary functions that interpolate $\SFb, \SFt$ and $\GFb, \GFt$.
\begin{definition}\label{Def.SG} Assume the same parameters as in Definition \ref{Def.ParametersBulk}, and let $p_c \in [z_c, c]$, $b = \sqrt{\kappa}$. For $z \in \mathbb{C} \setminus \{0, q, q^{-1}\}$ we define the functions
\begin{equation}\label{Eq.SFun}
S(z) = \log(1-q/z) - b^2 \log(1 - qz) - \frac{qb^2 p_c^2-(q^2b^2 +q^2)p_c+q}{(p_c-q)(1-qp_c)} \cdot \log (z).
\end{equation}
\begin{equation}\label{Eq.GFun}
G(z) = -\log(1-qz) - \frac{qp_c}{1- qp_c} \cdot \log (z).
\end{equation}
Notice that when $p_c = z_c$ we have $S(z) = \SFb(z)$, $G(z) = \GFb(z)$ as in Definition \ref{Def.SGBulk}, and when $p_c = c$ we have $S(z) = \SFt(z)$, $G(z) = \GFt(z)$ as in Definition \ref{Def.SGEdge}.
\end{definition}

The next two lemmas show that the real parts of $\SFb, \SFt$ and $\GFb, \GFt$ vary monotonically over circular contours.
\begin{lemma}\label{Lem.SmallCircleS} Assume the notation from Definitions \ref{Def.ParametersBulk}, \ref{Def.SGBulk} and \ref{Def.SGEdge}. If $R \in (0, \zc]$ and $z(\theta) = R e^{\pm \im \theta}$, then
\begin{equation}\label{Eq.SmallCircleS}
\frac{d}{d\theta} \Real[\SFb(z(\theta))] > 0 \mbox{ and } \frac{d}{d\theta} \Real[\SFt(z(\theta))] > 0 \mbox{ for } \theta \in (0, \pi).
\end{equation}
\end{lemma}
\begin{proof} Let $S(z)$ be as in Definition \ref{Def.SG}. Set $z(\theta) = Re^{\pm \im \theta}$, and note that 
\begin{equation}\label{Eq.CircularDerivativeS}
\frac{d}{d\theta} \Real[S(z(\theta))] = Rq \sin(\theta) \cdot\frac{R^2 q^2+1-2\cos(\theta)Rq-b^2\left( R^2 - 2\cos(\theta)Rq + q^2  \right)}{\left(R^2 - 2\cos(\theta)Rq + q^2\right)\left(R^2q^2+1-2\cos(\theta)Rq\right)}.
\end{equation}
We claim that for $R \in (0, \zc]$ we have
\begin{equation}\label{Eq.CosineUB}
\frac{R^2q^2+1-b^2\left( R^2 + q^2  \right)}{2Rq(1-b^2)} \geq 1.
\end{equation}
If (\ref{Eq.CosineUB}) holds, then we obtain for $\theta \in (0, \pi)$ that  
$$ \cos(\theta) < \frac{R^2q^2+1-b^2\left( R^2 + q^2  \right)}{2Rq(1-b^2)} \mbox{, and so }R^2 q^2+1-2\cos(\theta)Rq-b^2\left( R^2 - 2\cos(\theta)Rq + q^2  \right) > 0.$$
The last displayed equation, and (\ref{Eq.CircularDerivativeS}) then imply $\frac{d}{d\theta} \Real[S(z(\theta))] > 0$ for $\theta \in (0,\pi)$. Setting $p_c = \zc$ or $p_c = c$ in the last inequality gives the statements in the lemma.

To see why (\ref{Eq.CosineUB}) holds, we clear denominators and see it is equivalent to 
$$(1 -qb -Rq+Rb)(1 - Rb - qR + bq) \geq 0.$$
Since $R \in (0, \zc]$, have $1 - Rb - qR + bq \geq 0$, and so it suffices to show that for $R \in [0,\zc]$
$$f(R) \geq 0, \mbox{ where }f(R) = 1 -qb -Rq+Rb.$$
The latter holds as $f(R)$ is a linear function, $f(0) = 1 -qb > 0$, and $f(\zc) = \frac{b + q - 2q^2b}{b+q} > 0$.
\end{proof}

\begin{lemma}\label{Lem.MedCircles} Assume the notation from Definitions \ref{Def.ParametersBulk}, \ref{Def.SGBulk} and \ref{Def.SGEdge}. If $R \in (0, \infty)$ and $z(\theta) = R e^{\pm \im \theta}$, then
\begin{equation}\label{Eq.SmallCircleG}
\frac{d}{d\theta} \Real[\GFb(z(\theta))] < 0 \mbox{ and } \frac{d}{d\theta} \Real[\GFt(z(\theta))] < 0 \mbox{ for } \theta \in (0, \pi).
\end{equation}
\end{lemma}
\begin{proof} Let $G(z)$ be as in Definition \ref{Def.SG}. Set $z(\theta) = Re^{\pm \im \theta}$ and note that 
\begin{equation*}
\frac{d}{d\theta} \Real [G(z(\theta))] = -\frac{Rq \sin(\theta)}{R^2q^2 + 1 - 2Rq \cos(\theta)} < 0 \mbox{ for } \theta \in (0, \pi).
\end{equation*} 
Setting $p_c = \zc$ or $p_c = c$ in the last inequality gives the statement of the lemma.
\end{proof}

%
\subsection{Estimates along large contours}\label{Section3.3} The goal of this section is to establish the following lemma. 
\begin{lemma}\label{Lem.BigContour} Assume the notation from Definitions \ref{Def.ParametersBulk}, \ref{Def.SGBulk}, \ref{Def.ContoursBulk} and \ref{Def.SGEdge}. There exist $R_0 > q^{-1}$, $\theta_0 \in (\pi/4, \pi/2)$ and a function $\psi:(0,\infty) \rightarrow (0,\infty)$, depending on $q,\kappa,c$, such that for any $\varepsilon > 0 $ 
\begin{equation}\label{Eq.DecayBigContour}
\begin{split}
\Real[\SFb(z) - \SFb(\zc)] &\leq - \psi(\varepsilon) \mbox{ if } z \in C(\zc, \theta_0, R_0,0) \mbox{ and } |z - \zc| \geq \varepsilon, \\
\Real[\SFt(z) - \SFt(c)] &\leq - \psi(\varepsilon) \mbox{ if } z \in C(c, \theta_0, R_0,0) \mbox{ and } |z -c| \geq \varepsilon.
\end{split}
\end{equation}
\end{lemma}
The proof of the lemma is given at the end of the section after we establish several auxiliary statements. Throughout this section we continue to denote $b = \sqrt{\kappa}$.

\begin{lemma}\label{Lem.SlantedDescent} Assume the notation from Definitions \ref{Def.ParametersBulk} and \ref{Def.SG}. If $b \in [q, 1)$, then
\begin{equation}\label{QE3}
\frac{d}{dt} \Real [S(z(t))] < 0 \mbox{, where } t> 0, \hspace{2mm} z(t) = p_c + \im t.
\end{equation}
If $b \in (0,q]$, then
\begin{equation}\label{QE4}
\frac{d}{dt} \Real [S(z(t))] < 0 \mbox{, where } t> 0, \hspace{2mm} z(t) = p_c + te^{\im\pi/4}.
\end{equation}
\end{lemma}
\begin{proof} We set for convenience
\begin{equation}\label{QE4.5}
C = \frac{qb^2 p_c^2-(q^2 b^2 +q^2)p_c+q}{(p_c-q)(1-qp_c)} ,
\end{equation}
so that 
\begin{equation}\label{QE5}
\begin{split}
&S(z) = \log (1 - q/z) - b^2 \log(1 - qz) - C \log(z), \mbox{ and } S'(z) = \frac{Q(z)}{(q-z)(qz-1) z}, \\
&\mbox{ where } Q(z) = (qb^2 + Cq)z^2 + (- C - Cq^2 - q^2 - q^2b^2)z + q + Cq.
\end{split}
\end{equation}
Fix a complex parameter $u \in \mathbb{C}$, set $z(t) = p_c + u t$, and note that
\begin{equation}\label{QE6}
\begin{split}
&\frac{d}{dt} \Real [S(z(t))] \hspace{-0.5mm} = \hspace{-0.5mm} \Real [S'(z(t)) z'(t)] \hspace{-0.5mm} = \hspace{-0.5mm}\frac{\Real [u (q - p_c - \bar{u}t)(qp_c+q\bar{u}t-1) (p_c+\bar{u}t)  Q(p_c + ut) ]}{|q-p_c-ut|^2|qp_c+qut-1|^2|p_c+ut|^2}.
\end{split}
\end{equation}

{\bf \raggedleft Proof of (\ref{QE3}).} In this step we assume $b \in [q,1)$. Setting $u = \im$ in (\ref{QE6}), we get
\begin{equation}\label{QE7}
\frac{d}{dt} \Real \left[S(z(t)) \right] =  \frac{qt }{\left((q-p_c)^2 + t^2 \right)\left((qp_c-1)^2 + q^2t^2 \right)(p_c^2 + t^2)}\cdot P(t^2), 
\end{equation}
where $P(x) = a_1 x^2 + b_1 x + c_1$ with
\begin{equation}\label{QE8}
\begin{split}
&a_1 = -q(b^2 + C), \hspace{2mm} c_1 = -p_c[ b^2 (p_c - q)^2 - (1-qp_c)^2]\mbox{, }b_1 = \frac{a_2 p_c^2 + b_2 p_c + c_2}{(p_c-q)(1-qp_c)},\\
&  \mbox{ with } a_2 = q^4 + 2q^2b^2 - 2q^2 - b^2, \hspace{2mm} b_2 = 3q + qb^2 - q^3 - 3q^3 b^2, \hspace{2mm} c_2 = q^4 b^2 - 1.
\end{split}
\end{equation}
We mention that in deriving the above we used the formula for $C$ from (\ref{QE4.5}).

In view of (\ref{QE7}) and (\ref{QE8}), to prove (\ref{QE3}) it suffices to show that
\begin{equation}\label{QE9}
a_1 < 0, \hspace{2mm} b_1 < 0 \mbox{, and } c_1 \leq 0.
\end{equation}
From our parameter choice in Definition \ref{Def.ParametersBulk}, and the fact that $p_c \in [z_c, c] \subseteq [q, 1/q]$, we have $C,q > 0$, and so $a_1 < 0$. Since $p_c \geq \zc$, we have
$$p_c \geq \zc \iff p_c \geq \frac{1 + b q}{q + b}  \iff b (p_c - q) \geq 1- qp_c,$$
which implies that $c_1 \leq 0$. Finally, we observe that the quadratic polynomial $a_2x^2 + b_2x + c_2$ has discriminant
$$b_2^2 - 4a_2c_2 = (1-q^2)^2 \cdot (2b+q-qb^2-2q^2b) \cdot (2q^2b -qb^2+q-2b),$$
which is negative when $b \geq q$. Consequently, $a_2x^2 + b_2x + c_2 < 0$ for all $x \in \mathbb{R}$, and in particular $b_1 < 0$. Overall, we conclude (\ref{QE9}) and hence (\ref{QE3}).\\

{\bf \raggedleft Proof of (\ref{QE4}).} In this step we assume $b \in (0,q]$. Setting $u = e^{\im \pi/4}$ in (\ref{QE6}), we get  
\begin{equation}\label{QE10}
\begin{split}
&\frac{d}{dt} \operatorname{Re} S(z(t)) = \frac{at^2 }{|q-p_c-ut|^2|q p_c+qut-1|^2|p_c+ut|^2} \cdot \left(b^2Q_1(t) + Q_2(t) \right), \mbox{ where } \\
&Q_1(t) = -\frac{s(q^2 -  s^2 q p_c - s q t +   p_c^2 + s p_c t +  t^2)( 2 q p_c +   s q t - 1)}{1-q p_c } \\
&Q_2(t) =\frac{(s q^2 p_c^2 + 2 q^2 p_c t +  s q^2 t^2 - 2  s q p_c - 2 q t +  s)( q - 2 p_c -  s t)}{ p_c-q} 
\end{split}
\end{equation}
and we have set $s = 2^{1/2}$. In view of (\ref{QE10}), to prove (\ref{QE4}) it suffices to show
\begin{equation}\label{QE11}
Q_1(t) < 0 \mbox{ and } Q_2(t) < 0 \mbox{ for $t > 0$}.
\end{equation}

From our parameter choice in Definition \ref{Def.ParametersBulk} and the fact that $p_c \in [\zc, c]$, we have 
$$q - 2 p_c -  s t < 0, \mbox{ and }s q^2 p_c^2 + 2 q^2 p_c t +  s q^2 t^2 - 2  s q p_c - 2 q t +  s > 0.$$
Indeed, for the first one we use $p_c \geq \zc \geq 1 > q$, and for the second one we see that the quadratic polynomial in $t$ has discriminant $-4q^2(1-qp_c)^2 < 0$. Consequently, $Q_2(t) < 0$. Similarly, we have $Q_1(t) < 0$, since
$$q^2 -  s^2 q p_c - s q t +   p_c^2 + s p_c t +  t^2 > 0, \mbox{ and } 2 q p_c +   s q t - 1 > 0 .$$
The first inequality holds as the quadratic polynomial in $t$ has discriminant $-2(p_c-q)^2 < 0$. The second inequality holds as $sq = \sqrt{2} q > 0$, and
$$2qp_c \geq 2q \cdot \frac{1 + qb}{q+b} \geq 2q \cdot \frac{1 + q^2}{2q} \geq 1 + q^2 > 1,$$
where in the first inequality we used $p_c \geq \zc$, in the second we used that $b \in (0, q]$.
Overall, we conclude (\ref{QE11}) and hence  (\ref{QE4}).
\end{proof}

\begin{lemma}\label{Lem.BigCircles} Assume the notation from Definitions \ref{Def.ParametersBulk} and \ref{Def.SG}. For any $M > 0$ we can find $\mathsf{R} > q^{-1}$, depending on $q,b,p_c,M$, such that for $R \geq \mathsf{R}$ and $\theta \in [-\pi, \pi]$
\begin{equation}\label{RE1}
 \Real [S(Re^{\mathbf{i}\theta}) - S(p_c) ]\leq -M.
\end{equation}
\end{lemma}
\begin{proof} Since $S(\bar{z}) = \overline{S(z)}$, it suffices to prove (\ref{RE1}) for $\theta \in [0, \pi]$, which we assume in the sequel. Fix $R > q^{-1}$ and set $z(\theta) = Re^{\mathbf{i}\theta}$. Using (\ref{Eq.CircularDerivativeS}), and that $\sin(\theta) > 0$ for $\theta \in (0, \pi)$, we see that 
\begin{equation}\label{RE2}
\frac{d}{d\theta} \Real[S(z(\theta))] < 0  \mbox{ if and only if } \cos(\theta)>\frac{R^2q^2+1-b^2\left( R^2 + q^2  \right)}{2 Rq(1-b^2)}.
\end{equation}
The latter implies that on $[0,\pi]$ the function $\Real[S(z(\theta))]$ either monotonically decreases, monotonically increases, or it first decreases and then increases. In particular, the maximum of $\Real[S(z(\theta))]$ on $[0,\pi]$ is achieved either at $\theta = 0$, or $\theta = \pi$ (or both). We conclude that 
\begin{equation*}
\max_{\theta \in [0, \pi]} \Real[S(z(\theta))] \leq \max \{S(R), S(-R)\}.
\end{equation*}
The last equation implies the statement of the lemma once we note that 
$$\lim_{R \rightarrow \infty} \max \{S(R), S(-R)\} = - \infty.$$
To see the latter, we can use the definition of $S(z)$ in (\ref{Eq.SFun}), which gives as $R \rightarrow \infty$
$$\Real[S( \pm R)] \sim \left( - b^2 - \frac{qb^2 p_c^2-(q^2b^2 +q^2)p_c+q}{(p_c-q)(1-qp_c)}  \right) \log(R) = - \frac{q(1-q p_c) + b^2 (p_c - q) }{(p_c-q)(1-p_cq)} \cdot \log(R).$$
\end{proof}

With the above results in place we turn to the proof of Lemma \ref{Lem.BigContour}.
\begin{proof}[Proof of Lemma \ref{Lem.BigContour}] Let $S(z)$ be as in Definition \ref{Def.SG}. We will show that if $p_c \in \{\zc, c\}$, then we can find $R_0, \theta_0, \psi$ as in the statement of the lemma, such that for $z \in C(p_c, \theta_0, R_0, 0)$ with $|z - p_c| \geq \varepsilon$
\begin{equation}\label{Eq.Psi1}
\Real[S(z) - S(p_c)] \leq - \psi(\varepsilon).
\end{equation}
The latter statement clearly implies the statement of the lemma.

Before we go into the proof of (\ref{Eq.Psi1}), let us give a brief outline of the argument. From Lemma \ref{Lem.BigCircles} with $M = 1$ we can find $R_0$ sufficiently large, so that for $p_c \in \{\zc,c\}$ and $|z| = R_0$
\begin{equation}\label{Eq.Psi2}
\Real[S(z) - S(p_c)] \leq - 1.
\end{equation}
This specifies our choice of $R_0$. The choice of $\theta_0 \in (\pi/4, \pi/2)$ is different depending on whether $b \in [q,1)$ or $b \in (0,q)$. From Lemma \ref{Lem.SlantedDescent} we have that that $\Real[S(z(t)) - S(p_c)]$ decreases in $ t \geq 0$ when $b \in [q,1)$ for $z(t) = p_c + t e^{\im \pi/2}$, and when $b \in (0,q)$ for $z(t) = p_c + te^{\im \pi /4}$. Consequently, our choice for $\theta_0$ will be very close to $\pi/2$ when $b \in [q,1)$, and very close to $\pi/4$ when $b \in (0,q)$. The precise choice of $\theta_0$ is detailed in the first step below. In the second step we construct the function $\psi$, which is a small positive constant away from the origin, and near the origin takes a more complicated (implicit) form, related to Lemma \ref{Lem.SlantedDescent}. In the third step we show that our choice of $\psi$ satisfies (\ref{Eq.Psi1}). This will be a consequence of our choice of $\theta_0$ and (\ref{Eq.Psi2}) when $z \in C(p_c, \theta_0, R_0,0)$ is far away from $p_c$. To show that $\psi$ satisfies (\ref{Eq.Psi1}) even when $z$ is close to $p_c$, we will use the fact that for $\theta_0 \in [\theta_1, \theta_2]$, we can control $\Real[S(p_c + r e^{\im \theta_0})]$ by quantities of the form $\Real[S(p_c + r_1 e^{\im \theta_1})]$, and $\Real[S(p_c + r_2 e^{\im \theta_2})]$, where $r \asymp r_1$ and $r \asymp r_2$. When $b \in [q,1)$, we will use this statement with $\theta_1 = \pi/3, \theta_2 = \pi/2$, and when $b \in (0,q)$, we will use it with $\theta_1 = \pi/4, \theta_2 = \pi/3$. The bound for the angle $\pi/3$ will follow from Lemma \ref{Lem.DecayNearCritTheta}, and the bound for the other angle will come from Lemma \ref{Lem.SlantedDescent}. We now turn to the details of the proof.\\

Throughout the proof all the constants depend on $q,b,c$ -- we do not list this dependence explicitly.\\

{\bf \raggedleft Step 1.} From Lemma \ref{Lem.BigCircles} with $M = 1$ we can find $R_0 \geq 1 + q^{-1}$, so that (\ref{Eq.Psi2}) holds for $p_c \in \{\zc,c\}$ and $|z-p_c| = R_0$. This specifies our choice of $R_0$ for the rest of the proof. In the rest of this step we specify our choice of $\theta_0$, which depends on whether $b \in [q,1)$ or $b \in (0,q)$. 

From Lemma \ref{Lem.DecayNearCritTheta} we can find $\delta_1 \in (0,1)$ and $\epsilon_1 > 0$, such that for $r \in [0, \delta_1]$ and $p_c \in \{\zc, c\}$ 
\begin{equation}\label{Eq.Psi3}
\Real[S(p_c + re^{\pm \im \pi/3}) - S(p_c)] \leq - \epsilon_1 \cdot r^3.
\end{equation}
We fix $\delta_2 \in (0, \delta_1]$ small enough so that $2q^{-1} \delta_2 + \delta_2^2 \leq \delta_1^2$.

If $b \in [q,1)$, we conclude from Lemma \ref{Lem.SlantedDescent}, and the fact that $\overline{S(z)} = S(\bar{z})$, that for all $t > 0$ 
$$\Real [S(p_c + t e^{\pm \im \pi /2 }) - S(p_c)] < 0.$$
Using the latter and the uniform continuity of $\Real[S(z)]$ over the compact region $D = \{z \in \mathbb{C}: \delta_2 \leq |z-p_c| \leq R_0 \mbox{ and } \operatorname{Arg}(z-p_c) \in [\pi/4, \pi/2] \cup [-\pi/2, -\pi/4] \}$, we conclude that we can find a small enough $\epsilon \in (0,1)$, and $\theta_0 \in [\pi/3, \pi/2)$ sufficiently close to $\pi/2$, so that if $z = p_c + t e^{\pm \im \theta_0}$ with $t > 0$, and $z \in D$
\begin{equation}\label{Eq.Psi4}
\Real[S(z ) - S(p_c)] \leq - \epsilon.
\end{equation}
If $b \in (0,q)$, we similarly conclude from Lemma \ref{Lem.SlantedDescent} that for all $t > 0$ 
$$\Real [S(p_c + t e^{\pm \im \pi /4 }) - S(p_c)] < 0,$$
and so we can find $\epsilon \in (0,1)$, and $\theta_0 \in (\pi/4, \pi/3]$ sufficiently close to $\pi/4$, so that if $z = p_c + t e^{\pm \im\theta_0}$ with $t > 0$, and $z \in D$, then (\ref{Eq.Psi4}) holds. The latter specifies our choice of $\theta_0$ for the rest of the proof.\\

{\bf \raggedleft Step 2.} In this step we specify our choice of the function $\psi$, which will also depend on whether $b \in [q,1)$ or $b \in (0,q)$. When $\varepsilon \geq \delta_2$, we set 
\begin{equation}\label{Eq.Psi5}
\psi(\varepsilon) =  \epsilon,
\end{equation}
where we recall that $\epsilon$ is the implicit constant satisfying (\ref{Eq.Psi4}). Note that $\epsilon$ is different depending on whether $b \in [q,1)$ or $b \in (0,q)$.

When $b \in [q,1)$, we have from Lemma \ref{Lem.SlantedDescent} that there is a function $\psi_{\pi/2}:(0, \infty) \rightarrow (0, \infty)$, such that for $r \geq \varepsilon > 0$
\begin{equation}\label{Eq.Psi6}
\Real[S(p_c + re^{\pm \im \pi/2}) - S(p_c)] \leq - \psi_{\pi/2}(\varepsilon).
\end{equation}
Using the above implicit function, we define for $\varepsilon \in (0,\delta_2)$
\begin{equation}\label{Eq.Psi7}
\psi(\varepsilon) =  \min \left( \epsilon_1 q^3\varepsilon^6/(1+q)^3, \psi_{\pi/2}(\varepsilon), \epsilon \right).
\end{equation}
When $b \in (0,q)$, Lemma \ref{Lem.SlantedDescent} implies that there is a function $\psi_{\pi/4}:(0, \infty) \rightarrow (0, \infty)$, such that for $r \geq \varepsilon > 0$ 
\begin{equation}\label{Eq.Psi8}
\Real[S(p_c + re^{\pm \im \pi/4}) - S(p_c)] \leq - \psi_{\pi/4}(\varepsilon).
\end{equation}
As before, we define for $\varepsilon \in (0,\delta_2)$
\begin{equation}\label{Eq.Psi9}
\psi(\varepsilon) =  \min \left( \epsilon_1 \varepsilon^3, \psi_{\pi/4}(\varepsilon/2), \epsilon \right).
\end{equation}
Equations (\ref{Eq.Psi5}) and (\ref{Eq.Psi7}) define $\psi$ when $b \in [q,1)$, and (\ref{Eq.Psi5}) and (\ref{Eq.Psi9}) define $\psi$ when $b \in [q,1)$ for the rest of the proof.\\

{\bf \raggedleft Step 3.} In this final step we show that the $\psi$ from Step 2 satisfies (\ref{Eq.Psi1}). Suppose first that $\varepsilon \geq \delta_2$ and $|z - p_c| \geq \varepsilon$. If $|z| = R_0$ (i.e. $z$ is in the circular arc of the contour $C(p_c, \theta_0, R_0,0)$), we have
$$\Real[S(z) - S(p_c)] \leq - 1 \leq -\epsilon = -\psi (\varepsilon),$$
where the first inequality used (\ref{Eq.Psi2}), the second used that $\epsilon \in (0,1)$ and the third used (\ref{Eq.Psi5}). If instead $|z| < R_0$, then $z \in D$ as in Step 1, and so from (\ref{Eq.Psi4}) we conclude
$$\Real[S(z) - S(p_c)] \leq -\epsilon = -\psi (\varepsilon).$$
The last two displayed equations verify (\ref{Eq.Psi1}) when $\varepsilon \geq \delta_2$.

In the remainder we fix $\varepsilon \in (0, \delta_2)$ and proceed to verify (\ref{Eq.Psi1}) for $|z - p_c| \geq \varepsilon$. The argument from the previous paragraph implies for $|z - p_c| > \delta_2$ that
$$\Real[S(z) - S(p_c)] \leq -\epsilon \leq - \psi(\varepsilon),$$
where in the last inequality we used (\ref{Eq.Psi7}) and (\ref{Eq.Psi9}). Consequently, we may assume that $|z - p_c| \in [\varepsilon, \delta_2]$. As  $R_0 \geq q^{-1} + 1 > p_c + \delta_2 \geq |z|$, we see that $z$ is on one of the straight segments in the contour $C(p_c, \theta_0, R_0,0)$. Without loss of generality (using that $\overline{S(z)} = S(\bar{z})$), we may assume that $z = p_c + r e^{ \im \theta_0}$ for some $r \in [\varepsilon, \delta_2]$.\\

Suppose first that $b \in [q, 1)$. Let $\theta_1 = \pi/3$, $\theta_2 = \pi/2$, and note that $\theta_0 \in [\theta_1, \theta_2]$ by construction. For $i \in \{1,2\}$ we let $\zeta_i$ be the point where the ray $\{p_c + t e^{\im \theta_i}: t \geq 0\}$ intersects the zero-centered circle of radius $R = |z|$. Then, we have $|z| = R e^{\im \phi_0}$, $\zeta_i = R e^{\im \phi_i}$ with $\phi_0 \in [\phi_1, \phi_2]$, see Figure \ref{Fig.S3}. Also, we have $\zeta_i = p_c + r_i e^{\im \theta_i}$. From the argument after (\ref{RE2}) we have that over $\phi \in (0,\pi)$ the function $\Real[S(R e^{\im \phi})]$ is either increasing, decreasing, or first decreases and then increases. Consequently,
\begin{equation}\label{Eq.Psi10}
\Real[S(z)] \leq \max \left( \Real[S(\zeta_1)] , \Real[S(\zeta_2)]  \right) = \max \left( \Real[S(p_c + r_1 e^{\im \theta_1})], \Real[S(p_c + r_2 e^{\im \theta_2})]  \right).
\end{equation}
We also observe the following inequalities
\begin{equation}\label{Eq.QW1}
r_2 \geq r \mbox{ and } \delta_2 \geq r \geq r_1 \geq qr^2/(1+q).
\end{equation}
Indeed, from the law of cosines we have 
$$r_2^2 = r^2 + 2 p_c r \cos(\theta_0),$$
which implies the first inequality in (\ref{Eq.QW1}). Similarly, from the law of cosines we have 
$$r^2 + 2p_c r \cos(\theta_0) = r_1^2 + p_c r_1,$$
and since $\cos(\theta_0) \in [0,1/2]$ we conclude
$$r \geq r_1 \mbox{ and } r_1 (1 + q^{-1}) \geq r_1^2 + p_c r_1 \geq r^2.$$
The last displayed equation implies the second set of inequalities in (\ref{Eq.QW1}). 

\begin{figure}[h]
    \centering
     \begin{tikzpicture}[scale=1]
     \useasboundingbox (-2,-0.5) rectangle (20,4.5); 
        \def\tra{0} 

        \draw[->, thick, gray] (\tra,0)--(\tra + 6,0) node[right]{$\Real$};
        \draw[->, thick, gray] (\tra ,0)--( \tra ,4) node[above]{$\Imag$};
  \def\x{\tra}   
  \def\R{5}   
  \def\pct{4}   
  \def\tho{60} 
  \def\tht{90} 

  \node (O) at (\x,0) {};
  \node (X) at (\tra + \pct,0) {};
  
  \path (X) ++(\tho:1) coordinate (Xp);
  \path (X) ++(\tht:1) coordinate (Xm);

  \path[name path=BigCirc]   (O) circle[radius=\R];
  \path[name path=RayOne]   (X) -- ($(X) +10*(\tho:1)$);
  \path[name path=RayTwo]  (X) -- ($(X) +10*(\tht:1)$);

  \path[name intersections={of=RayOne and BigCirc,by=zetaOne}];
  \path[name intersections={of=RayTwo and BigCirc,by=zetaTwo}];

  \draw[-,gray] (\x,0) -- (\x + \R, 0);
  \draw[-, thick] ($(X)$) -- (zetaOne);
  \draw[-, thick] ($(X)$) -- (zetaTwo);
  \draw[-, thin] (\x,0) -- (zetaOne);
  \draw[-, thin] (\x,0) -- (zetaTwo);
  \pgfmathsetmacro{\xR}{\x + \R}

  \draw let
      \p1 = ($(zetaTwo) - (\tra,0)$),
      \n1 = {atan2(\y1,\x1)},                    
    in
      [-, dashed, gray] (\xR,0)  arc[start angle=0, end angle=\n1, radius=\R];   

   \fill (\x + \R,0)     circle (1.5pt) node[below = 2pt]  {$R$};
   \fill (\x,0)     circle (1.5pt) node[below = 2pt]  {$0$};
   \fill (X)     circle (1.5pt) node[below = 2pt]  {$p_c$};
   \fill (zetaOne)   circle (1.5pt) node[right]  {$\zeta_1$};
   \fill (zetaTwo)   circle (1.5pt) node[above]  {$\zeta_2$};
 
   \pgfmathsetmacro{\xR}{\x + 2}

  \draw let
      \p1 = ($(zetaOne) - (\tra,0)$),
      \n1 = {atan2(\y1,\x1)},                    
    in
      [->] (\xR,0) arc[start angle=0, end angle=\n1, radius=2];
  \node at (\xR+0.2,0.3) {$\phi_1$};

  \pgfmathsetmacro{\xR}{\x + 2.5}    
  \draw let
      \p1 = ($(zetaTwo) - (\tra,0)$),
      \n1 = {atan2(\y1,\x1)},                    
    in
      [->] (\xR,0) arc[start angle=0, end angle=\n1, radius=2.5];   
    \node at (\xR,1.1) {$\phi_2$};
    
    \pgfmathsetmacro{\xR}{\x + \R - 0.5}    
    \draw[->] (\xR,0) arc[start angle=0, end angle= 60, radius=0.5]; 
    \node at (\xR + 0.2,0.3) {$\theta_1$};

\path let
    \p1 =  ($(zetaOne) - (\tra,0)$),
    \p2 =  ($(zetaTwo) - (\tra,0)$),
    \n1 = {atan2(\y1,\x1)},
    \n2 = {atan2(\y2,\x2)},
    \n3 = {ifthenelse(\n2<\n1, \n2+360, \n2)}, 
    \n4 = {(\n1+\n3)/2}                        
  in
    coordinate (M) at (\n4:\R); 

  \fill ($(M) + (\tra,0)$) circle (2pt) node[right, yshift = 1pt] {$z$};


\def\tra{8} 

        \draw[->, thick, gray] (\tra,0)--(\tra + 6,0) node[right]{$\Real$};
        \draw[->, thick, gray] (\tra ,0)--( \tra ,4) node[above]{$\Imag$};
  \def\x{\tra}   
  \def\R{5}   
  \def\pct{4}   
  \def\tho{45} 
  \def\tht{60} 

  \node (O) at (\x,0) {};
  \node (X) at (\tra + \pct,0) {};
  
  \path (X) ++(\tho:1) coordinate (Xp);
  \path (X) ++(\tht:1) coordinate (Xm);

  \path[name path=BigCirc]   (O) circle[radius=\R];
  \path[name path=RayOne]   (X) -- ($(X) +10*(\tho:1)$);
  \path[name path=RayTwo]  (X) -- ($(X) +10*(\tht:1)$);

  \path[name intersections={of=RayOne and BigCirc,by=zetaOne}];
  \path[name intersections={of=RayTwo and BigCirc,by=zetaTwo}];

  \draw[-,gray] (\x,0) -- (\x + \R, 0);
  \draw[-, thick] ($(X)$) -- (zetaOne);
  \draw[-, thick] ($(X)$) -- (zetaTwo);
  \draw[-, thin] (\x,0) -- (zetaOne);
  \draw[-, thin] (\x,0) -- (zetaTwo);
  \pgfmathsetmacro{\xR}{\x + \R}

  \draw let
      \p1 = ($(zetaTwo) - (\tra,0)$),
      \n1 = {atan2(\y1,\x1)},                    
    in
      [-, dashed, gray] (\xR,0)  arc[start angle=0, end angle=\n1, radius=\R];   

   \fill (\x + \R,0)     circle (1.5pt) node[below = 2pt]  {$R$};
   \fill (\x,0)     circle (1.5pt) node[below = 2pt]  {$0$};
   \fill (X)     circle (1.5pt) node[below = 2pt]  {$p_c$};
   \fill (zetaOne)   circle (1.5pt) node[right]  {$\zeta_1$};
   \fill (zetaTwo)   circle (1.5pt) node[above]  {$\zeta_2$};
 
   \pgfmathsetmacro{\xR}{\x + 2.5}

  \draw let
      \p1 = ($(zetaOne) - (\tra,0)$),
      \n1 = {atan2(\y1,\x1)},                    
    in
      [->] (\xR,0) arc[start angle=0, end angle=\n1, radius=2.5];
  \node at (\xR+0.22,0.23) {$\phi_1$};

  \pgfmathsetmacro{\xR}{\x + 3.5}    
  \draw let
      \p1 = ($(zetaTwo) - (\tra,0)$),
      \n1 = {atan2(\y1,\x1)},                    
    in
      [->] (\xR,0) arc[start angle=0, end angle=\n1, radius=3.5];   
    \node at (\xR + 0.22,0.4) {$\phi_2$};
    
    \pgfmathsetmacro{\xR}{\x + \R - 0.5}    
    \draw[->] (\xR,0) arc[start angle=0, end angle= 45, radius=0.5]; 
    \node at (\xR + 0.2,0.2) {$\theta_1$};

\path let
    \p1 =  ($(zetaOne) - (\tra,0)$),
    \p2 =  ($(zetaTwo) - (\tra,0)$),
    \n1 = {atan2(\y1,\x1)},
    \n2 = {atan2(\y2,\x2)},
    \n3 = {ifthenelse(\n2<\n1, \n2+360, \n2)}, 
    \n4 = {(\n1+\n3)/2}                        
  in
    coordinate (M) at (\n4:\R); 

  \fill ($(M) + (\tra,0)$) circle (2pt) node[right, yshift = 4pt] {$z$};
    
    \end{tikzpicture} 
    \caption{The figure depicts the points $\zeta_1, \zeta_2$, and the angles $\phi_1, \phi_2$ when $\theta_1 = \pi/3$ and $\theta_2 = \pi/2$ (left), and when $\theta_1 = \pi/4$ and $\theta_2 = \pi/3$ (right).}
    \label{Fig.S3}
\end{figure}
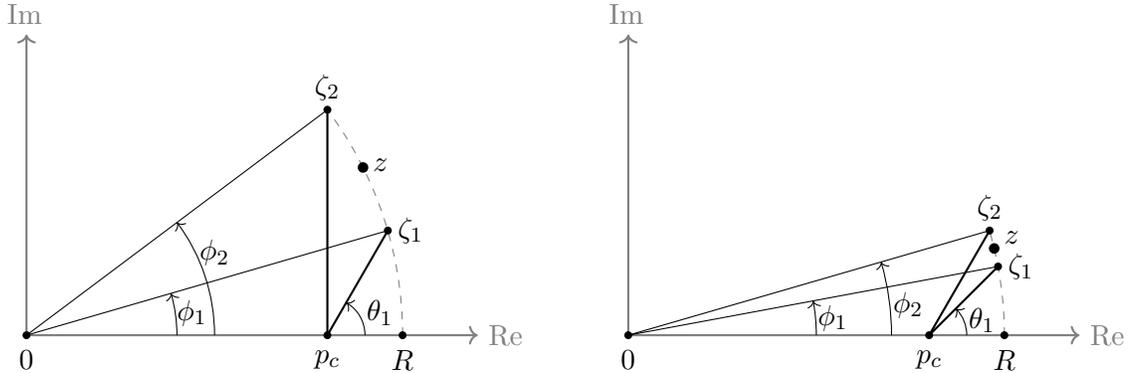

From the second set of inequalities in (\ref{Eq.QW1}) and (\ref{Eq.Psi3}) we conclude
\begin{equation}\label{Eq.QW2}
 \Real[S(p_c + r_1 e^{\im \theta_1})] \leq - \epsilon_1 \cdot r_1^3 \leq -\epsilon_1 q^3 r^6/(1+q)^3 \leq - \epsilon_1 q^3 \varepsilon^6/(1+q)^3.
\end{equation}
From the first inequality in (\ref{Eq.QW1}) and the definition of $\psi_{\pi/2}$ we conclude
\begin{equation}\label{Eq.QW3}
 \Real[S(p_c + r_2 e^{\im \theta_2})] \leq - \psi_{\pi/2}(\varepsilon).
\end{equation}
Combining (\ref{Eq.Psi10}), (\ref{Eq.QW2}) and (\ref{Eq.QW3}) with the definition of $\psi$ in (\ref{Eq.Psi7}) establishes (\ref{Eq.Psi1}) when $\varepsilon \in (0, \delta_2)$ and $b \in [q,1)$.\\

Suppose finally that $b \in (0,q)$. We now set $\theta_1 = \pi/4$, $\theta_2 = \pi/3$ and note that $\theta_0 \in [\theta_1, \theta_2]$ by construction. We define $\zeta_1, \zeta_2, r_1, r_2, \phi_1, \phi_2$ as in the previous two paragraphs (for the new angles $\theta_1, \theta_2$) and note that in place of (\ref{Eq.QW1}) we have
\begin{equation}\label{Eq.QW4}
\delta_1 \geq r_2 \geq r \mbox{ and }  r_1 \geq r/2.
\end{equation}
Indeed, from the law of cosines and our choice of $\delta_2$ we have 
$$r_2^2 + p_c r_2 = r^2 + 2 p_c r \cos(\theta_0) \leq \delta_2^2 + 2q^{-1} \delta_2 \leq \delta_1^2,$$
implying the first set of inequalities in (\ref{Eq.QW4}) as $\cos(\theta_0) \in [1/2,\sqrt{2}/2]$. Another law of cosines gives
$$r^2 + 2p_c r \cos(\theta_0) = r_1^2 + \sqrt{2} p_c r_1,$$
which implies the last inequality in (\ref{Eq.QW4}) once we use $\cos(\theta_0) \in [1/2,\sqrt{2}/2]$.

From the last inequality in (\ref{Eq.QW4}) and the definition of $\psi_{\pi/4}$ we conclude
\begin{equation}\label{Eq.QW5}
 \Real[S(p_c + r_1 e^{\im \theta_1})] \leq - \psi_{\pi/4}(\varepsilon/2).
\end{equation}
From the first set of inequalities in (\ref{Eq.QW4}) and (\ref{Eq.Psi3}) we conclude
\begin{equation}\label{Eq.QW6}
 \Real[S(p_c + r_2 e^{\im \theta_2})] \leq - \epsilon_1 \cdot r_2^3 \leq - \epsilon_1 \cdot r^3 \leq -\epsilon_1 \cdot \varepsilon^3.
\end{equation}
Combining (\ref{Eq.Psi10}), (\ref{Eq.QW5}) and (\ref{Eq.QW6}) with the definition of $\psi$ in (\ref{Eq.Psi9}) establishes (\ref{Eq.Psi1}) when $\varepsilon \in (0, \delta_2)$ and $b \in (0,q)$. This suffices for the proof.
\end{proof}

%
\section{Kernel convergence for the bottom curves}\label{Section4} The goal of this section is to establish the following statement.

\begin{proposition}\label{Prop.KernelConvBottom} Assume the same notation as in Lemma \ref{Lem.PrelimitKernelBulk} with $\theta = \theta_0$, $R = R_0$ as in Lemma \ref{Lem.BigContour}. If $x_N, y_N \in \mathbb{R}$ are sequences such that $\lim_{N \rightarrow \infty} x_N = x$, $\lim_{N \rightarrow \infty} y_N = y$, then for any $s,t \in \mathcal{T}$
\begin{equation}\label{Eq.KernelLimitBottom}
\begin{split}
&\lim_{N \rightarrow \infty} K_{11}^N(s,x_N;t,y_N) = 0, \hspace{2mm} \lim_{N \rightarrow \infty} K_{22}^N(s,x_N;t,y_N) =  0,\\
& \lim_{N \rightarrow \infty} \hspace{-1mm}K^N_{12}(s,x_N; t,y_N) = e^{2\fq^3s^3/3 - 2\fq^3 t^3/3 + \fq s x - \fq ty} K^{\mathrm{Airy}}\left(-\fq s,  x + \fq^2 s^2 ; - \fq t, y + \fq^2 t^2 \right),
\end{split}
\end{equation}
where $K^{\mathrm{Airy}}$ is the extended Airy kernel from (\ref{Eq.S1AiryKer}). 
\end{proposition}

The proof of Proposition \ref{Prop.KernelConvBottom} is given in Section \ref{Section4.2}. In Section \ref{Section4.1} we derive suitable estimates for the functions that appear in the kernel $K^N$ in Lemma \ref{Lem.PrelimitKernelBulk} along the contours $\Gamma_N, \gamma_N$ and $\tilde{\gamma}_N$.

%
\subsection{Function bounds}\label{Section4.1} 

In what follows we fix parameters as in Definition \ref{Def.ParametersBulk} and $\theta_0, R_0$ as in Lemma \ref{Lem.BigContour}. In addition, we work with the contours $\Gamma_N, \gamma_N, \tilde{\gamma}_N$ as in (\ref{Eq.PrelimitContours}) with $\theta = \theta_0$, $R = R_0$. We also assume that $x, y \in [-L,L]$ for a fixed $L >0$. In the inequalities below we will encounter various constants $A_i,a_i > 0$ with $A_i$ sufficiently large, and $a_i$ sufficiently small, depending on $q, \kappa, c,t_1, \dots, t_m, \theta_0, R_0, L$ -- we do not list this dependence explicitly. In addition, the inequalities will hold provided that $N$ is sufficiently large, depending on the same set of parameters, which we will also not mention further.  

Let $\delta_1(\theta), \epsilon_1(\theta)$ be as in Lemma \ref{Lem.DecayNearCritTheta}, $\delta_2, \epsilon_2$ be as in Lemma \ref{Lem.DecayNearCritGen} and set $\delta = \min(\delta_1(\theta_0), \delta_2), \epsilon = \min(\epsilon_1(\theta_0), \epsilon_2)$. If $z \in \Gamma_N$ and $|z - \zc| \leq \delta$, we have from Lemmas \ref{Lem.PowerSeriesSG} and \ref{Lem.DecayNearCritTheta} that
\begin{equation}\label{Eq.S1BoundZClose}
\Real[\SFb(z) - \SFb(\zc)] \leq - \epsilon |z-\zc|^3 + [\sigmaq^3/3 + \epsilon] \sec^3(\theta_0) N^{-1} + C_0 \sec^4(\theta_0)N^{-4/3}.
\end{equation}
If $z \in \Gamma_N$ and $|z - \zc| \geq \delta \geq \sec(\theta_0) N^{-1/3}$, we have from Lemma \ref{Lem.BigContour} that
\begin{equation}\label{Eq.S1BoundZFar}
\Real[\SFb(z) - \SFb(\zc)] \leq - \psi(\delta).
\end{equation}
If $w \in \gamma_N$ and $|w - \zc| \leq N^{-1/12} \leq \delta$, we have from Lemma \ref{Lem.DecayNearCritGen} that
\begin{equation}\label{Eq.S1BoundWClose}
\Real[\SFb(w) - \SFb(\zc)] \geq  \epsilon |w-\zc|^3. 
\end{equation}
If $w \in \gamma_N$ and $|w - \zc| \geq N^{-1/12}$, we have from Lemma \ref{Lem.SmallCircleS} that
\begin{equation}\label{Eq.S1BoundWFar}
\begin{split}
&\Real[\SFb(w) - \SFb(\zc)] \geq \Real[\SFb(\zc - e^{\pm \im 2\pi/3} N^{-1/12}) - \SFb(\zc)] \geq  \epsilon N^{-1/4}, 
\end{split}
\end{equation}
where in the last inequality we used (\ref{Eq.S1BoundWClose}). 
From Lemma \ref{Lem.PowerSeriesSG} we have for $z \in \Gamma_N \cup \gamma_N$ with $|z - \zc| \leq \delta$ 
\begin{equation}\label{Eq.G1BoundClose}
|\GFb(z) - \GFb(\zc)| \leq  \fq \sigmaq^2 |z-\zc|^2 + C_0 |z-\zc|^3.
\end{equation}
By the boundedness of the contours $\Gamma_N, \gamma_N$, we can find $A_0 > 0$, such that for $z \in \Gamma_N \cup \gamma_N$
\begin{equation}\label{Eq.G1BoundFar}
|\GFb(z) - \GFb(\zc)| \leq  A_0.
\end{equation}
By Taylor expanding the logarithm we can find $A_1 > 0$, such that for $z \in \Gamma_N \cup \gamma_N \cup \tilde{\gamma}_N$
\begin{equation}\label{Eq.G1BoundZLog}
|\log(z/\zc) | \leq  A_1 | z- \zc|.
\end{equation}

From Lemma \ref{Lem.DecayNearCritGen} we have for $z \in \tilde{\gamma}_N$ and $|z- \zc| \leq N^{-1/12} \leq \delta$
\begin{equation}\label{Eq.G1BoundZClose}
\Real[\GFb(z) - \GFb(\zc)] \leq - \epsilon |z-\zc|^2. 
\end{equation}
If $z \in \tilde{\gamma}_N$ and $|z - \zc| \geq N^{-1/12}$, we have from Lemma \ref{Lem.MedCircles} that
\begin{equation}\label{Eq.G1BoundZFar}
\begin{split}
&\Real[\GFb(z) - \GFb(\zc)] \leq \Real[\GFb(\zc + e^{\pm \im \pi/2} N^{-1/12}) - \GFb(\zc)] \leq -\epsilon N^{-1/6},
\end{split}
\end{equation}
where in the last inequality we used (\ref{Eq.G1BoundZClose}).\\

We now proceed to find suitable estimates for the functions $F^N_{ij}$ and $H^N_{ij}$ from (\ref{Eq.DefIN11Bulk}), (\ref{Eq.DefIN12Bulk}) and (\ref{Eq.DefIN22Bulk}). By combining (\ref{Eq.S1BoundZClose}), (\ref{Eq.S1BoundZFar}), (\ref{Eq.G1BoundFar}) and (\ref{Eq.G1BoundZLog}), we conclude that for some $A_2, a_2 > 0$ and all $z,w \in \Gamma_N$ we have 
\begin{equation}\label{Eq.F11Bound}
\begin{split}
&\left| F^N_{11}(z,w)  \right| \leq \exp \left( - \psi(\delta)N/2  \right) \mbox{ if } \max(|z-\zc|, |w-\zc|) \geq \delta, \\
&\left| F^N_{11}(z,w)  \right| \leq \exp \left( - a_2 N (|z-\zc|^3 + |w-\zc|^3) + A_2 N^{2/3} ( |z-\zc|^2 + |w-\zc|^2) + A_2  \right), \\
& \mbox{ if } \max(|z-\zc|, |w-\zc|) \leq \delta.
\end{split}
\end{equation}
By combining (\ref{Eq.S1BoundZClose}-\ref{Eq.G1BoundZLog}), we conclude that for some $A_3, a_3 > 0$ and all $z \in \Gamma_N$, $w \in \gamma_N$ we have 
\begin{equation}\label{Eq.F12Bound}
\begin{split}
&\left| F^N_{12}(z,w)  \right| \leq \exp \left( - (\epsilon/2)N^{3/4}  \right) \mbox{ if } |z-\zc| \geq \delta \mbox{ or }  |w-\zc| \geq N^{-1/12}, \\
&\left| F^N_{12}(z,w)  \right| \leq \exp \left( - a_3 N (|z-\zc|^3 + |w-\zc|^3) + A_3N^{2/3} (|z-\zc|^2 + |w-\zc|^2) + A_3   \right), \\
& \mbox{ if } |z-\zc| \leq \delta \mbox{ and } |w-\zc| \leq N^{-1/12}.
\end{split}
\end{equation}
By combining (\ref{Eq.S1BoundWClose}-\ref{Eq.G1BoundZLog}), we conclude that for some $A_4, a_4 > 0$ and all $z,w \in \gamma_N$ we have 
\begin{equation}\label{Eq.F22Bound}
\begin{split}
&\left| F^N_{22}(z,w)  \right| \leq \exp \left( - (\epsilon/2)N^{3/4}  \right) \mbox{ if } \max(|z-\zc|, |w-\zc|) \geq N^{-1/12}, \\
&\left| F^N_{22}(z,w)  \right| \leq \exp \left( - a_4 N (|z-\zc|^3+ |w-\zc|^3) + A_4 N^{2/3} (|z-\zc|^2 + |w-\zc|^2 ) + A_4  \right), \\
& \mbox{ if } z,w \in \max(|z-\zc|, |w-\zc|) \leq N^{-1/12}.
\end{split}
\end{equation}
As $H^N_{ij}$ are essentially rational functions, we have for some $A_5 > 0$ that
\begin{equation}\label{Eq.HijBound}
\begin{split}
&\left| H^N_{11}(z,w)  \right| \leq A_5 N^{1/3}  \mbox{ if } z, w \in \Gamma_N, \\
&\left| H^N_{12}(z,w)  \right| \leq A_5 N^{2/3}  \mbox{ if } z \in \Gamma_N \mbox{ and } w \in \gamma_N,\\
&\left| H^N_{22}(z,w)  \right| \leq A_5 N^{1/3}  \mbox{ if } z,w \in \gamma_N,
\end{split}
\end{equation}
and also 
\begin{equation}\label{Eq.HijBound2}
\begin{split}
&\left| \sigmaq \zc N^{1/3} z^{-1} \right| \leq A_5 N^{1/3}  \mbox{ if } z \in \tilde{\gamma}_N, \\
&\left| \sigmaq \zc N^{1/3}  \cdot \frac{(zc-1) (1-qz)^{\kappa N - \lfloor \kappa N \rfloor}}{z(z^2-1)(1-qc)^{\kappa N - \lfloor \kappa N \rfloor}} \right| \leq A_5 N^{1/3}  \mbox{ if } z \in \Gamma_N,\\
&\left| \sigmaq \zc N^{1/3} \cdot \frac{1}{(c  w - 1)(1-qc)^{\kappa N - \lfloor \kappa N \rfloor}(1-qw)^{\kappa N - \lfloor \kappa N \rfloor}} \right| \leq A_5 N^{1/3}  \mbox{ if } z \in \gamma_N.
\end{split}
\end{equation}
We mention that the extra $N^{1/3}$ factor in the second line of (\ref{Eq.HijBound}) comes from the $z-w$ term in the denominator of $H^N_{12}$, for which we have $|z-w| \geq N^{-1/3}$ from the way $\Gamma_N, \gamma_N$ are defined. \\

We next find suitable bounds for the functions in (\ref{Eq.DefRN12Bulk}) and (\ref{Eq.DefRN22Bulk}). From (\ref{Eq.S1BoundZClose}), (\ref{Eq.S1BoundZFar}), (\ref{Eq.S1BoundWClose}), (\ref{Eq.S1BoundWFar}), (\ref{Eq.G1BoundFar}) and (\ref{Eq.G1BoundZLog}) we can find $A_6, a_6 > 0$, such that
\begin{equation}\label{Eq.FijBoundinR}
\begin{split}
&\left|F_{12}^N(z,c)\right| \leq \exp\left(A_6 N^{2/3} + N [\SFb(\zc) - \SFb(c)] \right) \leq \exp(-a_6 N)  \mbox{ if } z \in \Gamma_N, \\
&\left|F_{22}^N(z,c)\right| \leq \exp\left(A_6 N^{2/3} + N [\SFb(\zc) - \SFb(c)] \right) \leq \exp(-a_6 N)  \mbox{ if } z \in \gamma_N, \\
&\left|F_{22}^N(c,w)\right| \leq \exp\left(A_6 N^{2/3} + N [\SFb(\zc) - \SFb(c)] \right) \leq \exp(-a_6 N)  \mbox{ if } w \in \gamma_N,
\end{split}
\end{equation}
where in the last inequalities we used that $\SFb(\zc) - \SFb(c) < 0$, as established in (\ref{Eq.DiffS}). Lastly, from (\ref{Eq.G1BoundZLog}), (\ref{Eq.G1BoundZClose}) and (\ref{Eq.G1BoundZFar}), we can find $A_7, a_7 > 0$, such that for $z \in \bar{\gamma}_N$, $s, t \in \mathcal{T}$ with $s > t$
\begin{equation}\label{Eq.G1BoundinR}
\begin{split}
&\left| e^{(T_s - T_t) \bar{\GFb}(z) +  \sigmaq \zc N^{1/3}(y-x)  \log (z/\zc)} \right| \leq e^{- a_7 N^{1/2}} \mbox{ if } |z - \zc| \geq N^{-1/12},\\
&\left| e^{(T_s - T_t) \bar{\GFb}(z) +  \sigmaq \zc N^{1/3}(y-x)  \log (z/\zc)} \right| \leq e^{A_7 - a_7 N^{2/3}|z-\zc|^2} \mbox{ if } |z - \zc| \leq N^{-1/12}.
\end{split}
\end{equation}

%
\subsection{Proof of Proposition \ref{Prop.KernelConvBottom}}\label{Section4.2} For clarity we split the proof into three steps. In the first step we show that we can truncate the contours $\gamma_N, \Gamma_N$ in the formulas for $I_{ij}^N$ in Lemma \ref{Lem.PrelimitKernelBulk} without changing these functions too much, and also get some estimates for $R^N_{12}, R^N_{22}$. In the second step we prove the first line in (\ref{Eq.KernelLimitBottom}), and in the third step we prove the second line in (\ref{Eq.KernelLimitBottom}).   \\

{\bf \raggedleft Step 1.} Fix $L>0$, such that $x_N, y_N \in [-L,L]$. Let $\delta, \epsilon$ be as in the beginning of Section \ref{Section4.1}. Let $\gamma_N(0), \tilde{\gamma}_N(0)$ denote the parts of $\gamma_N, \tilde{\gamma}_N$ that are contained in the disc $\{z: |z - \zc| \leq N^{-1/12}\}$. We also denote by $\Gamma_N(0)$ the part of $\Gamma_N$ that is contained in the disc $\{z: | z -\zc| \leq \delta\}$. Let $I_{11}^{N,0}, I_{12}^{N,0}, I_{22}^{N,0}$ be as in (\ref{Eq.DefIN11Bulk}), (\ref{Eq.DefIN12Bulk}), (\ref{Eq.DefIN22Bulk}) but with $\gamma_N, \Gamma_N$ replaced with $\gamma_N(0), \Gamma_N(0)$. From the first lines in (\ref{Eq.F11Bound}) and (\ref{Eq.HijBound})
\begin{equation}\label{Eq.TruncateI11}
\begin{split}
& \lim_{N \rightarrow \infty} \left| I^N_{11}(s,x_N; t, y_N) -  I^{N,0}_{11}(s,x_N; t, y_N) \right| \leq \lim_{N \rightarrow \infty} \|\Gamma_N\| \cdot \|\Gamma_N\| \cdot  A_5 N^{1/3}e^{- \psi(\delta) N/2} = 0,
\end{split}
\end{equation}
where for a contour $\gamma$, we write $\|\gamma\|$ for its arc-length. From the first line in (\ref{Eq.F12Bound}) and the second line in (\ref{Eq.HijBound})
\begin{equation}\label{Eq.TruncateI12}
\begin{split}
& \lim_{N \rightarrow \infty} \left| I^N_{12}(s,x_N; t, y_N) -  I^{N,0}_{12}(s,x_N; t, y_N) \right| \leq \lim_{N \rightarrow \infty} \|\Gamma_N\| \cdot \|\gamma_N\| \cdot  A_5 N^{2/3}e^{- (\epsilon/2) N^{3/4}} = 0.
\end{split}
\end{equation}
From the first line in (\ref{Eq.F22Bound}) and the third line in (\ref{Eq.HijBound})
\begin{equation}\label{Eq.TruncateI22}
\begin{split}
& \lim_{N \rightarrow \infty} \left| I^N_{22}(s,x_N; t, y_N) -  I^{N,0}_{22}(s,x_N; t, y_N) \right| \leq \lim_{N \rightarrow \infty} \|\gamma_N\| \cdot \|\gamma_N\| \cdot  A_5 N^{1/3}e^{- (\epsilon/2) N^{3/4}} = 0.
\end{split}
\end{equation}

We also note from (\ref{Eq.G1BoundZLog}), the first two lines in (\ref{Eq.HijBound2}), and the first lines in (\ref{Eq.FijBoundinR}) and (\ref{Eq.G1BoundinR}) 
\begin{equation}\label{Eq.R12BulkDecay}
\begin{split}
& \lim_{N \rightarrow \infty} \left| R^N_{12}(s,x_N; t, y_N) - R^{N,0}_{12} \right| \leq \lim_{N \rightarrow \infty}  \|\tilde{\gamma}_N\|  A_5N^{1/3} e^{-a_7 N^{1/2}} +   \|\Gamma_N\| A_5N^{1/3} e^{-a_6 N} = 0,
\end{split}
\end{equation}
where 
\begin{equation}\label{Eq.R120Def}
R^{N,0}_{12} = \frac{-{\bf 1}\{s > t \} \sigmaq \zc N^{1/3} }{2 \pi \im} \int_{\tilde{\gamma}_{N}(0)}\frac{dz}{z} e^{(T_s - T_t) \bar{\GFb}(z)} \cdot e^{ \sigmaq \zc N^{1/3}(y_N-x_N)  \log (z/\zc)} .
\end{equation}
In addition, we have from the third line in (\ref{Eq.HijBound2}) and the last two lines in (\ref{Eq.FijBoundinR})
\begin{equation}\label{Eq.R22BulkDecay}
\begin{split}
& \lim_{N \rightarrow \infty} \left| R^N_{22}(s,x_N; t, y_N) \right| \leq \lim_{N \rightarrow \infty}  2 \|\gamma_N\| \cdot  A_5N^{1/3} e^{-a_6 N} = 0.
\end{split}
\end{equation}

{\bf \raggedleft Step 2.} Recall from Definition \ref{Def.AiryLE} the contour $\mathcal{C}_{a}^{\varphi}=\{a+|s|e^{\mathrm{sgn}(s)\im\varphi}, s\in \mathbb{R}\}$. For $r > 0$ we define the contour $\mathcal{C}_{a}^{\varphi}[r]$, so that outside of the disc $\{z: |z-a| \leq r\}$ it agrees with $\mathcal{C}_{a}^{\varphi}$ and inside the disc it is a vertical segment that connects $a + re^{-\im \varphi}$ and $a + re^{\im \varphi}$, see Figure \ref{Fig.ContoursBulk}. By changing variables $z = \zc + \tilde{z}N^{-1/3}$ and $w = \zc + \tilde{w}N^{-1/3}$, and applying the estimates from the second line in (\ref{Eq.F11Bound}) and the first line in (\ref{Eq.HijBound}) we obtain
\begin{equation}\label{Eq.I11Vanish}
\left|I^{N,0}_{11}(s,x_N; t, y_N) \right| \leq \frac{A_5 N^{-1/3}}{(2\pi)^2} \int_{\mathcal{C}_{0}^{\theta_0}[r_0]} |d\tilde{z}| \int_{\mathcal{C}_{0}^{\theta_0}[r_0]}|d\tilde{w}| e^{- a_2 (|\tilde{z}|^3 + |\tilde{w}|^3) + A_2 ( |\tilde{z}|^2 + |\tilde{w}|^2) + A_2},
\end{equation}
where $r_0 = \sec(\theta_0)$, and $|d\tilde{z}|, |d\tilde{w}|$ denote integration with respect to arc-length. Note that the integral on the right side of (\ref{Eq.I11Vanish}) is finite due to the cubic terms in the exponential. Consequently, the right side in (\ref{Eq.I11Vanish}) vanishes in the $N \rightarrow \infty$ limit due to the factor $N^{-1/3}$. Combining (\ref{Eq.TruncateI11}) and (\ref{Eq.I11Vanish}), we conclude the first limit in (\ref{Eq.KernelLimitBottom}).

Similarly to the previous paragraph, by changing variables $z = \zc +  \tilde{z}N^{-1/3}$ and $w = \zc + \tilde{w}N^{-1/3}$, and applying the estimates from the second line in (\ref{Eq.F22Bound}) and the third line in (\ref{Eq.HijBound}) we obtain
\begin{equation}\label{Eq.I22Vanish}
\left|I^{N,0}_{22}(s,x_N; t, y_N) \right| \leq \frac{A_5 N^{-1/3}}{(2\pi)^2} \int_{\mathcal{C}_{0}^{2\pi/3}} |d\tilde{z}| \int_{\mathcal{C}_{0}^{2\pi/3}}|d\tilde{w}| e^{- a_4 (|\tilde{z}|^3 + |\tilde{w}|^3) + A_4 ( |\tilde{z}|^2 + |\tilde{w}|^2) + A_4},
\end{equation}
which vanishes as $N \rightarrow \infty$ due to the factor $N^{-1/3}$. Combining (\ref{Eq.TruncateI22}), (\ref{Eq.R22BulkDecay}) and (\ref{Eq.I22Vanish}), we conclude the second limit in (\ref{Eq.KernelLimitBottom}).\\

{\bf \raggedleft Step 3.} We start by computing the limit of $I^{N,0}_{12}(s,x_N; t, y_N)$. Changing variables $z = \zc + \tilde{z}N^{-1/3}$, $w = \zc + \tilde{w} N^{-1/3}$, we get
\begin{equation}\label{Eq.I12Bulk1}
\begin{split}
&I^{N,0}_{12}(s,x_N; t, y_N) =  \frac{1}{(2\pi \im)^{2}}\int_{\mathcal{C}_{0}^{\theta_0}[r_0]} d\tilde{z} \int_{\mathcal{C}_{0}^{2\pi/3}} d\tilde{w} {\bf 1} \{|\tilde{w}| \leq N^{1/4}, |\tilde{z}| \leq \delta N^{1/3}\} \\
&\times F_{12}^N(\zc + \tilde{z}N^{-1/3},\zc + \tilde{w}N^{-1/3}) \cdot N^{-2/3} H^{N}_{12}(\zc + \tilde{z}N^{-1/3},\zc + \tilde{w}N^{-1/3}).
\end{split}
\end{equation}
From the Taylor expansion formulas for $\SFb, \GFb$ in Lemma \ref{Lem.PowerSeriesSG} we have the pointwise limit
\begin{equation}\label{Eq.F12PointwiseBulk}
\begin{split}
\lim_{N \rightarrow \infty} F_{12}^N(\zc + \tilde{z}N^{-1/3},\zc + \tilde{w}N^{-1/3}) = e^{\sigmaq^3\tilde{z}^3/3 - \sigmaq^3\tilde{w}^3/3 +s \fq\sigmaq^2\tilde{z}^2 - t\fq \tilde{w}^2 - \sigmaq x \tilde{z} + \sigmaq y \tilde{w}}. 
\end{split}
\end{equation}
In addition, directly from the definition of $H^N_{12}$ in (\ref{Eq.DefIN12Bulk}) we get the pointwise limit
\begin{equation}\label{Eq.H12PointwiseBulk}
\begin{split}
\lim_{N \rightarrow \infty} N^{-2/3}H_{12}^N(\zc + \tilde{z}N^{-1/3},\zc + \tilde{w}N^{-1/3}) = \frac{\sigmaq}{\tilde{z} - \tilde{w}}.
\end{split}
\end{equation}
We also have from the second lines in (\ref{Eq.F12Bound}) and (\ref{Eq.HijBound}) that the integrand in (\ref{Eq.I12Bulk1}) is bounded in absolute value by 
$$ A_5\exp \left( - a_3(|\tilde{z}|^3 + |\tilde{w}|^3) + A_3 (|\tilde{z}|^2 + |\tilde{w}|^2) + A_3 \right).$$ 
The last few observations and the dominated convergence theorem now yield
\begin{equation}\label{Eq.I12Bulk2}
\begin{split}
&\lim_{N \rightarrow \infty} I^{N,0}_{12}(s,x_N; t, y_N) = \frac{1}{(2\pi \im)^{2}}\int_{\mathcal{C}_{0}^{\theta_0}[\sigmaq r_0]} dz \int_{\mathcal{C}_{0}^{2\pi/3}} dw \frac{e^{z^3/3 - w^3/3 +s \fq z^2 - t\fq w^2 - xz +  yw}}{z-w},
\end{split}
\end{equation}
where we also changed variables $z = \sigmaq \tilde{z}$ and $w = \sigmaq \tilde{w}$. \\

We next compute the limit of $R^{N,0}_{12}(s,x_N; t, y_N)$ from (\ref{Eq.R120Def}). Setting $z = \zc + \tilde{z}N^{-1/3}$ gives
\begin{equation}\label{Eq.R12Bulk1}
\begin{split}
&R^{N,0}_{12}(s,x_N; t, y_N) = \frac{-{\bf 1}\{s > t \} \sigmaq \zc }{2 \pi \im} \int_{\im \mathbb{R}} \frac{d\tilde{z}}{\zc + \tilde{z}N^{-1/3}} \cdot  {\bf 1} \{|\tilde{z}| \leq N^{1/4} \} \\
&\times e^{(T_s - T_t) \bar{\GFb}(\zc + \tilde{z}N^{-1/3})} \cdot e^{ \sigmaq \zc N^{1/3}(y_N-x_N)  \log (1 + \zc^{-1} \tilde{z} N^{-1/3})}.
\end{split}
\end{equation}
From the Taylor expansion of $\GFb$ in Lemma \ref{Lem.PowerSeriesSG} the integrand in (\ref{Eq.R12Bulk1}) converges pointwise to 
$$ \frac{1}{\zc} \cdot e^{(s-t) \fq \sigmaq^2 \tilde{z}^2 + \sigmaq (y-x) \tilde{z}}.$$
We also have from the first line in (\ref{Eq.HijBound2}) and the second line in (\ref{Eq.G1BoundinR}) that the integrand in (\ref{Eq.R12Bulk1}) is bounded in absolute value by
$$ A_5 \exp \left(A_7 - a_7|\tilde{z}|^2  \right).$$
The last few observations and the dominated convergence theorem now yield
\begin{equation}\label{Eq.R12Bulk2}
\begin{split}
&\lim_{N \rightarrow \infty} R^{N,0}_{12}(s,x_N; t, y_N) = \frac{-{\bf 1}\{s > t \}  }{2 \pi \im} \int_{\im \mathbb{R}} dz e^{(s-t) \fq z^2 + (y-x)z}
\end{split}
\end{equation}
where we also changed variables $z = \sigmaq \tilde{z}$. \\

Combining (\ref{Eq.TruncateI12}), (\ref{Eq.R12BulkDecay}), (\ref{Eq.I12Bulk2}) and (\ref{Eq.R12Bulk2}), we conclude 
\begin{equation}\label{Eq.K12LimBulk}
\begin{split}
\lim_{N \rightarrow \infty} K^N_{12} (s,x_N; t, y_N) = & -  \frac{{\bf 1}\{s > t \}  }{2 \pi \im} \int_{\im \mathbb{R}} dz e^{(s-t) \fq z^2 + (y-x)z} \\
&+ \frac{1}{(2\pi \im)^{2}}\int_{\mathcal{C}_{0}^{\theta_0}[\sigmaq r_0]} dz \int_{\mathcal{C}_{0}^{2\pi/3}} dw \frac{e^{z^3/3 - w^3/3 +s \fq z^2 - t\fq w^2 - xz +  yw}}{z-w}.
\end{split}
\end{equation}
What remains to prove the second line in (\ref{Eq.KernelLimitBottom}) is to verify that the right side of (\ref{Eq.K12LimBulk}) agrees with
$$A \cdot K^{\mathrm{Airy}}\left(-\fq s,  x + \fq^2 s^2 ; - \fq t, y + \fq^2 t^2 \right), \mbox{ where }  A = e^{2\fq^3s^3/3 - 2\fq^3 t^3/3 + \fq s x - \fq ty} .$$
Using the formula for $K^{\mathrm{Airy}}$ from (\ref{Eq.S1AiryKer}) with changed variables $z \rightarrow z -t_1$ and $w \rightarrow w - t_2$, we see that it suffices to show that
\begin{equation}\label{Eq.MatchBulk}
\begin{split}
& [\mbox{right side of (\ref{Eq.K12LimBulk})}] = -  A \cdot \frac{{\bf 1}\{ s > t\} }{\sqrt{4\pi \fq (s - t)}}  e^{ - \frac{(y + \fq^2 t^2 - x - \fq^2 s^2)^2}{4\fq(s - t)} - \frac{\fq(s - t)(x + y + \fq^2 s^2 + \fq^2 t^2)}{2} + \frac{\fq^3(s - t)^3}{12} }  \\
& + \frac{A}{(2\pi \im)^2} \int_{\mathcal{C}_{\alpha}^{\pi/3}} d z \int_{\mathcal{C}_{\beta}^{2\pi/3}} dw \frac{e^{(z+ \fq s)^3/3 -(x+ \fq^2 s^2)(z+ \fq s ) - (w+ \fq t)^3/3 + (y + \fq^2 t^2)(w-t_2) }}{z - w}.
\end{split}
\end{equation}
where $\alpha > \beta$. After deforming $\mathcal{C}_{0}^{\theta_0}[\sigmaq r_0]$ and $ \mathcal{C}_{0}^{2\pi/3}$ to $\mathcal{C}_{\alpha}^{\pi/3}$ and $\mathcal{C}_{\beta}^{2\pi/3}$, respectively, we see that the second line of (\ref{Eq.K12LimBulk}) agrees with the second line of (\ref{Eq.MatchBulk}). We now change variables $z = \im u$, and directly compute for $s > t$
$$\frac{1 }{2 \pi \im} \int_{\im \mathbb{R}} dz e^{(s-t) \fq z^2 + (y-x)z} =   \frac{1  }{2 \pi} \int_{ \mathbb{R}}  e^{- (s-t) \fq u^2 + \im (y-x)u} du =  \frac{1 }{\sqrt{4 \pi \fq (s-t)} }  \cdot e^{-\frac{(y-x)^2}{4\fq (s-t)}},$$
where in the last equality we used the formula for the characteristic function of a Gaussian with mean $0$ and variance $\frac{1}{2\fq (s-t)}$. Using the last identity one readily shows that the first line of (\ref{Eq.K12LimBulk}) agrees with the first line of (\ref{Eq.MatchBulk}). This completes the proof of (\ref{Eq.MatchBulk}) and hence the proposition.

%
\section{Kernel convergence for the top curve}\label{Section5} The goal of this section is to establish the following statement.

\begin{proposition}\label{Prop.KernelConvTop} Assume the same notation as in Lemma \ref{Lem.PrelimitKernelEdge}, where for each $s \in \mathcal{T}$, we set $\theta_s = \theta_0$, $R_s = R_0$ as in Lemma \ref{Lem.BigContour} for $\kappa = s$. If $x_N, y_N \in \mathbb{R}$ are sequences such that $\lim_{N \rightarrow \infty} x_N = x$, $\lim_{N \rightarrow \infty} y_N = y$, then for any $s,t \in \mathcal{T}$
\begin{equation}\label{Eq.KernelLimitTop}
\begin{split}
&\lim_{N \rightarrow \infty} K_{11}^N(s,x_N;t,y_N) = 0, \hspace{2mm} \lim_{N \rightarrow \infty} K_{22}^N(s,x_N;t,y_N) =  0,\\
& \lim_{N \rightarrow \infty} K^N_{12}(s,x_N; t,y_N) =  \kbm\left(s - \kappa_0,  x  ; t - \kappa_0, y  \right), 
\end{split}
\end{equation}
where 
\begin{equation}\label{Eq.KBM}
\kbm\left(s,  x  ; t, y  \right) = \frac{e^{-x^2/2s}}{\sqrt{2\pi s}} - {\bf 1}\{s > t\} \cdot \frac{e^{-(x-y)^2/2(s-t)}}{\sqrt{2\pi (s-t)}}.
\end{equation}
\end{proposition}

The proof of Proposition \ref{Prop.KernelConvTop} is given in Section \ref{Section5.2}. In Section \ref{Section5.1} we derive suitable estimates for the functions that appear in the kernel $K^N$ in Lemma \ref{Lem.PrelimitKernelEdge} along the contours $\{\Gamma_{\kappa,N}, \gamma_{\kappa}: \kappa \in \mathcal{T}\}$ and $\tilde{\gamma}_N$. In Section \ref{Section5.3} we relate $\kbm$ to a standard Brownian motion.

%
\subsection{Function bounds}\label{Section5.1} In what follows we use the notation from Lemma \ref{Lem.PrelimitKernelEdge} and for each $s \in \mathcal{T}$ we set $\theta_s = \theta_0$, $R_s = R_0$, $\psi_s = \psi$ as in Lemma \ref{Lem.BigContour} for $\kappa = s$. In addition, we work with the contours $\{\Gamma_{\kappa, N}, \gamma_{\kappa}: \kappa \in \mathcal{T} \}$ and $\tilde{\gamma}_N$ as in (\ref{Eq.ContoursEdge}). We also assume that $x, y \in [-L,L]$ for a fixed $L >0$. In the inequalities below we will encounter various constants $A_i,a_i > 0$ with $A_i$ sufficiently large, and $a_i$ sufficiently small, depending on $q, c, \mathcal{T}, \{\theta_\kappa: \kappa \in \mathcal{T}\}, \{R_{\kappa}: \kappa \in \mathcal{T}\}$ and $L$ -- we do not list this dependence explicitly. In addition, the inequalities will hold provided that $N$ is sufficiently large, depending on the same set of parameters, which we will also not mention further.  

Let $\delta_1(\theta), \epsilon_1(\theta)$ be as in Lemma \ref{Lem.DecayNearCritTheta}, $\delta_2, \epsilon_2$ be as in Lemma \ref{Lem.DecayNearCritGen} and set 
$$\delta = \min(\{\delta_1(\theta_\kappa) : \kappa \in \mathcal{T} \}, \delta_2), \hspace{2mm} \epsilon = \min( \{\epsilon_1(\theta_\kappa) : \kappa \in \mathcal{T}\}, \epsilon_2), \hspace{2mm} \psi(\varepsilon) = \min_{ \kappa \in \mathcal{T}} \psi_{\kappa}(\varepsilon).$$
Fix $\kappa \in \mathcal{T}$. If $z \in \Gamma_{\kappa, N}$ and $|z - c| \leq \delta$, we have from Lemmas \ref{Lem.PowerSeriesSG} and \ref{Lem.DecayNearCritTheta} that
\begin{equation}\label{Eq.S2BoundZClose}
\Real[\SFt(z) - \SFt(c)] \leq - \epsilon |z-c|^2 + \left( \sigmap^2(\kappa - \kappa_0)/2c^2 + \epsilon  \right)\sec^2(\theta_{\kappa}) N^{-1} + C_0 \sec^3(\theta_{\kappa}) N^{-3/2}.
\end{equation}
If $z \in \Gamma_{\kappa, N}$ and $|z - c| \geq \delta \geq \sec(\theta_{\kappa})N^{-1/2}$, we have from Lemma \ref{Lem.BigContour} that
\begin{equation}\label{Eq.S2BoundZFar}
\Real[\SFt(z) - \SFt(c)] \leq - \psi(\delta).
\end{equation}
If $w \in \gamma_{\kappa}$, we have from Lemmas \ref{Lem.DiffS2} and \ref{Lem.SmallCircleS} that
\begin{equation}\label{Eq.S2BoundW}
\Real[\SFt(w) - \SFt(c)] \geq \SFt(\zc(\kappa)) - \SFt(c) > 0.
\end{equation}
By Taylor expanding the logarithm we can find $A_1 > 0$, such that for $z \in \Gamma_{\kappa} \cup \gamma_{\kappa} \cup \tilde{\gamma}_N$
\begin{equation}\label{Eq.BoundZLogEdge}
|\log(z/c) | \leq  A_1 | z- c|.
\end{equation}
From Lemma \ref{Lem.DecayNearCritGen} we have for $z \in \tilde{\gamma}_N$ and $|z- c| \leq N^{-1/12} \leq \delta$
\begin{equation}\label{Eq.G2BoundZClose}
\Real[\GFt(z) - \GFt(c)] \leq - \epsilon |z-c|^2. 
\end{equation}
If $z \in \tilde{\gamma}_N$ and $|z - c| \geq N^{-1/12}$, we have from Lemma \ref{Lem.MedCircles} that
\begin{equation}\label{Eq.G2BoundZFar}
\begin{split}
&\Real[\GFt(z) - \GFt(c)] \leq \Real[\GFt(c + e^{\pm \im \pi/2} N^{-1/12}) - \GFt(c)] \leq -\epsilon N^{-1/6},
\end{split}
\end{equation}
where in the last inequality we used (\ref{Eq.G2BoundZClose}).\\

We now proceed to find suitable estimates for the functions $F^N_{ij}$ and $H^N_{ij}$ from (\ref{Eq.DefIN11Edge}-\ref{Eq.DefRN22Edge}). From (\ref{Eq.S2BoundZClose}), (\ref{Eq.S2BoundZFar}) and (\ref{Eq.BoundZLogEdge}) we conclude that for some $A_2,a_2 > 0$ and all $s,t \in \mathcal{T}$, $z \in \Gamma_{s, N}$, $w \in \Gamma_{t, N}$ 
\begin{equation}\label{Eq.F11BoundEdge}
\begin{split}
&\left| F^N_{11}(z,w)  \right| \leq \exp \left( - \psi(\delta) N /2  \right) \mbox{ if } \max(|z-c|, |w-c|) \geq \delta, \\
&\left| F^N_{11}(z,w)  \right| \leq \exp \left(A_2 - a_2 N |z-c|^2 - a_2 N |w-c|^2 \right),  \mbox{ if } \max(|z-c|, |w-c|) \leq \delta, \\
&\left| F^N_{12}(z,c)  \right| \leq \exp \left( - \psi(\delta) N /2  \right) \mbox{ if } |z-c| \geq \delta, \\
&\left| F^N_{12}(z,c)  \right| \leq \exp \left(A_2 - a_2 N |z-c|^2 \right),  \mbox{ if } |z-c| \leq \delta.
\end{split}
\end{equation}
From (\ref{Eq.S2BoundZClose}), (\ref{Eq.S2BoundZFar}), (\ref{Eq.S2BoundW}) and (\ref{Eq.BoundZLogEdge}) we conclude that for some $a_3 > 0$ and all $s,t \in \mathcal{T}$ 
\begin{equation}\label{Eq.F12BoundEdge}
\begin{split}
&\left| F^N_{12}(z,w)  \right| \leq \exp \left( - a_3 N  \right) \mbox{ if } z \in \Gamma_{s, N} \mbox{ and } w \in \gamma_{t}, \\
&\left| F^N_{22}(z,w)  \right| \leq \exp \left( - a_3 N  \right) \mbox{ if } z \in \gamma_{s} \mbox{ and } w \in \gamma_{t}, \\
&\left| F^N_{22}(z,c)  \right| \leq \exp \left( - a_3 N  \right) \mbox{ if } z \in \gamma_{s} \mbox{ and }\left| F^N_{22}(c,w)  \right| \leq \exp \left( - a_3 N  \right) \mbox{ if } w \in \gamma_{t}. \\
\end{split}
\end{equation}
As $H^N_{ij}$ are essentially rational functions, we have for some $A_3 > 0$ that
\begin{equation}\label{Eq.HijBoundEdge}
\begin{split}
&\left| H^N_{11}(z,w)  \right| \leq A_3 N^{1/2}  \mbox{ if } z, w \in \cup_{\kappa \in \mathcal{T}} \Gamma_{\kappa, N} \\
& \left| H^N_{12}(z,w)  \right| \leq A_3 N^{1/2}   \mbox{ if } z\in \cup_{\kappa \in \mathcal{T}} \Gamma_{\kappa, N} \mbox{ and } w \in \cup_{\kappa \in \mathcal{T}} \gamma_{\kappa},\\
& \left| H^N_{22}(z,w)  \right| \leq A_3 N^{1/2}  \mbox{ if } z, w\in \cup_{\kappa \in \mathcal{T}}\gamma_{\kappa},
\end{split}
\end{equation}
and also for each  $s, t \in \mathcal{T}$
\begin{equation}\label{Eq.HijBoundEdge2}
\begin{split}
&\left| \sigmap N^{1/2} \cdot \frac{(1-qz)^{sN - \lfloor sN \rfloor}(1-qc)^{tN - \lfloor tN \rfloor}}{z(1-qz)^{tN - \lfloor tN \rfloor}(1-qc)^{sN - \lfloor sN \rfloor}} \right| \leq A_3 N^{1/2} \mbox{ if } z \in \tilde{\gamma}_N, \\
&\left| \sigmap N^{1/2} \cdot \frac{(zc-1) (1-qz)^{s N - \lfloor s N \rfloor}}{z(z^2-1)(1-qc)^{s N - \lfloor s N \rfloor}}\right| \leq A_3 N^{1/2} \mbox{ if } z \in \Gamma_{s, N} , \\
&\left| \sigmap N^{1/2} \cdot \frac{(1-qc)^{s N - \lfloor s N \rfloor}}{(c  z - 1)(1-qz)^{s N - \lfloor s N \rfloor}} \right| \leq A_3 N^{1/2}  \mbox{ if } z \in \gamma_{s}.
\end{split}
\end{equation}
Lastly, from (\ref{Eq.BoundZLogEdge}), (\ref{Eq.G2BoundZClose}) and (\ref{Eq.G2BoundZFar}) we can find $A_4, a_4 > 0$, such that for $z \in \tilde{\gamma}_N$, $s, t \in \mathcal{T}$ with $s > t$
\begin{equation}\label{Eq.G2BoundinR}
\begin{split}
&\left| e^{(s-t) N \bar{\GFt}(z) + \sigmap (-  x +  y) N^{1/2} \log(z/c) } \right| \leq \exp \left( - a_4 N^{5/6} \right) \mbox{ if } |z - c| \geq N^{-1/12},\\
&\left| e^{(s-t) N \bar{\GFt}(z) + \sigmap (-  x +  y) N^{1/2} \log(z/c) } \right| \leq \exp \left(A_4 -  a_4N|z-c|^2  \right) \mbox{ if } |z - c| \leq N^{-1/12}.
\end{split}
\end{equation}

%
\subsection{Proof of Proposition \ref{Prop.KernelConvTop} }\label{Section5.2} For clarity we split the proof into two steps. In the first step we prove the first line in (\ref{Eq.KernelLimitTop}), and in the second step we prove the second line in (\ref{Eq.KernelLimitTop}).   \\

{\bf \raggedleft Step 1.} In this step we prove the first line in (\ref{Eq.KernelLimitTop}) and establish some estimates for $K^N_{12}$. Fix $L>0$, such that $x_N, y_N \in [-L,L]$. Let $\delta, \epsilon$ be as in the beginning of Section \ref{Section5.1}. Let $\tilde{\gamma}_N(0)$ denote the part of $\tilde{\gamma}_N$ that is contained in the disc $\{z: |z - c| \leq N^{-1/12}\}$, and let $\Gamma_{\kappa,N}(0)$ denote the part of $\Gamma_{\kappa,N}$ that is contained in the disc $\{z: |z - c| \leq \delta \}$.  

From the second line in (\ref{Eq.F12BoundEdge}) and the third line in (\ref{Eq.HijBoundEdge})
\begin{equation}\label{Eq.TruncateI22Edge}
\begin{split}
& \lim_{N \rightarrow \infty} \left| I^N_{22}(s,x_N; t, y_N)  \right|  \leq \lim_{N \rightarrow \infty} \|\gamma_{s}\| \cdot \|\gamma_{t}\| \cdot A_3 N^{1/2}e^{- a_3 N} = 0,
\end{split}
\end{equation}
where we recall that for a contour $\gamma$, we write $\|\gamma\|$ for its arc-length. From the third lines in (\ref{Eq.F12BoundEdge}) and (\ref{Eq.HijBoundEdge2}) we get 
\begin{equation}\label{Eq.TruncateR22Edge}
\begin{split}
& \lim_{N \rightarrow \infty} \left| R^N_{22}(s,x_N; t, y_N)  \right|  \leq \lim_{N \rightarrow \infty} \|\gamma_{s}\| \cdot  A_3 N^{1/2}e^{- a_3 N} + \|\gamma_{t}\| \cdot  A_3 N^{1/2}e^{- a_3 N}  = 0.
\end{split}
\end{equation}
From (\ref{Eq.TruncateI22Edge}) and (\ref{Eq.TruncateR22Edge}) we conclude the second limit in (\ref{Eq.KernelLimitTop}).

Let $I_{11}^{N,0}$ be as in (\ref{Eq.DefIN11Edge}) but with $ \Gamma_{s, N}, \Gamma_{t, N}$ replaced with $\Gamma_{s,N}(0), \Gamma_{t, N}(0)$. Using the first lines in (\ref{Eq.F11BoundEdge}) and (\ref{Eq.HijBoundEdge}), we get
\begin{equation}\label{Eq.TruncateI11Edge}
\begin{split}
& \lim_{N \rightarrow \infty} \left| I^N_{11}(s,x_N; t, y_N) -  I^{N,0}_{11}(s,x_N; t, y_N) \right| \leq \lim_{N \rightarrow \infty} \|\Gamma_{s, N}\| \cdot \|\Gamma_{t, N}\| \cdot A_3N^{1/2}e^{- \psi(\delta)N/2} = 0.
\end{split}
\end{equation}
By changing variables $z = c + \tilde{z}N^{-1/2}$ and $w = c + \tilde{w}N^{-1/2}$, and applying the estimates from the second line in (\ref{Eq.F11BoundEdge}) and the first line in (\ref{Eq.HijBoundEdge}), we obtain
\begin{equation}\label{Eq.I11VanishEdge}
\lim_{N \rightarrow \infty} \left|I^{N,0}_{11}(s,x_N; t, y_N) \right| \leq \lim_{N \rightarrow \infty} \frac{A_3 N^{-1/2}}{(2\pi)^2} \int_{\mathcal{C}_{0}^{\theta_s}[r_s]} |d\tilde{z}| \int_{\mathcal{C}_{0}^{\theta_t}[r_t]}|d\tilde{w}| e^{A_2 - a_2 (|\tilde{z}|^2 + |\tilde{w}|^2) } = 0,
\end{equation}
where $r_s = \sec(\theta_{s})$ and $r_t = \sec(\theta_t)$. We recall that $|d\tilde{z}|, |d\tilde{w}|$ denote integration with respect to arc-length, and that the contours $\mathcal{C}_{a}^{\phi}[r]$ were defined above (\ref{Eq.I11Vanish}). Equations (\ref{Eq.TruncateI11Edge}) and (\ref{Eq.I11VanishEdge}) imply the first limit in (\ref{Eq.KernelLimitTop}).\\

In the remainder of this step we get some estimates for $I^N_{12}$ and $R^N_{12}$ that will be used in the second step. From the first line in (\ref{Eq.F12BoundEdge}) and the second line in (\ref{Eq.HijBoundEdge}) we get
\begin{equation}\label{Eq.TruncateI12Edge}
\begin{split}
& \lim_{N \rightarrow \infty} \left| I^N_{12}(s,x_N; t, y_N)  \right| \leq \lim_{N \rightarrow \infty} \|\Gamma_{s, N}\| \cdot \|\gamma_{t}\| \cdot A_3 N^{1/2}e^{- a_3 N} = 0.
\end{split}
\end{equation}
We also get from the third line in (\ref{Eq.F11BoundEdge}), the first two lines in (\ref{Eq.HijBoundEdge2}) and the first line in (\ref{Eq.G2BoundinR})
\begin{equation}\label{Eq.R12EdgeDecay}
\begin{split}
& \lim_{N \rightarrow \infty} \left| R^N_{12}(s,x_N; t, y_N) - R^{N,1}_{12} - R^{N,2}_{12} \right| \leq \lim_{N \rightarrow \infty} \| \Gamma_{s,N}\| \cdot A_3 N^{1/2} e^{-\psi(\delta) N /2} \\
& + \lim_{N \rightarrow \infty} \|\tilde{\gamma}_N \| \cdot A_3 N^{1/2} e^{- a_4 N^{5/6} }  = 0.
\end{split}
\end{equation}
where 
\begin{equation*}
\begin{split}
&R^{N,1}_{12} =\frac{-{\bf 1}\{s > t \}   \sigmap N^{1/2} }{2 \pi \im}  \int_{\tilde{\gamma}_N(0)} H^1_N(z)  dz, \mbox{ and } R^{N,2}_{12} = \frac{\sigmap N^{1/2}}{2\pi \im} \int_{\Gamma_{s,N}(0)} H^2_N(z)dz,\\
& H^1_N(z) = \frac{{\bf 1}\{ |z - c| \leq N^{-1/12}\} (1-qz)^{sN - \lfloor sN \rfloor}(1-qc)^{tN - \lfloor tN \rfloor}e^{(s-t)N \bar{\GFt}(z) + \sigmap (-  x +  y) N^{1/2} \log(z/c)  }}{z(1-qz)^{tN - \lfloor tN \rfloor}(1-qc)^{sN - \lfloor sN \rfloor}}  ,\\
&H^2_N(z)=  \frac{{\bf 1}\{ |z - c| \leq \delta\}  (zc-1) (1-qz)^{s N - \lfloor s N \rfloor}F_{12}^N(z,c)}{z(z^2-1)(1-qc)^{s N - \lfloor s N \rfloor}}.
\end{split}
\end{equation*}

{\bf \raggedleft Step 2.} From (\ref{Eq.TruncateI12Edge}) and (\ref{Eq.R12EdgeDecay}), we see that to establish the second line in (\ref{Eq.KernelLimitTop}) it suffices to show 
\begin{equation}\label{Eq.EdgeRed1}
\lim_{N \rightarrow \infty} R^{N,1}_{12} = - {\bf 1}\{s > t\} \cdot \frac{e^{-(x-y)^2/2(s-t)}}{\sqrt{2\pi (s-t)}} \mbox{ and } \lim_{N \rightarrow \infty} R^{N,2}_{12} = \frac{e^{-x^2/2(s-\kappa_0)}}{\sqrt{2\pi (s-\kappa_0)}}.
\end{equation}

Changing variables $z = c + \tilde{z} N^{-1/2}$, we see that 
\begin{equation*}
R^{N,1}_{12} = \frac{-{\bf 1}\{s > t \}   \sigmap  }{2 \pi \im} \int_{\im \mathbb{R}}  H^1_N(c + \tilde{z}  N^{-1/2} )d\tilde{z}, \hspace{2mm} R^{N,2}_{12} = \frac{\sigmap}{2\pi \im} \int_{\mathcal{C}^{\theta_s}_0[r_s]} H^2_N(c + \tilde{z}  N^{-1/2} )d\tilde{z}.
\end{equation*}
Using the Taylor expansion formulas for $\SFt, \GFt$ from Lemma \ref{Lem.PowerSeriesSG}, we get the pointwise limits
$$\lim_{N \rightarrow \infty} H^1_N(c + \tilde{z}  N^{-1/2} ) = \frac{e^{(s-t)(\sigmap^2/2c^2)\tilde{z}^2 - \sigmap (x-y)\tilde{z}/ c}}{c}  , \hspace{2mm} \lim_{N \rightarrow \infty} H^2_N(c + \tilde{z}  N^{-1/2} ) = \frac{e^{[\sigmap^2 (s - \kappa_0)/2c^2]\tilde{z}^2 - \sigmap x\tilde{z}/ c}}{c}.$$
On the other hand, we have from the fourth line in (\ref{Eq.F11BoundEdge}), the first two lines in (\ref{Eq.HijBoundEdge2}) and the second line in (\ref{Eq.G2BoundinR}) 
$$|H^1_N(c + \tilde{z}  N^{-1/2} )| \leq A_3 \exp\left( A_4 - a_4 |\tilde{z}|^2 \right), \hspace{2mm} |H^2_N(c + \tilde{z}  N^{-1/2} )| \leq A_3 \exp\left( A_2 - a_2 |\tilde{z}|^2 \right).$$
The last two equations and the dominated convergence theorem yield
\begin{equation*}
\lim_{N \rightarrow \infty} R^{N,1}_{12} = \frac{-{\bf 1}\{s > t \}  }{2 \pi \im} \int_{i \mathbb{R}} e^{z^2(s-t)/2 - x \tilde{z} } dz, \hspace{2mm} \lim_{N \rightarrow \infty} R^{N,2}_{12} = \frac{1}{2\pi \im} \int_{\mathcal{C}^{\theta_s}_0[r_s\sigmap/c]} e^{z^2(s-\kappa_0)/2 - x \tilde{z} } dz,
\end{equation*}
where we also changed variables $z = \tilde{z} \sigmap/c$. The last equation implies (\ref{Eq.EdgeRed1}) once we recognize the integrals as the characteristic functions of Gaussian variables.

%
\subsection{Determinantal structure of Brownian motion}\label{Section5.3} In this section we explain the relationship between Brownian motion and the kernel $\kbm$ from (\ref{Eq.KBM}). The precise statement is contained in the following proposition. 

\begin{proposition}\label{Lem.BMDPP} Let $(B_t: t \geq 0)$ be a standard Brownian motion with $B_0 = 0$. Fix $m \in \mathbb{N}$, $0 < t_1 < t_2 < \cdots < t_m$ and set $\mathcal{T} = \{t_1, \dots, t_m\}$. Then, the random measure 
$$M(A) = \sum_{i = 1}^m {\bf 1}\{(t_i, B_{t_i}) \in A\}$$
is a determinantal point process on $\mathbb{R}^2$ with reference measure $\mu_{\mathcal{T}} \times \mathrm{Leb}$ and correlation kernel $\kbm$ as in (\ref{Eq.KBM}). We recall that $\mu_{\mathcal{T}}$ is the counting measure on $\mathcal{T}$ and $\mathrm{Leb}$ is the Lebesgue measure on $\mathbb{R}$.
\end{proposition}
\begin{proof} Define the random measure on $\mathbb{R}^2$
$$\tilde{M}(A) = \sum_{i = 1}^m {\bf 1}\{ (i, B_{t_i}) \in A \}.$$
From the proof of \cite[Proposition 2.1]{KM10} we have that $\tilde{M}$ is a determinantal point process on $\mathbb{R}^2$ with reference measure $\mu_{\llbracket 1, m \rrbracket} \times \mathrm{Leb}$ and correlation kernel
$$\tilde{K}(u,x; v,y) = S^{u,v}(x,y) - {\bf 1}\{ u > v\} \cdot \frac{e^{-(x-y)^2/2(t_u - t_v)}}{\sqrt{2\pi (t_u - t_v)}},$$
where 
$$S^{u,v}(x,y) = \frac{1}{2\pi \im} \oint_{C_1} dz \frac{e^{-(z-x)^2/2t_u}}{z\sqrt{2\pi t_u}} \int_{\mathbb{R}} dy' \frac{e^{-(y' + \im y)^2/2t_v}}{\sqrt{2\pi t_v}},$$
and $C_1$ is the positively-oriented zero-centered circle of unit radius. The above formula can be deduced from the last displayed equation on \cite[page 482]{KM10} upon setting $N = 1$, $x_1 = 0$, $m = u$, $n = v$ and changing variables $z \rightarrow z/\sqrt{t_u}$, $y' \rightarrow y'/\sqrt{t_v}$. The first integral can be evaluated as the residue at $z = 0$, and the second integral is equal to $1$. Consequently, we get
$$\tilde{K}(u,x; v,y) = \frac{e^{-x^2/2t_u}}{\sqrt{2\pi t_u}} - {\bf 1}\{ u > v\} \cdot \frac{e^{-(x-y)^2/2(t_u - t_v)}}{\sqrt{2\pi (t_u - t_v)}}.$$

Let $f: \mathbb{R} \rightarrow \mathbb{R}$ be a piece-wise linear increasing bijection such that $f(i) = t_i$ for $i \in \llbracket 1, m \rrbracket$. Define $\phi: \mathbb{R}^2 \rightarrow \mathbb{R}^2$ through $\phi(s, x) = \left(f(s),  x\right),$ and note that $M = \tilde{M} \phi^{-1}$. Consequently, from \cite[Proposition 2.13(5)]{ED24a} we conclude that $M$ is determinantal with reference measure $\mu_{\mathcal{T}} \times \mathrm{Leb}$ and correlation kernel $K(x,s;t,y) = \tilde{K}(f^{-1}(s),x; f^{-1}(t),y)$, which from our last formula implies the statement of the proposition.
\end{proof}
\begin{remark}\label{Rem.KMError} We mention that upon inspecting the proof of \cite[Proposition 2.1]{KM10}, we believe that the $y'$ integral in \cite[Equation (2.1)]{KM10} should be over $\mathbb{R} + \im A$ instead of $\mathbb{R}$, where $A \in \mathbb{R}$ is such that $|A| > |z|$ for all $z \in \Gamma(\xi^N)$. The latter can be traced to the first displayed equation on \cite[page 483]{KM10}, where the geometric series is only ensured to converge if $\left| \sqrt{\frac{t_m}{t_n}} \frac{z}{\im y'} \right| < 1$.
\end{remark}

%
\section{Gibbsian line ensembles}\label{Section6} Propositions \ref{Prop.KernelConvBottom} and \ref{Prop.KernelConvTop} will help us show that the convergence in Theorems \ref{Thm.Main1} and \ref{Thm.Main2} occurs in the finite-dimensional sense. In order to conclude the uniform over compacts convergence in those theorems, we will seek to apply the general tightness framework for interlacing geometric line ensembles from \cite{dimitrov2024tightness}. In Section \ref{Section6.1} we introduce the required notation and definitions in that framework, and in Section \ref{Section6.2} we prove some auxiliary results we will require later in the paper. Lastly, in Section \ref{Section6.3} we explain how the Pfaffian Schur processes from Section \ref{Section2} can be interpreted as interlacing geometric line ensembles.

%
\subsection{Definitions}\label{Section6.1} In this section we introduce some basic definitions and notation from \cite{dimitrov2024tightness} and refer the interested reader to the same paper for additional details.

Given an interval $\Lambda \subset \mathbb{R}$, we endow it with the subspace topology of the usual topology on $\mathbb{R}$. We let $(C(\Lambda), \mathcal{C})$ denote the space of continuous functions $f: \Lambda \rightarrow \mathbb{R}$ with the topology of uniform convergence over compacts, see \cite[Chapter 7, Section 46]{Munkres}, and Borel $\sigma$-algebra $\mathcal{C}$. Given a set $\Sigma \subset \mathbb{Z}$, we endow it with the discrete topology and denote by $\Sigma \times \Lambda$ the set of all pairs $(i,x)$ with $i \in \Sigma$ and $x \in \Lambda$ with the product topology. We also denote by $\left(C (\Sigma \times \Lambda), \mathcal{C}_{\Sigma}\right)$ the space of real-valued continuous functions on $\Sigma \times \Lambda$ with the topology of uniform convergence over compact sets and Borel $\sigma$-algebra $\mathcal{C}_{\Sigma}$.

The following defines the notion of a line ensemble.
\begin{definition}\label{CLEDef}
Let $\Sigma \subseteq \mathbb{Z}$ and $\Lambda \subseteq \mathbb{R}$ be an interval. A {\em $\Sigma$-indexed line ensemble $\mathcal{L}$} is a random variable defined on a probability space $(\Omega, \mathcal{F}, \mathbb{P})$ that takes values in $\left(C (\Sigma \times \Lambda), \mathcal{C}_{\Sigma}\right)$. We will often slightly abuse notation and write $\mathcal{L}: \Sigma \times \Lambda \rightarrow \mathbb{R}$, even though it is not $\mathcal{L}$ which is such a function, but $\mathcal{L}(\omega)$ for every $\omega \in \Omega$. For $i \in \Sigma$ we write $\mathcal{L}_i(\omega) = (\mathcal{L}(\omega))(i, \cdot)$ and note that $\mathcal{L}_i: \Omega \rightarrow C(\Lambda)$ is $(\mathcal{C}, \mathcal{F})-$measurable. If $a,b \in \Lambda$ satisfy $a \leq b$, we let $\mathcal{L}_i[a,b]$ denote the restriction of $\mathcal{L}_i$ to $[a,b]$.
\end{definition}

\begin{definition}\label{DefDLE}
Let $\Sigma \subseteq \mathbb{Z}$ and $\llbracket T_0, T_1 \rrbracket$ be a non-empty integer interval in $\mathbb{Z}$. Consider the set $Y$ of functions $f: \Sigma \times \llbracket T_0, T_1 \rrbracket \rightarrow \mathbb{Z}$ such that $f(i, j+1) - f(i,j) \in \mathbb{Z}_{\geq 0}$ when $i \in \Sigma$ and $j \in\llbracket T_0, T_1 -1 \rrbracket$ and let $\mathcal{D}$ denote the discrete $\sigma$-algebra on $Y$. We call a function $g: \llbracket T_0, T_1 \rrbracket \rightarrow \mathbb{Z}$, such that $f( j+1) - f(j) \in \mathbb{Z}_{\geq 0}$ when $j \in\llbracket T_0, T_1 -1 \rrbracket$, an {\em increasing path} and elements in $Y$ are {\em collections of increasing paths}. A $\Sigma$-{\em indexed geometric line ensemble $\mathfrak{L}$ on $\llbracket T_0, T_1 \rrbracket$}  is a random variable defined on a probability space $(\Omega, \mathcal{B}, \mathbb{P})$, taking values in $Y$ such that $\mathfrak{L}$ is a $(\mathcal{B}, \mathcal{D})$-measurable function. Unless otherwise specified, we will assume that $T_0 \leq T_1$ are both integers, although the above definition makes sense if $T_0 = -\infty$, or $T_1 = \infty$, or both.
\end{definition}

We think of geometric line ensembles as collections of random increasing paths on the integer lattice, indexed by $\Sigma$. Observe that one can view an increasing path $L$ on $\llbracket T_0, T_1 \rrbracket$ as a continuous curve by linearly interpolating the points $(i, L(i))$, see Figure \ref{Fig.DiscreteLE}. This allows us to define $ (\mathfrak{L}(\omega)) (i, s)$ for non-integer $s \in [T_0,T_1]$ and to view geometric line ensembles as line ensembles in the sense of Definition \ref{CLEDef}. In particular, we can think of $\mathfrak{L}$ as a random element in $C (\Sigma \times \Lambda)$ with $\Lambda = [T_0, T_1]$. We will often slightly abuse notation and write $\mathfrak{L}: \Sigma \times \llbracket T_0, T_1 \rrbracket \rightarrow \mathbb{Z}$, even though it is not $\mathfrak{L}$ which is such a function, but rather $\mathfrak{L}(\omega)$ for each $\omega \in \Omega$. Furthermore we write $L_i = (\mathfrak{L}(\omega)) (i, \cdot)$ for the index $i \in \Sigma$ path. If $L$ is an increasing path on $\llbracket T_0, T_1 \rrbracket$ and $a, b \in \llbracket T_0, T_1 \rrbracket$ satisfy $a \leq b$, we let $L\llbracket a, b \rrbracket$ denote the restriction of $L$ to $\llbracket a,b\rrbracket$. \\

Let $t_i, z_i \in \mathbb{Z}$ for $i = 1,2$ be given such that $t_1 \leq t_2$ and $z_1 \leq z_2$. We denote by $\Omega(t_1,t_2,z_1,z_2)$ the collection of increasing paths that start from $(t_1,z_1)$ and end at $(t_2,z_2)$, by $\mathbb{P}_{\mathrm{Geom}}^{t_1,t_2, z_1, z_2}$ the uniform distribution on $\Omega(t_1,t_2,z_1,z_2)$ and write $\mathbb{E}^{t_1,t_2,z_1,z_2}_{\mathrm{Geom}}$ for the expectation with respect to this measure. One thinks of the distribution $\mathbb{P}_{\mathrm{Geom}}^{t_1,t_2, z_1, z_2}$ as the law of a random walk with i.i.d. geometric increments with parameter $q \in (0,1)$ that starts from $z_1$ at time $t_1$ and is conditioned to end in $z_2$ at time $t_2$ -- this interpretation does not depend on the choice of $q \in (0,1)$. Notice that by our assumptions on the parameters the state space $\Omega(t_1,t_2,z_1,z_2)$ is non-empty.  

Given $k \in \mathbb{N}$, $T_0, T_1 \in \mathbb{Z}$ with $T_0 < T_1$ and $\vec{x}, \vec{y} \in \mathbb{Z}^k$, we let $\mathbb{P}^{T_0,T_1, \vec{x},\vec{y}}_{\mathrm{Geom}}$ denote the law of $k$ independent geometric bridges $\{B_i: \llbracket T_0, T_1 \rrbracket  \rightarrow \mathbb{Z} \}_{i = 1}^k$ from $B_i(T_0) = x_i$ to $B_i(T_1) = y_i$. Equivalently, this is the law of $k$ independent random increasing paths $B_i \in \Omega(T_0,T_1,x_i,y_i)$ for $i \in \llbracket 1, k \rrbracket$ that are uniformly distributed or just the uniform measure on 
$$\Omega_{\mathrm{Geom}}(T_0, T_1, \vec{x}, \vec{y}) = \Omega(T_0,T_1,x_1,y_1) \times \cdots \times \Omega(T_0,T_1,x_k,y_k).$$
This measure is well-defined provided that $\Omega(T_0,T_1,x_i,y_i)$ are non-empty for $i \in \llbracket 1, k \rrbracket$, which holds if $y_i \geq x_i$ for all $i \in \llbracket 1, k \rrbracket$.

The following definition introduces the notion of an $(f,g)$-interlacing geometric line ensemble, which in simple terms is a collection of $k$ independent geometric bridges, conditioned on interlacing with each other and the graphs of two functions $f$ and $g$.
\begin{definition}\label{DefAvoidingLawBer}
Let $k \in \mathbb{N}$ and $\mathfrak{W}_k$ denote the set of signatures of length $k$, i.e.
\begin{equation}\label{DefSig}
\mathfrak{W}_k = \{ \vec{x} = (x_1, \dots, x_k) \in \mathbb{Z}^k: x_1 \geq  x_2 \geq  \cdots \geq  x_k \}.
\end{equation}
Let $\vec{x}, \vec{y} \in \mathfrak{W}_k$, $T_0, T_1 \in \mathbb{Z}$ with $T_0 < T_1$, and $f: \llbracket T_0, T_1 \rrbracket \rightarrow (-\infty, \infty]$ and $g: \llbracket T_0, T_1 \rrbracket \rightarrow [-\infty, \infty)$ be two functions. With the above data we define the {\em $(f,g)$-interlacing geometric line ensemble on the interval $\llbracket T_0, T_1 \rrbracket$ with entrance data $\vec{x}$ and exit data $\vec{y}$} to be the $\Sigma$-indexed geometric line ensemble $\mathfrak{Q}$ with $\Sigma = \llbracket 1, k\rrbracket$ on $\llbracket T_0, T_1 \rrbracket$ and with the law of $\mathfrak{Q}$ equal to $\mathbb{P}^{T_0,T_1, \vec{x},\vec{y}}_{\mathrm{Geom}}$ (the law of $k$ independent uniform increasing paths $\{B_i: \llbracket T_0, T_1 \rrbracket \rightarrow \mathbb{Z} \}_{i = 1}^k$ from $B_i(T_0) = x_i$ to $B_i(T_1) = y_i$), conditioned on 
\begin{equation}\label{EventInter}
\begin{split}
\ice  = &\left\{ B_i(r-1) \geq B_{i+1}(r)  \mbox{ for all $r \in \llbracket T_0 + 1, T_1 \rrbracket$ and $i \in \llbracket 0 , k \rrbracket$} \right\},
\end{split}
\end{equation}
with the convention that $B_0(x) = f(x)$ and $B_{k+1}(x) = g(x)$.

The above definition is well-posed if there exist $B_i \in \Omega(T_0,T_1,x_i,y_i)$ for $i \in \llbracket 1, k \rrbracket$ that satisfy the conditions in $\ice$. We denote by $\Omega_{\ice}(T_0, T_1, \vec{x}, \vec{y}, f,g)$ the set of collections of $k$ increasing paths that satisfy the conditions in $\ice$ and then the distribution of $\mathfrak{Q}$ is simply the uniform measure on $\Omega_{\ice}(T_0, T_1, \vec{x}, \vec{y}, f,g)$. We denote the probability distribution of $\mathfrak{Q}$ as $\mathbb{P}_{\ice, \mathrm{Geom}}^{T_0,T_1, \vec{x}, \vec{y}, f, g}$ and write $\mathbb{E}_{\ice, \mathrm{Geom}}^{T_0, T_1, \vec{x}, \vec{y}, f, g}$ for the expectation with respect to this measure. If $f=+\infty$ and $g=-\infty$, we drop them from the notation and simply write $\Omega_{\ice}(T_0,T_1,\vec{x},\vec{y})$, $\mathbb{P}^{T_0, T_1, \vec{x},\vec{y}}_{\ice, \mathrm{Geom}}$, and $\mathbb{E}^{T_0, T_1, \vec{x},\vec{y}}_{\ice, \mathrm{Geom}}$.
\end{definition}

The following definition introduces the notion of the interlacing Gibbs property.
\begin{definition}\label{DefSGP}
Fix a set $\Sigma = \llbracket 1, N \rrbracket$ with $N \in \mathbb{N}$ or $N = \infty$ and $T_0, T_1\in \mathbb{Z}$ with $T_0 \leq T_1$. A $\Sigma$-indexed geometric line ensemble $\mathfrak{L} : \Sigma \times \llbracket T_0, T_1 \rrbracket \rightarrow \mathbb{Z}$ is said to satisfy the {\em interlacing Gibbs property} if it is interlacing, meaning that 
$$ L_i(j-1) \geq L_{i+1}(j) \mbox{ for all $i \in \llbracket 1, N - 1 \rrbracket$ and $j \in \llbracket T_0 + 1, T_1 \rrbracket$},$$
and for any finite $K = \llbracket k_1, k_2 \rrbracket \subseteq \llbracket 1, N - 1 \rrbracket$ and $a,b \in \llbracket T_0, T_1 \rrbracket$ with $a < b$ the following holds.  Suppose that $f, g$ are two increasing paths drawn in $\{ (r,z) \in \mathbb{Z}^2 : a \leq r \leq b\}$ and $\vec{x}, \vec{y} \in \mathfrak{W}_k$ with $k=k_2-k_1+1$ altogether satisfy that $\mathbb{P}(A) > 0$ where $A$ denotes the event $$A =\{ \vec{x} = ({L}_{k_1}(a), \dots, {L}_{k_2}(a)), \vec{y} = ({L}_{k_1}(b), \dots, {L}_{k_2}(b)), L_{k_1-1} \llbracket a,b \rrbracket = f, L_{k_2+1} \llbracket a,b \rrbracket = g \},$$
where if $k_1 = 1$ we adopt the convention $f = \infty = L_0$. Then, writing $k = k_2 - k_1 + 1$, we have for any $\{ B_i \in \Omega(a, b, x_i , y_i) \}_{i = 1}^k$ that
\begin{equation}\label{SchurEq}
\mathbb{P}\left( L_{i + k_1-1}\llbracket a,b \rrbracket = B_{i} \mbox{ for $i \in \llbracket 1, k \rrbracket$} \, \vert  A \, \right) = \mathbb{P}_{\ice, \mathrm{Geom}}^{a,b, \vec{x}, \vec{y}, f, g} \left( \cap_{i = 1}^k\{ Q_i = B_i \} \right).
\end{equation}
\end{definition}

%
\subsection{Auxiliary lemmas}\label{Section6.2} In this section we establish two results about the measures $\mathbb{P}^{0, n, \vec{x},\vec{y},f,g}_{\ice, \mathrm{Geom}}$ from Definition \ref{DefAvoidingLawBer}, which we require later in the paper.

\begin{lemma}\label{GLELemma1} Fix $p \in (0, \infty)$, $\epsilon \in (0,1)$, $M^{\mathrm{side}}_1, M^{\mathrm{side}}_2 \in \mathbb{R}$. There exist $N_1 \in \mathbb{N}$, depending on $p, \epsilon, M^{\mathrm{side}}_1, M^{\mathrm{side}}_2$, and $M \in (0, \infty)$, depending on $p, \epsilon$ alone, such that the following holds for all $n \geq N_1$. If we assume that:
\begin{itemize}
\item $\vec{x}, \vec{y}\in \mathfrak{W}_2$ as in (\ref{DefSig}), and $x_2 \geq M^{\mathrm{side}}_1 \cdot n^{1/2}$, $y_2 - pn \geq M^{\mathrm{side}}_2 \cdot n^{1/2}$,
\item $g: \llbracket 0, n \rrbracket \rightarrow [-\infty, \infty)$, 
\item the set $\Omega_{\ice}(0,n,\vec{x},\vec{y}, \infty, g)$ is non-empty,
\end{itemize}
then the following inequality holds 
\begin{equation}\label{Eq.GLELemma1}
\mathbb{P}_{\ice, \mathrm{Geom}}^{0, n, \vec{x}, \vec{y}, \infty, g}\left( n^{-1/2}(Q_2\left(sn\right) - psn)  \geq (1-s) M^{\mathrm{side}}_1  + s M^{\mathrm{side}}_2 - M \mbox{ for all } s\in [0,1]\right) > 1- \epsilon.
\end{equation}
\end{lemma}
\begin{remark}\label{Rem.DoesntDip}
In plain words, the above lemma says that a line ensemble with two curves, whose second curve has endpoints that are not too low, does not dip much.
\end{remark}
\begin{proof} Set $\sigma = \sqrt{p(1+p)}$, and let $(B_1, B_2)$ have law $\mathbb{P}_{\mathrm{avoid}}^{0, 1, \vec{z}, \vec{w},  \infty, -\infty}$ as in \cite[Definition 2.14]{dimitrov2024tightness} with $\vec{z} = (0, -\sigma^{-1})$, $\vec{w} = (0, - \sigma^{-1})$. In plain words, $(B_1, B_2)$ are two independent Brownian bridges with $B_1(0) = B_1(1) = 0$, $B_2(0) = B_2(1) = -\sigma^{-1}$, conditioned to avoid each other on $[0,1]$. We fix $M > 0$, sufficiently large (depending on $p$ and $\epsilon$) so that 
\begin{equation}\label{eq:highE2}
\mathbb{P}(\sigma B_2(s) > - M  \mbox{ for all } s\in [0,1]  ) > 1 - \epsilon/2.
\end{equation}
This specifies the choice of $M$ in the lemma.

Define the vectors
$$\vec{u}^n=\left(\lfloor M^{\mathrm{side}}_1 \cdot n^{1/2} \rfloor, \lfloor (M^{\mathrm{side}}_1 - 1 )\cdot n^{1/2} \rfloor\right), \hspace{2mm} \vec{v}^n=\left(\lfloor pn + M^{\mathrm{side}}_2 \cdot n^{1/2}\rfloor, \lfloor pn + (M^{\mathrm{side}}_2 - 1) \cdot n^{1/2} \rfloor\right),$$
and note that for all large $n$
$$u^n_1 \leq x_1, \hspace{2mm} u^n_2 \leq x_2, \hspace{2mm} v^n_1 \leq y_1, \hspace{2mm} v^n_2 \leq y_2, \hspace{2mm} u^n_1 \leq v^n_1, \hspace{2mm} u^n_2 \leq v^n_2.$$ 
The last two inequalities imply $\Omega_{\ice}(0,n,\vec{u}^n,\vec{v}^n, \infty, -\infty) \neq \emptyset$. From the monotone coupling in \cite[Lemma 2.12]{dimitrov2024tightness}, it suffices to show that for all large $n$,
\begin{equation}\label{eq:highE1}
\mathbb{P}_{\ice, \mathrm{Geom}}^{0, n, \vec{u}^n, \vec{v}^n}\left( n^{-1/2}(Q_2\left(sn\right) - psn)  \geq (1-s) M^{\mathrm{side}}_1  + s M^{\mathrm{side}}_2 - M \mbox{ for all } s\in [0,1]\right) > 1- \epsilon.
\end{equation}

From \cite[Lemma 2.16]{dimitrov2024tightness} applied to $d_n = n$, $a = 0$, $b = 1$, $A_n = 0$, $B_n = n$, $f_n = \infty$, $g_n = -\infty$, $\vec{X}^n = \vec{u}^n$, $\vec{Y}^n = \vec{v}^n$, and an affine shift, we conclude that the $\llbracket 1, 2 \rrbracket$-indexed line ensembles $\mathcal{Q}^n$ on $[0,1]$, defined by 
$$\mathcal{Q}^n_i(s) =  n^{-1/2} \cdot \left(Q_i(sn) - psn \right) - (1-s) M^{\mathrm{side}}_1  - s M^{\mathrm{side}}_2  \mbox{ for } t\in [0,1], i = 1,2$$
converge weakly to $(\sigma B_1, \sigma B_2)$, where $(B_1, B_2)$ are as above. The latter and (\ref{eq:highE2}) give
$$\liminf_{n \rightarrow \infty} \mathbb{P}_{\ice, \mathrm{Geom}}^{0, n, \vec{u}^n, \vec{v}^n}(\mathcal{Q}^n_2(s) \geq - M \mbox{ for all } s\in [0,1]) \geq \mathbb{P}(\sigma B_2(s) > - M  \mbox{ for all } s\in [0,1]  ) > 1 - \epsilon/2,$$
which implies (\ref{eq:highE1}).
\end{proof}

We next turn to the second key result of the section.
\begin{lemma}\label{GLELemma2} Fix $\epsilon \in (0,1)$, $p, M^{\mathrm{side}} \in (0,\infty)$, $M^{\mathrm{bot}} \in \mathbb{R}$, $k \in \mathbb{Z}_{\geq 2}$. There exist $N_2 \in \mathbb{N}$ and $M^{\mathrm{sep}} > 0$, depending on all previously listed constants, such that the following holds for all $n \geq N_2$. Suppose that:
\begin{itemize}
\item $\vec{x}, \vec{y} \in \mathfrak{W}_k$ as in (\ref{DefSig}), and $|x_i|  \leq M^{\mathrm{side}} \cdot n^{1/2}$, $|y_i - pn| \leq M^{\mathrm{side}} \cdot n^{1/2}$ for all $i\in \llbracket 2, k \rrbracket$;
\item $x_1 \geq M^{\mathrm{sep}} \cdot n^{1/2}$, $y_1 - pn \geq M^{\mathrm{sep}} \cdot n^{1/2}$;
\item $g: \llbracket 0, n \rrbracket \rightarrow [-\infty, \infty)$ is such that $g(r) \leq pr + M^{\mathrm{bot}} \cdot n^{1/2} $ for all $r \in \llbracket 0, n \rrbracket$;
\item The sets $\Omega_{\ice}(0,n,\vec{x},\vec{y}, \infty, g)$ and $\Omega_{\ice}(0,n,\vec{u},\vec{v}, \infty, g)$ are non-empty, where $u_i = x_{i+1}$, $v_i = y_{i+1}$ for $i \in \llbracket 1, k-1\rrbracket$.
\end{itemize}
Then, there exists a coupling of $\mathfrak{Q} = (Q_1, \dots, Q_k)$, which has law $\mathbb{P}_{\ice, \mathrm{Geom}}^{0, n, \vec{x}, \vec{y}, \infty, g}$, and $\tilde{\mathfrak{Q}} = (\tilde{Q}_1, \dots, \tilde{Q}_{k-1})$, which has law $\mathbb{P}_{\ice, \mathrm{Geom}}^{0, n, \vec{u}, \vec{v}, \infty, g}$ on the same probability space $(\Omega, \mathcal{F}, \mathbb{P})$, such that 
\begin{equation}\label{Eq.EqualEnsembles}
\mathbb{P}\left(\tilde{Q}_{i}(r) = Q_{i+1}(r) \mbox{ for all } r\in \llbracket 0, n\rrbracket, i \in \llbracket 1, k-1 \rrbracket\right) > 1 -\epsilon. 
\end{equation}
\end{lemma}
\begin{remark}
In plain words, the above lemma says that if we start with an interlacing geometric line ensemble with $k$ curves, whose top curve has very high endpoints, then the bottom $k-1$ curves of this ensemble with high probability behave like an ensemble with $k-1$ curves. I.e., the bottom $k-1$ curves do not feel the effect of the top curve. 
\end{remark}
\begin{proof}
Let $(\Omega, \mathcal{F}, \mathbb{P})$ be a probability space that supports two independent i.i.d. sequences $\{Q_1^m\}_{m \geq 1}$, $\{\tilde{\mathfrak{Q}}^m\}_{m \geq 1}$, where $Q_1^m$ is uniform on $\Omega(0,n,x_1,y_1, \infty, -\infty)$, and $\tilde{\mathfrak{Q}}^m = (\tilde{Q}^m_1, \dots, \tilde{Q}^m_{k-1})$ is uniform on $\Omega_{\ice}(0,n,\vec{u},\vec{v}, \infty, g)$. Let $K$ be the smallest positive integer, such that 
$$ Q^K_1(r-1) \geq \tilde{Q}^K_1(r)  \mbox{ for all $r \in \llbracket 1, n \rrbracket$},$$
and note that $K$ is a geometric random variable with parameter strictly less than $1$, as we assumed $\Omega_{\ice}(0,n,\vec{x},\vec{y}, \infty, g) \neq \emptyset$. Define $\mathfrak{Q}:= (Q_1^K, \tilde{Q}_1^K, \tilde{Q}^K_1, \dots, \tilde{Q}^K_{k-1})$ and $\tilde{\mathfrak{Q}}:= \tilde{\mathfrak{Q}}^1$, and note that $\mathfrak{Q} $ has law $\mathbb{P}_{\ice, \mathrm{Geom}}^{0, n, \vec{x}, \vec{y}, \infty, g}$, while $\tilde{\mathfrak{Q}}$ has law $\mathbb{P}_{\ice, \mathrm{Geom}}^{0, n, \vec{u}, \vec{v}, \infty, g}$ by construction. The latter specifies our coupling, and to conclude (\ref{Eq.EqualEnsembles}), it suffices to show that there exist $N_2 \in \mathbb{N}$ and $M^{\mathrm{sep}} > 0$, such that $\mathbb{P}(K=1)\ge 1-\epsilon$ for $n \geq N_2$, or equivalently
\begin{equation}\label{eq:coupleE1}
\mathbb{P}(Q_1^{1}(r-1) < \tilde{Q}_1^{1} (r) \mbox{ for some } r \in \llbracket 1, n \rrbracket ) \leq \epsilon.
 \end{equation}

Set $\sigma = \sqrt{p(1+p)}$, and define 
$$ z_{i} = w_i = \sigma^{-1} (|M^{\mathrm{bot}}| + M^{\mathrm{side}} + k-i ) \mbox{ for } i \in \llbracket 1, k-1 \rrbracket, \hspace{2mm} \tilde{g}(s) = \sigma^{-1} M^{\mathrm{bot}} \mbox{ for } s \in [0,1].$$
Let $B_1$ be a standard Brownian bridge on $[0,1]$ with $B_1(0) =B_1(1) = 0$, and let $\tilde{\mathcal{B}} = (\tilde{B}_1, \dots, \tilde{B}_{k-1})$ be independent from $B_1$, and with law $\mathbb{P}_{\mathrm{avoid}}^{0, 1, \vec{z}, \vec{w},  \infty, \tilde{g}}$ as in \cite[Definition 2.14]{dimitrov2024tightness}. By the continuity of $B_1, \tilde{B}_1$, we can choose $M^{\mathrm{sep}} > 0$ large enough, so that
 \begin{equation}\label{Eq.OutOfOrder}
        \mathbb{P}(B_1(s) + \sigma^{-1} M^{\mathrm{sep}} \leq \tilde{B}_1(s) \mbox{ for some } s \in [0,1] ) \leq \epsilon/2.
\end{equation}
This specifies our choice of $M^{\mathrm{sep}}$ for the remainder of the proof.\\

We now define  
$$X^n_1 = \lfloor  M^{\mathrm{sep}} \cdot n^{1/2} \rfloor, \hspace{2mm} Y^n_1 = \lfloor  pn + M^{\mathrm{sep}} \cdot n^{1/2} \rfloor, \hspace{2mm} U_{i}^n = \lceil (|M^{\mathrm{bot}}| + M^{\mathrm{side}} + k - i) \cdot n^{1/2} \rceil,  $$ 
$$V_{i}^n = \lceil pn + ( |M^{\mathrm{bot}}| + M^{\mathrm{side}} + k-i) \cdot n^{1/2} \rceil, \mbox{ for $i\in \llbracket 1, k-1 \rrbracket$ }, \hspace{2mm} G_n(r) = pr + M^{\mathrm{bot}} \cdot n^{1/2} \mbox{ for } r \in \llbracket 0, n \rrbracket.$$
From the assumptions in the lemma, we have the following inequalities
$$x_1 \geq X^n_1, \hspace{2mm} y_1 \geq Y^n_1, \hspace{2mm} U_{i}^n \geq u_{i}, \hspace{2mm} V_{i}^n \geq v_{i} \mbox{ for } i \in \llbracket 1,k-1 \rrbracket, \mbox{ and } G_n(r) \geq g(r) \mbox{ for } r\in \llbracket 0, n \rrbracket.$$
We also observe that for all large $n$ the sets $\Omega(0,n,X^n_1,Y^n_1, \infty, -\infty)$, $\Omega(0,n,\vec{U}^n, \vec{V}^n, \infty, G_n)$ are non-empty. For the former this follows from $Y^n_1 \geq X^n_1$, while for the latter it follows from \cite[Lemma 2.16(1)]{dimitrov2024tightness}, with $k$ there equal to $k-1$ here, and
\begin{equation}\label{Eq.ApplConv}
\begin{split}
&d_n = n, \hspace{2mm} \vec{x} = \vec{z}, \hspace{2mm} \vec{y} = \vec{w}, \hspace{2mm} g(t) = \sigma^{-1} M^{\mathrm{bot}}, \hspace{2mm} f(t) = \infty, \hspace{2mm} \vec{X}^n = \vec{U}^n = (U_1^n, \dots, U^{n}_{k-1}), \hspace{2mm} \\
& \vec{Y}^n = \vec{V}^n = (V_1^n, \dots, V^{n}_{k-1}), \hspace{2mm} A_n = 0, \hspace{2mm} B_n = n, \hspace{2mm} a = 0, \hspace{2mm} b = 1.
\end{split}
\end{equation}
From the monotone coupling \cite[Lemma 2.12]{dimitrov2024tightness}, by possibly enlarging $(\Omega, \mathcal{F}, \mathbb{P})$, we may assume that the space supports $L^n_1$ with law $\mathbb{P}_{\mathrm{Geom}}^{0, n, X_1^n, Y_1^n}$, and $\tilde{\mathfrak{L}}^n= (\tilde{L}^n_1, \dots, \tilde{L}^n_{k-1})$ with law $\mathbb{P}_{\ice, \mathrm{Geom}}^{0, n, \vec{U}^n, \vec{V}^n, \infty, G_n}$, such that $L^n_1$ is independent of $\tilde{\mathfrak{L}}^n$, and $\mathbb{P}$-almost surely
$$L^n_1(r) \leq Q_1^1(r) \mbox{, and } \tilde{L}_i^n(r) \geq \tilde{Q}_i^1(r) \mbox{ for } i \in \llbracket 1, k - 1 \rrbracket, r \in \llbracket 0, n \rrbracket.$$
In particular, we see that to conclude (\ref{eq:coupleE1}), it suffices to show for all large $n$ 
\begin{equation}\label{Eq.CloseRed1}
\mathbb{P}(L_1^{n}(r-1) < \tilde{L}_1^{n} (r) \mbox{ for some } r \in \llbracket 1, n \rrbracket ) \leq \epsilon.
\end{equation}

From \cite[Lemma 2.16]{dimitrov2024tightness} with parameters as in (\ref{Eq.ApplConv}), we know that the $\llbracket 1, k-1 \rrbracket$-indexed line ensembles $\tilde{\mathcal{L}}^n = (\tilde{\mathcal{L}}^n_1, \dots, \tilde{\mathcal{L}}^n_{k-1})$ on $[0,1]$, defined by 
$$\tilde{\mathcal{L}}^n_i(s) =  \sigma^{-1} n^{-1/2} \cdot \left(\tilde{L}^n_i(sn) - psn \right) \mbox{ for } i \in \llbracket 1, k -1 \rrbracket, s \in [0,1], $$
converge weakly to $\tilde{\mathcal{B}}$ as above. In addition, by the same lemma for $k = 1$, and parameters as in (\ref{Eq.ApplConv}), but with $x_1 = y_1 = \sigma^{-1} M^{\mathrm{sep}}$ and $\vec{X}^n = X^n_1$, $\vec{Y}^n = Y^n_1$, we know that the sequence $\mathcal{L}^n_1 \in C([0,1])$, defined by
$$\mathcal{L}^n_1(s) =  \sigma^{-1} n^{-1/2} \cdot \left(L^n_1(sn) - psn \right) \mbox{ for } s \in [0,1],$$
converges weakly to $B_1 + \sigma^{-1} M^{\mathrm{sep}}$ as above. As $\tilde{\mathcal{L}}^n$ and $\mathcal{L}^n_1$ are independent, we in fact have that $(\mathcal{L}^n_1, \tilde{\mathcal{L}}^n)$ converge jointly to $(B_1, \tilde{B})$. Consequently, 
\begin{equation*}
\begin{split}
&\limsup_{n \rightarrow \infty} \mathbb{P}(L_1^{n}(r-1) < \tilde{L}_1^{n} (r) \mbox{ for some } r \in \llbracket 1, n \rrbracket ) \\
& \leq \limsup_{n \rightarrow \infty} \mathbb{P}(\mathcal{L}_1^{n}(s-1/n) < \tilde{\mathcal{L}}_1^{n} (s) \mbox{ for some } s \in [1/n, 1] ) \\
& \leq \mathbb{P}(B_1(s) + \sigma^{-1} M^{\mathrm{sep}} \leq \tilde{B}_1(s) \mbox{ for some } s \in [0, 1] ).
\end{split}
\end{equation*}
The last inequality and (\ref{Eq.OutOfOrder}) imply (\ref{Eq.CloseRed1}).
\end{proof}

%
\subsection{Pfaffian Schur processes as line ensembles}\label{Section6.3}
In this section we explain how the Pfaffian Schur processes from Definition \ref{Def.SchurProcess} can be interpreted as $\mathbb{N}$-indexed geometric line ensembles that satisfy the interlacing Gibbs property. The precise statement is contained in the following lemma.

\begin{lemma}\label{Lem.InterlacingGibbs} Assume the same notation as in Definition \ref{Def.SchurProcess} with parameters as in (\ref{Eq.HomogeneousParameters}). Let $\mathfrak{L} = \{L_i\}_{i \geq 1}$ be the $\mathbb{N}$-indexed geometric line ensemble on $\llbracket 0, N \rrbracket$, defined by
\begin{equation}\label{Eq.SchurLE}
L_i(s)  = \lambda^{N-s+1}_i \mbox{ for } i \geq 1, s \in \llbracket 0, N \rrbracket.
\end{equation}
Then, $\mathfrak{L}$ satisfies the interlacing Gibbs property from Definition \ref{DefSGP}. 
\end{lemma}
\begin{proof} The fact that $\mathfrak{L}$ is an $\mathbb{N}$-indexed geometric line ensemble that is interlacing follows from Defintion \ref{Def.SchurProcess} and (\ref{Eq.SkewSchur}), which imply that with full probability 
\begin{equation}\label{Eq.InterlaceGibbs}
\lambda^1 \succeq \lambda^2 \succeq \cdots \succeq \lambda^N \succeq \lambda^{N+1} = \emptyset.
\end{equation}
Let $\mu^0$, $\mu^N$ satisfy $\mathbb{P}(\mathfrak{L}(\cdot, 0) = \mu^0, \mathfrak{L}(\cdot, N) = \mu^N) > 0$. From (\ref{Eq.SchurProcess}) and (\ref{Eq.SchurLE}), we have for any sequence of partitions $\{\mu^t: t \in \llbracket 1, N-1\rrbracket\}$ that
\begin{equation}\label{Eq.SchurInterlacingGibbs}
\begin{split}
&\mathbb{P}\left( \cap_{t \in \llbracket 0, N \rrbracket} \{\mathfrak{L}(\cdot, t) = \mu^t\} | \mathfrak{L}(\cdot, 0) = \mu^0, \mathfrak{L}(\cdot, N) = \mu^N \right) \propto \prod_{r \in \llbracket 1, N \rrbracket} s_{\mu^r/\mu^{r-1}}(q)  \\
& \propto \prod_{i \geq 1} q^{\mu^N_i - \mu^0_i} \times {\bf 1} \{\mu^0 \preceq \mu^{1} \preceq \cdots \preceq \mu^N\} \propto {\bf 1} \{\mu^0 \preceq \mu^{1} \preceq \cdots \preceq \mu^N\},
\end{split}
\end{equation}
where in going from the first to the second line we used (\ref{Eq.SkewSchur}), and the various constants of proportionality depend on $c,q,a,b,N,\mu^0, \mu^N$. 

Equation (\ref{Eq.SchurInterlacingGibbs}) shows that for each $m \geq 1$, the law of $\{L_i\}_{i = 1}^m$ is a convex combination of measures of the form $\mathbb{P}_{\ice, \mathrm{Geom}}^{0, N, \vec{x}, \vec{y}, \infty, g}$ as in Definition \ref{DefAvoidingLawBer} with different $\vec{x}, \vec{y}, g$. As each $\mathbb{P}_{\ice, \mathrm{Geom}}^{0, N, \vec{x}, \vec{y}, \infty, g}$ satisfies the interlacing Gibbs property (see \cite[Lemma 2.10]{dimitrov2024tightness}) we conclude the same for $\mathfrak{L}$.
\end{proof}

%
\section{Weak convergence of the top curve}\label{Section7} The goal of this section is to prove Theorem \ref{Thm.Main1}. In Section \ref{Section7.1} we prove two general results that allow us to conclude the finite-dimensional convergence of a sequence of random vectors, that form point processes, for which we know a priori are weakly convergent. These results are then applied in Section \ref{Section7.2} to prove the finite-dimensional convergence of $(Y^{j,N}_1: j \in \llbracket 1, m \rrbracket)$ from Definition \ref{Def.ScalingEdge}, stated as Proposition \ref{Prop.FinitedimEdge}. Finally, in Section \ref{Section7.3} we prove Theorem \ref{Thm.Main1}, using the previously established convergence of $(Y^{j,N}_1: j \in \llbracket 1, m \rrbracket)$ and results from \cite{ED24a}.

%
\subsection{Finite-dimensional convergence for point processes}\label{Section7.1} As we will see later, the kernel convergence established in Proposition \ref{Prop.KernelConvTop} allows us to conclude that the point processes $M^N$ from Lemma \ref{Lem.PrelimitKernelEdge} converge weakly. Our goal is then to upgrade this to finite-dimensional convergence for the vectors $(Y^{j,N}_1: j \in \llbracket 1, m \rrbracket)$. In this section, we establish two general results that facilitate this step. For the sake of potential future applications, we formulate these results in a more general setting than currently needed. Throughout, we freely use the definitions and notation related to point processes from \cite[Section 2]{ED24a}.

We will work with the space $[-\infty,\infty)$, which is homeomorphic to $[0,\infty)$ via the map $x \rightarrow e^x$. Fix $m \in \mathbb{N}$, $t_1 < \cdots < t_m$ and let $\mathcal{T} = \{t_1, \dots, t_m\}$. Suppose that $(X^{j,N}_i: i \geq 1,  j \in \llbracket 1, m \rrbracket)$ is a sequence of random elements in $[-\infty, \infty)^{\infty}$, endowed with the product topology and corresponding Borel $\sigma$-algebra, such that 
\begin{equation}\label{Eq.OrdY}
X^{j,N}_i(\omega) \geq X^{j, N}_{i+1}(\omega) \mbox{ for each $\omega \in \Omega$, $i \geq 1$, and $j \in \llbracket 1, m \rrbracket$.}
\end{equation}
We further suppose that the random measures $M^N$ on $\mathbb{R}^2$, defined by 
\begin{equation}\label{Eq.YMeasForm}
M^N(\omega,A) = \sum_{i \geq 1} \sum_{j = 1}^m {\bf 1}\{(t_j, X^{j,N}_{i}(\omega)) \in A\},
\end{equation}
are locally finite, and hence are point processes. With the above definitions in place, we now state the first key result of the section.
\begin{lemma}\label{Lem.FDC} Assume the same notation as in the previous paragraph. Fix $r \in \mathbb{N}$ and a random vector $X = (X^{j}_i: i \in \llbracket 1, r \rrbracket ,  j \in \llbracket 1, m \rrbracket)$ in $\mathbb{R}^{rm}$, such that 
\begin{equation}\label{Eq.OrdY2}
X^{j}_{i}(\omega) \geq X^{j}_{i + 1}(\omega) \mbox{ for } i \in \llbracket 1, r-1\rrbracket \mbox{ and } j \in \llbracket 1, m \rrbracket,
\end{equation}
and let $M$ be the point process on $\mathbb{R}^2$ formed by $\{(t_j, X^{j}_i): i \in \llbracket 1, r \rrbracket ,  j \in \llbracket 1, m \rrbracket \}$. Suppose further that $M^N$ as in (\ref{Eq.YMeasForm}) converge weakly to $M$, and that $\{X^{j,N}_i\}_{N \geq 1}$ are real-valued and tight in $\mathbb{R}$ for each $i \in \llbracket 1, r \rrbracket ,  j \in \llbracket 1, m \rrbracket$. Then, $(X^{j,N}_i: i \in \llbracket 1, r \rrbracket ,  j \in \llbracket 1, m \rrbracket) \Rightarrow X$ as random vectors in $\mathbb{R}^{rm}$.
\end{lemma}
\begin{proof} The proof we present is an adaptation of \cite[Proposition 2.19]{ED24a}. Set $\tilde{X}^N = (X^{j,N}_i: i \in \llbracket 1, r \rrbracket ,  j \in \llbracket 1, m \rrbracket)$, which by assumption is a tight sequence in $\mathbb{R}^{rm}$. Let $\tilde{X}^{N_v}$ be a weakly convergent subsequence of $\tilde{X}^N$ converging to some random vector $\tilde{X} \in \mathbb{R}^{rm}$. Since $\tilde{X}^N$ is tight, it suffices to show that $\tilde{X}$ and $X$ have the same distribution. From (\ref{Eq.OrdY2}) and the fact that $M$ is the point process on $\mathbb{R}^2$ formed by $\{(t_j, X^{j}_i): i \in \llbracket 1, r \rrbracket ,  j \in \llbracket 1, m \rrbracket \}$, we obtain for any $a_{i,j} \in \mathbb{R}$ that
\begin{equation}\label{Eq.MultiWeak5}
\begin{split}
&\mathbb{P}\left(X_i^j \leq a_{i,j} \mbox{ for } i \in \llbracket 1, r \rrbracket, j \in \llbracket 1, m \rrbracket \right)  \\
&=  \mathbb{P}\left(M(\{t_j\} \times (a_{i,j}, \infty) ) < i \mbox{ for } i \in \llbracket 1, r \rrbracket, j \in \llbracket 1, m \rrbracket \right).
\end{split}
\end{equation}

Note that $M^N$ are random elements of $\mathcal{M}_{\mathbb{R}^2}$ -- the space of locally bounded measures on $\mathbb{R}^2$ with the vague topology -- which is Polish. As $M^N$ are weakly convergent, they are precompact, and hence tight by Prohorov's theorem \cite[Theorem 5.2]{Bill}. It follows that $(\tilde{X}^N, M^N)$, regarded as random elements in $\mathbb{R}^{rm} \times \mathcal{M}_{\mathbb{R}^2}$ with the product topology, are tight. By possibly passing to a further subsequence, still denoted $N_v$, we may assume that $(\tilde{X}^{N_v}, M^{N_v})$ converges weakly to some $(\tilde{X}^{\infty}, M^{\infty})$. By Skorohod's representation theorem \cite[Theorem 6.7]{Bill}, we may assume that $(\tilde{X}^{N_v}, M^{N_v})$ and $(\tilde{X}^{\infty}, M^{\infty})$ are defined on the same probability space $(\Omega, \mathcal{F}, \mathbb{P})$ and for each $\omega \in \Omega$ 
\begin{equation}\label{Eq.MultiWeak1}
M^{N_v}(\omega) \xrightarrow{v} M^{\infty}(\omega) \text { and } X_i^{j,N_v}(\omega) \rightarrow \tilde{X}_i^{j,\infty}(\omega) \in \mathbb{R} \mbox{ for } i \in \llbracket 1, r \rrbracket, j \in \llbracket 1, m \rrbracket.  
\end{equation}
As weak limits are unique, $M^{\infty}$ has the same distribution as $M$, and $\tilde{X}^{\infty}$ has the same distribution as $\tilde{X}$. Thus, it remains to show that $\tilde{X}^{\infty}$ has the same distribution as $X$. In view of (\ref{Eq.MultiWeak5}) and the distributional equality of $M^{\infty}$ and $M$, we see that it suffices to show that for any $a_{i,j} \in \mathbb{R}$
\begin{equation}\label{Eq.MultiWeak6}
\begin{split}
&E_1 = E_2, \mbox{ where } E_1 =  \{ \omega \in \Omega :  \tilde{X}_i^{j, \infty}(\omega) \leq a_{i,j} \mbox{ for } i \in \llbracket 1, r \rrbracket, j \in \llbracket 1, m \rrbracket \},  \\
& E_2 = \{ \omega \in \Omega: M^{\infty}(\omega)(\{t_j\} \times (a_{i,j}, \infty))  < i \mbox{ for } i \in \llbracket 1, r \rrbracket, j \in \llbracket 1, m \rrbracket \}.
\end{split}
\end{equation}

Suppose that $\omega \in E_1$. Fix $\varepsilon > 0$ and $\delta > 0$, such that $\delta < \min(t_{i} - t_{i-1}: i  \in \llbracket 2, m \rrbracket)$. From the pointwise convergence in (\ref{Eq.MultiWeak1}), we have for all large $v$:
$$X_i^{j,N_v}(\omega) < a_{i,j} + \varepsilon \mbox{ for }  i \in \llbracket 1, r \rrbracket, j \in \llbracket 1, m \rrbracket.$$
Combining the latter with (\ref{Eq.OrdY}) and (\ref{Eq.YMeasForm}), we conclude for any fixed $B > \varepsilon + \max_{i \in \llbracket 1, r \rrbracket, j \in \llbracket 1, m \rrbracket} a_{i,j}$, $i \in \llbracket 1, r \rrbracket$, $j \in \llbracket 1, m \rrbracket$ and all large $v$:
$$ M^{N_v}(\omega)\left(\{t_j\} \times [a_{i,j} + \varepsilon, \infty) \right)= M^{N_v}(\omega)\left((t_j - \delta, t_j + \delta) \times [a_{i,j} + \varepsilon, B) \right) \leq i - 1.$$
From the vague convergence in (\ref{Eq.MultiWeak1}) and \cite[Lemma 4.1]{kallenberg2017random}, we conclude for $i \in \llbracket 1, r \rrbracket, j \in \llbracket 1, m \rrbracket$
\begin{equation*}
\begin{split}
 &M^{\infty}(\omega)\left(\{t_j\} \times (a_{i,j} + \varepsilon, \infty) \right) \leq \lim_{B \rightarrow \infty} M^{\infty}(\omega)\left((t_j - \delta, t_j + \delta) \times (a_{i,j} + \varepsilon, B) \right) \\
 & \leq \lim_{B \rightarrow \infty} \liminf_{v \rightarrow \infty}M^{N_v}(\omega)\left((t_j - \delta, t_j + \delta) \times [a_{i,j} + \varepsilon, B) \right) \leq i - 1.
 \end{split}
 \end{equation*}
The last displayed equation implies $\omega \in E_2$ and so $E_1 \subseteq E_2$.

Suppose now that $\omega \in E_1^c$. Then, there exist $i \in \llbracket 1, r \rrbracket$, $j \in \llbracket 1, m \rrbracket$ and $\varepsilon > 0$, depending on $\omega$, such that $\tilde{X}^{j,\infty}_i(\omega) > a_{i,j} + \varepsilon$. From the pointwise convergence in (\ref{Eq.MultiWeak1}) we conclude for all large $v$:
$$X_i^{j,N_v}(\omega) \geq  a_{i,j} + \varepsilon \mbox{ and } X_1^{j,N_v}(\omega) \leq \tilde{X}_1^{j, \infty}(\omega) + \varepsilon.$$
Combining the latter with (\ref{Eq.OrdY}) and (\ref{Eq.YMeasForm}), we conclude for all large $v$:
$$ M^{N_v}(\omega)\left(\{t_j\} \times [a_{i,j} + \varepsilon, \tilde{X}_1^{j, \infty}(\omega) + \varepsilon] \right) \geq i.$$
From the vague convergence in (\ref{Eq.MultiWeak1}) and \cite[Lemma 4.1]{kallenberg2017random}, we find 
\begin{equation*}
\begin{split}
 &M^{\infty}(\omega)\left(\{t_j\} \times [a_{i,j} + \varepsilon, \tilde{X}_1^{j, \infty}(\omega) + \varepsilon] \right) \geq  \limsup_{v \rightarrow \infty}M^{N_v}(\omega)\left(\{t_j\} \times [a_{i,j} + \varepsilon, \tilde{X}_1^{j, \infty}(\omega) + \varepsilon] \right) \geq i.
 \end{split}
 \end{equation*}
The last displayed equation implies $\omega \in E_2^c$ and so $E_1^c \subseteq E_2^c$, concluding the proof of (\ref{Eq.MultiWeak6}).
\end{proof}

To effectively apply Lemma \ref{Lem.FDC}, it is necessary to establish tightness of the random variables $\{X^{j,N}_i\}_{N \geq 1}$ for each fixed $i \in \llbracket 1, r \rrbracket ,  j \in \llbracket 1, m \rrbracket$. By projecting the measures $M^N$ to $\{t_j\} \times \mathbb{R} \cong \mathbb{R}$ for some fixed $j \in \llbracket 1, m \rrbracket$, the original tightness question can be reformulated in terms of point processes on $\mathbb{R}$. The following result provides sufficient conditions under which tightness for such point processes on $\mathbb{R}$ can be concluded.
\begin{lemma}\label{Lem.tightcri} Fix $ r\in \mathbb{N}$. Suppose that $X^N=\left(X_i^N: i \geq 1\right)$ is a sequence of random elements in $[-\infty, \infty)^{\infty}$, such that 
\begin{equation}\label{Eq.tightcriE01}
X_1^N(\omega) \geq X_2^N(\omega) \geq \cdots, \mbox{ and }  X_1^N(\omega), \dots, X^N_r(\omega) \in \mathbb{R}.
\end{equation}
Denote by
\begin{equation}\label{Eq.tightcriE0}
M^N(\omega, A)=\sum_{i \geq 1} 1\left\{X_i^N(\omega) \in A\right\} 
\end{equation}
the corresponding random measures and suppose that $M^N$ are point processes on $\mathbb{R}$. Assume that
\begin{enumerate}
\item $M^N$ converge weakly to a point process $M$ on $\mathbb{R}$ as $N \rightarrow \infty$;
\item $\mathbb{P}(M(\mathbb{R}) \geq r)=1$;
\item $\lim _{a \rightarrow \infty} \limsup_{N \rightarrow \infty} \mathbb{P}\left(X_1^N \geq a \right)=0$.
\end{enumerate}
Then, $\{ X_k^N\}_{N \geq 1}$ forms a tight sequence of real random variables for each $k \in \llbracket 1, r \rrbracket$.
\end{lemma}
\begin{proof}
The proof we present is an adaptation of \cite[Proposition 2.21]{ED24a}.
From (\ref{Eq.tightcriE01}) and condition (3) we conclude that the sequence $X^N$ is tight in $[-\infty, \infty)^{\infty}$. Let $X^{N_v}$ be a subsequence that converges weakly to some $\hat{X}^{\infty}$, and note from (\ref{Eq.tightcriE01}) that $\hat{X}^{\infty}_1(\omega) \geq \hat{X}^{\infty}_2(\omega) \geq \cdots$. To prove the proposition, it suffices to show that
\begin{equation}\label{Eq.tightcriE3}
\mathbb{P}\left(\hat{X}_r^{\infty} = -\infty\right)=0.    
\end{equation}

By condition (1), the sequence $M^N$ is weakly convergent and hence forms a tight sequence of random elements in $\mathcal{M}_{\mathbb{R}}$ -- the space of locally bounded measures on $\mathbb{R}$ with the vague topology. Here, we implicitly used that $\mathcal{M}_{\mathbb{R}}$ is a Polish space and Prohorov's theorem \cite[Theorem 5.2]{Bill}. This means that the joint sequence $(X^{N_v},M^{N_v})$ is tight in $[-\infty, \infty)^{\infty} \times \mathcal{M}_{\mathbb{R}}$ (with the product topology), and thus precompact. By passing to a further subsequence, still denoted $N_v$, we may assume that $\left(X^{N_v}, M^{N_v}\right)$ converge weakly to some $\left(X^{\infty}, M^{\infty}\right)$. Furthermore, by Skorohod's representation theorem \cite[Theorem 6.7]{Bill}, we may assume that $( X^{N_v}, M^{N_v})$  and  $(X^{\infty}, M^{\infty} )$ are all defined on the same probability space  $(\Omega, \mathcal{F}, \mathbb{P})$ and for each $\omega \in \Omega$
\begin{equation}\label{Eq.tightcriE1}
M^{N_v}(\omega) \xrightarrow{v} M^{\infty}(\omega) \text { and } X_i^{N_v}(\omega) \rightarrow X_i^{\infty}(\omega) \in[-\infty, \infty).  
\end{equation}
Since weak limits are unique, $\hat{X}^{\infty}$ has the same distribution as $X^{\infty}$, and $M^{\infty}$ has the same distribution as $M$ (as in the statement of the lemma).

Let $U$ be the set of $\omega$, such that $M^{\infty}(\omega)(\mathbb{R})\geq r$. From condition (2), we have $\mathbb{P}(U)=1$. Fixing  $\omega \in U$, we can find $a<b$ with $a$ small and $b$ large enough so that
$$ M^{\infty}(\omega)((a, b))  \geq r.$$
By \cite[Lemma 4.1]{kallenberg2017random}, and the vague convergence in \eqref{Eq.tightcriE1}, there exists $v_0 \in \mathbb{N}$ (depending on $\omega)$, such that for $v \geq v_0$ we have
$$
M^{N_v}(\omega)((a, b)) \geq r.
$$
By combining the last statement with \eqref{Eq.tightcriE01} and \eqref{Eq.tightcriE0}, we conclude for all $\omega \in U$ and $v \geq v_0$ that $X_r^{N_v}(\omega) \geq a$, which by the pointwise convergence in \eqref{Eq.tightcriE1} implies
$$
U \cap\left\{X_r^{\infty}(\omega)=-\infty\right\}=\emptyset.
$$
The last equation, the fact that $X^{\infty} \overset{d}{=} \hat{X}^{\infty}$, and $\mathbb{P}(U)=1$ imply \eqref{Eq.tightcriE3}.
\end{proof}

%
\subsection{Finite-dimensional convergence of the top curve}\label{Section7.2} The goal of this section is to establish the following result.

\begin{proposition}\label{Prop.FinitedimEdge}
 Assume the same notation as in Definition \ref{Def.ScalingEdge}. Then, the sequence of random vectors $(Y_1^{1, N}, \dots, Y_1^{m,N})$ converges in the finite-dimensional sense to $\left(B_{\kappa_1-\kappa_0}, \dots, B_{\kappa_m- \kappa_0}\right)$, where $(B_t: t\geq 0)$ is a standard Brownian motion.
\end{proposition}
\begin{proof} We adopt the same notation as in Lemma \ref{Lem.PrelimitKernelEdge}, where for each $s \in \mathcal{T}$, we set $\theta_s = \theta_0$, $R_s = R_0$ as in Lemma \ref{Lem.BigContour} for $\kappa = s$. For clarity, the proof is divided into three steps. In the first step, we assume that $\{Y^{j,N}_1\}_{N \geq 1}$ is tight for each $j \in \llbracket 1, m \rrbracket$, and conclude the proposition by applying Lemma \ref{Lem.FDC}. In the second step, we prove this tightness by assuming the sequence is tight from above -- see (\ref{Eq.TailY1}) -- and invoking Lemma \ref{Lem.tightcri}. In the third step, we verify the tightness from above by estimating the upper-tails of the distributions of $\{Y_1^{j,N}\}_{N \geq 1}$.\\

{\bf \raggedleft Step 1.} We claim that 
\begin{equation}\label{Eq.TightY}
\mbox{ the sequence } \{Y_1^{j,N}\}_{N \geq 1} \mbox{ is tight for each $j \in \llbracket 1, m \rrbracket$.}
\end{equation}
We prove (\ref{Eq.TightY}) in the steps below. Here, we assume its validity and complete the proof of the proposition.\\

We aim to apply Lemma \ref{Lem.FDC} with $r = 1$, 
$$X_i^{j,N} = Y_i^{j,N},  \hspace{2mm} X_1^j = B_{\kappa_j - \kappa_0}, \hspace{2mm} t_j = \kappa_j \mbox{ for } i \geq 1, j \in \llbracket 1, m \rrbracket.$$
Note that condition (\ref{Eq.OrdY}) holds by the definition of $Y_i^{j,N}$, and since $r = 1$, condition (\ref{Eq.OrdY2}) is vacuously satisfied. In addition, the tightness of $\{ X_1^{j,N} \}_{N \geq 1}$ follows from (\ref{Eq.TightY}). Thus, it remains to show that the point processes $M^N$ from Lemma \ref{Lem.PrelimitKernelEdge} converge weakly to the point process $M$ formed by $\{(\kappa_j, B_{\kappa_j - \kappa_0}) : j \in \llbracket 1, m \rrbracket\}$.

By Proposition \ref{Lem.BMDPP} and a change of variables (see \cite[Proposition 2.13(5)]{ED24a} with $\phi(s,x) = (s + \kappa_0, x)$), the $M$ is a determinantal point process on $\mathbb{R}^2$ with reference measure $\mu_{\mathcal{T}} \times \
{Leb}$ and correlation kernel 
$$K^{\mathrm{det}}(s,x;t,y) = \kbm(s-\kappa_0, x; t - \kappa_0,y),$$
where $\kbm$ is as in (\ref{Eq.KBM}). From \cite[Lemma 5.9]{DY25} we conclude that $M$ is a Pfaffian point process on $\mathbb{R}^2$ with the same reference measure and correlation kernel
\begin{equation}\label{Eq.PfKernYLim}
K^{\mathrm{Pf}}(s,x;t,y) = \begin{bmatrix}
    0 & \kbm\left(s - \kappa_0,  x  ; t - \kappa_0, y  \right)\\
    - \kbm\left(t - \kappa_0,y ; s - \kappa_0, x  \right) & 0
\end{bmatrix}.
\end{equation}
From Lemma \ref{Lem.PrelimitKernelEdge} we know (for large $N$) that $M^N$ is a Pfaffian point process on $\mathbb{R}^2$ with correlation kernel $K^N$ as in (\ref{Eq:EdgeKerDecomp}) and reference measure $\mu_{\mathcal{T},\nu(N)}$. Proposition \ref{Prop.KernelConvTop} implies that for each fixed $s, t \in \mathcal{T}$, the kernels $K^{N}(s,\cdot;t,\cdot)$ converge uniformly over compact sets of $\mathbb{R}^2$ to $K^{\mathrm{Pf}}(s,\cdot;t,\cdot)$ from (\ref{Eq.PfKernYLim}). Using that $\kbm\left(s - \kappa_0,  x  ; t - \kappa_0, y  \right)$ is continuous in $x,y$ for fixed $s,t \in \mathcal{T}$ and hence locally bounded, it follows from \cite[Proposition 5.14]{DY25} that $M^N$ converges weakly to a Pfaffian point process $M^{\infty}$ with reference measure $\mu_{\mathcal{T}} \times \mathrm{Leb}$ and correlation kernel $K^{\mathrm{Pf}}$. As $M$ and $M^{\infty}$ are both Pfaffian with the same kernel and reference measure, they have the same law, see \cite[Proposition 5.8(3)]{DY25}. We thus conclude that $M^N$ converge weakly to $M$, and since all other conditions of Lemma \ref{Lem.FDC} are met, the lemma implies the desired result. \\

{\bf \raggedleft Step 2.} We claim that for each $j \in \llbracket 1, m \rrbracket$
\begin{equation}\label{Eq.TailY1}
\lim_{a \rightarrow \infty} \limsup_{N \rightarrow \infty} \mathbb{P}(Y_1^{j,N} \geq a) = 0.
\end{equation}
We establish (\ref{Eq.TailY1}) in the next step. Here, we assume its validity and complete the proof of (\ref{Eq.TightY}). \\

We aim to apply Lemma \ref{Lem.tightcri} with $r = 1$, and $X^N_i = Y^{j,N}_i$ for $i \geq 1$. Note that (\ref{Eq.tightcriE01}) holds by the definition of $Y^{j,N}_i$, and condition (3) in Lemma \ref{Lem.tightcri} is verified by (\ref{Eq.TailY1}). It remains to check conditions (1) and (2). 

Define the point processes $M^{j,N}$ on $\mathbb{R}$ by
\begin{equation}\label{Eq.Mjn}
M^{j, N}(A) = \sum_{i \geq 1} {\bf 1}\{Y_i^{j,N} \in A\}.
\end{equation}
From Lemma \ref{Lem.PrelimitKernelEdge} and \cite[Lemma 5.13]{DY25} we know (for large $N$) that $M^{j,N}$ is a Pfaffian point process on $\mathbb{R}$ with reference measure $\nu_{\kappa_j}(N)$ and correlation kernel $K^{j,N}(x,y) = K^N(\kappa_j, x; \kappa_j,y)$, where $K^N$ is as in (\ref{Eq:EdgeKerDecomp}). From Proposition \ref{Prop.KernelConvTop} we know that $K^{j,N}(x,y)$ converges uniformly over compact sets of $\mathbb{R}^2$ to the kernel
\begin{equation}\label{Eq.KerYLimSlice}
K^{j,\infty}(x,y) = \begin{bmatrix}
    0 & \kbm\left(\kappa_j - \kappa_0,  x  ; \kappa_j - \kappa_0, y  \right)\\
    - \kbm\left(\kappa_j - \kappa_0,  y  ; \kappa_j - \kappa_0, x  \right) & 0
\end{bmatrix},
\end{equation}
where $\kbm$ is as in (\ref{Eq.KBM}). Moreover, the measures $\nu_{\kappa_j}(N)$ converge vaguely to the Lebesgue measure on $\mathbb{R}$. Therefore, by \cite[Proposition 5.10]{DY25} we conclude that $M^{j,N}$ converge weakly to a Pfaffian point process $M^{j,\infty}$ with correlation kernel as in (\ref{Eq.KerYLimSlice}) and with reference measure $\Leb$. This verifies condition (1) in Lemma \ref{Lem.tightcri}.

From our work in Step 1 and \cite[Lemma 5.13]{DY25}, we know that the measure $M^{\mathrm{BM}}$, defined by $ M^{\mathrm{BM}} = {\bf 1}\{B_{\kappa_j - \kappa_0} \in A \}$, is a Pfaffian point process on $\mathbb{R}$ with correlation kernel $K^{j,\infty}$ and reference measure $\mathrm{Leb}$. As $M^{j,\infty}$ and $M^{\mathrm{BM}}$ are both Pfaffian with the same kernel and reference measure, they have the same law -- see \cite[Proposition 5.8(3)]{DY25}. As $M^{\mathrm{BM}}(\mathbb{R}) = 1$ almost surely, it follows that $M^{j,\infty}$ satisfies condition (2) in Lemma \ref{Lem.tightcri}. 

In conclusion, the sequence $\{Y^{j,N}_i\}_{i \geq 1}$ satisfies the assumptions of Lemma \ref{Lem.tightcri} with $r = 1$, from which (\ref{Eq.TightY}) follows.\\

{\bf \raggedleft Step 3.} In this step, we fix $j \in \llbracket 1, m \rrbracket$ and prove (\ref{Eq.TailY1}). If $M^{j,N}$ is as in (\ref{Eq.Mjn}), then for any $a \in \mathbb{R}$
\begin{equation}\label{Eq.TailBoundMoment}
\sum_{i \geq 1} \mathbb{P}\left(Y^{j,N}_i \geq a \right) = \mathbb{E}\left[ \sum_{i \geq 1} {\bf 1}\{Y_i^{j,N} \in  [a, \infty) \}   \right] = \mathbb{E}\left[M^{j,N}([a, \infty)) \right].
\end{equation}
As explained in Step 2, we have (for large $N$) that $M^{j,N}$ is a Pfaffian point process on $\mathbb{R}$ with reference measure $\nu_{\kappa_j}(N)$ and correlation kernel $K^{j,N}(x,y) = K^N(\kappa_j, x; \kappa_j,y)$, where $K^N$ is as in (\ref{Eq:EdgeKerDecomp}). From \cite[(2.13)]{ED24a} and \cite[(5.12)]{DY25} we know for any bounded Borel set $A \subset \mathbb{R}$
$$\mathbb{E}\left[M^{j,N}(A) \right] = \int_A K^N_{12}(\kappa_j, x; \kappa_j,x) \nu_{\kappa_j}(N)(dx).$$
Setting $A = [a,b]$ and letting $b \rightarrow \infty$, we obtain by Lemma \ref{Lem.PrelimitKernelEdge} and monotone convergence that
\begin{equation}\label{Eq.FactMomEdge}
\mathbb{E}\left[M^{j,N}([a, \infty)) \right] = \frac{1}{\sigmap N^{1/2}}\sum_{x \in \Lambda_{\kappa_j}(N), x \geq a_N} K^N_{12}(\kappa_j, x; \kappa_j,x),
\end{equation}
where $a_N = \min\{y \in \Lambda_{\kappa_j}(N): y \geq a\}$. Next, we can replace 
$$K^N_{12}(\kappa_j, x; \kappa_j,x) = I^N_{12}(\kappa_j, x; \kappa_j,x) + R^N_{12}(\kappa_j, x; \kappa_j,x),$$
on the right side of (\ref{Eq.FactMomEdge}), exchange the order of the sum and the integrals in the definitions of $I^N_{12}$ and $R^N_{12}$ from (\ref{Eq.DefIN12Edge}) and (\ref{Eq.DefRN12Edge}), and evaluate the resulting geometric series to obtain
\begin{equation*}
\begin{split}
&\mathbb{E}\left[M^{j,N}([a, \infty)) \right] = U_N(a) + V_N(a), \mbox{ where } U_N(a) = \frac{\sigmap^{-1} N^{-1/2} }{(2\pi \im)^{2}}\oint_{\Gamma_{\kappa_j, N}} \hspace{-3mm} dz \oint_{\gamma_{\kappa_j}} dw  \frac{F_{12}^N(z,w) H_{12}^N(z,w)}{1 - w/z} ,\\
& V_N(a) =  \frac{1}{2\pi \im} \oint_{\Gamma_{\kappa_j,N}} dz \frac{F_{12}^N(z,c) (zc-1) (1-qz)^{\kappa_j N - \lfloor \kappa_j N \rfloor}}{z(z^2-1)(1-qc)^{\kappa_j N - \lfloor \kappa_j N \rfloor}} \cdot \frac{1}{1 - c/z}.
\end{split}
\end{equation*}
Here, $F^N_{12}, H^N_{12}$ are defined as in (\ref{Eq.DefIN12Edge}) with $s =t = \kappa_j$ and $x = y = a_N$. We mention that the two geometric series involved are absolutely convergent due to (\ref{Eq.ContoursNestedEdge}). Combining the last displayed equation with (\ref{Eq.TailBoundMoment}), we see that to prove (\ref{Eq.TailY1}), it suffices to show:
\begin{equation}\label{Eq.EdgeFDRed1}
\limsup_{a \rightarrow \infty} \limsup_{N \rightarrow \infty} |U_N(a)| = 0, \mbox{ and } \limsup_{a \rightarrow \infty} \limsup_{N \rightarrow \infty} |V_N(a)| = 0.
\end{equation}
In the remainder of the step we verify (\ref{Eq.EdgeFDRed1}) using the estimates from Section \ref{Section5.1}. In the inequalities below we will encounter various constants $B_i,b_i > 0$ with $B_i$ sufficiently large, and $b_i$ sufficiently small, depending on $q, c, \mathcal{T}, \{\theta_\kappa: \kappa \in \mathcal{T}\}, \{R_{\kappa}: \kappa \in \mathcal{T}\}$ -- we do not list this dependence explicitly. In addition, the inequalities will hold provided that $N$ is sufficiently large, depending on the same set of parameters, which we will also not mention further.\\

From (\ref{Eq.ContoursNestedEdge}), for $z \in \Gamma_{\kappa_j, N}$, $w \in \gamma_{\kappa_j}$, and $x \geq 0$, we have
\begin{equation*} 
\left| \frac{1}{1 - w/z } \right| \leq \frac{1}{1 - z_c(\kappa_j)/c}, \mbox{ and } \left|e^{-  \sigmap x N^{1/2} \log (z/c) +  \sigmap x N^{1/2} \log(w/c)  } \right| \leq 1.
\end{equation*}
Combining these bounds with (\ref{Eq.S2BoundZClose}), (\ref{Eq.S2BoundZFar}), (\ref{Eq.S2BoundW}) and the second line in (\ref{Eq.HijBoundEdge}), we obtain for some $B_1, b_1 > 0$ and all $a \geq 0$
\begin{equation*} 
\left| \frac{F_{12}^N(z,w) H_{12}^N(z,w)}{1 - w/z} \right| \leq B_1 e^{-b_1 N},
\end{equation*}
where we recall that $F^N_{12}, H^N_{12}$ are as in (\ref{Eq.DefIN12Edge}) with $s =t = \kappa_j$ and $x = y = a_N \geq a$. The last bound and the bounded lengths of $\Gamma_{\kappa_j,N}, \gamma_{\kappa_j}$ imply the first statement in (\ref{Eq.EdgeFDRed1}).\\

From (\ref{Eq.ContoursNestedEdge}), for some $B_2, b_2 > 0$ and all $z \in \Gamma_{\kappa_j, N}$, $x \geq 0$, we have
\begin{equation}\label{Eq.EdgeFDB1} 
\left| \frac{1}{1 - c/z }\right| \leq B_2 N^{1/2}, \mbox{ and } \left|e^{-  \sigmap x N^{1/2} \log (z/c) +  \sigmap x N^{1/2} \log(c/c)  } \right| \leq B_2 e^{-b_2 x}.
\end{equation}
Let $\delta, \epsilon$ be as in the beginning of Section \ref{Section5.1} and let $\Gamma_{\kappa_j,N}(0)$ denote the part of $\Gamma_{\kappa_j,N}$, contained in the disc $\{z: |z - c| \leq \delta \}$. In addition, let $V_N^0(a)$ be defined analogously to $V_N(a)$, but with $\Gamma_{\kappa_j, N}$ replaced with $\Gamma_{\kappa_j, N}(0)$. Combining (\ref{Eq.EdgeFDB1}) with the third line in (\ref{Eq.F11BoundEdge}), the second line in (\ref{Eq.HijBoundEdge2}), and the bounded length of $\Gamma_{\kappa_j, N}$, gives for some $B_3,b_3 > 0$ and all $a \geq 0$
\begin{equation}\label{Eq.EdgeFDB2} 
\left| V_N(a) - V_N^0(a) \right| \leq B_3 e^{-b_3 N - b_2 a_N} \leq B_3 e^{-b_3 N - b_2 a}.
\end{equation}
Next, changing variables $z = c + \tilde{z} N^{-1/2}$, and applying (\ref{Eq.EdgeFDB1}), the fourth line in (\ref{Eq.F11BoundEdge}), and the second line in (\ref{Eq.HijBoundEdge2}) gives for some $B_4,b_4 > 0$ and all $a \geq 0$
\begin{equation}\label{Eq.EdgeFDB3} 
\left|V_N^0(a) \right| \leq B_4 e^{-b_2 a_N} \cdot \int_{\mathcal{C}^{\theta_{\kappa_j}}_0[r_{\kappa_j}]} e^{-b_4|\tilde{z}|^2} |d\tilde{z}|,
\end{equation}
where we recall that $|d\tilde{z}|$ denotes integration with respect to arc-length, $r_{\kappa_j} = \sec(\theta_{\kappa_j})$ and the contours $\mathcal{C}_{a}^{\phi}[r]$ were defined above (\ref{Eq.I11Vanish}). Equations (\ref{Eq.EdgeFDB2}) and (\ref{Eq.EdgeFDB3}) imply the second statement in (\ref{Eq.EdgeFDRed1}). This concludes the proof of (\ref{Eq.EdgeFDRed1}) and hence the proposition.
\end{proof}

%
\subsection{Proof of Theorem \ref{Thm.Main1}}\label{Section7.3} Assume the same parameters as in Definition \ref{Def.ParametersBulk}, and let $\mathbb{P}_N$ be the Pfaffian Schur process from Definition \ref{Def.SchurProcess} with parameters as in (\ref{Eq.HomogeneousParameters}). If $(\lambda^1, \dots, \lambda^N)$ has law $\mathbb{P}_N$, we define the $\llbracket 1, 2 \rrbracket$-indexed geometric line ensembles $\mathfrak{L}^{\mathrm{top}, N} = (L^{\mathrm{top},N}_1, L^{\mathrm{top},N}_2)$ on $\mathbb{Z}$ by
\begin{equation}\label{Eq.DLETop}
L^{\mathrm{top},N}_i(s)  = \begin{cases} \lambda^{N-s+1}_i - \lfloor (c-q)^{-1}qN\rfloor  &\mbox{ if } i \in \llbracket 1, 2 \rrbracket, s \in \llbracket 1, N \rrbracket,  \\
- \lfloor (c-q)^{-1}qN\rfloor  &\mbox{ if } i \in \llbracket 1, 2 \rrbracket, s \leq 0, \\
\lambda^1_i -  \lfloor (c-q)^{-1}qN\rfloor  &\mbox{ if } i \in \llbracket 1, 2 \rrbracket, s \geq N+1. 
\end{cases}
\end{equation}
As mentioned in Section \ref{Section6.1}, by interpolating the points $(s, L^{\mathrm{top},N}_i(s))$ for $s \in \mathbb{Z}$, we may view $\mathfrak{L}^{\mathrm{top}, N}$ as a random element in $C(\llbracket 1, 2 \rrbracket \times \mathbb{R})$. Define the rescaled process $\mathcal{L}^{\mathrm{top},N}_1 \in \mathcal{C}\left((\kappa_0,1)\right) $ by
\begin{equation}\label{Eq.ScaledDLETop}
\mathcal{L}^{\mathrm{top},N}_1(t) = \sigmap^{-1} N^{-1/2} \cdot \left(L^{\mathrm{top},N}_1(tN) -  \pp t N \right) \mbox{ for } t\in (\kappa_0,1).
\end{equation}
From Definition \ref{Def.TopCurveScaledLPP} and Lemma \ref{Prop.LPPandSchur}, we know that $\mathcal{L}^{\mathrm{top},N}_1$ has the same law as $\mathcal{U}_1^{\mathrm{top},N}$. Consequently, it suffices to prove that 
\begin{equation}\label{Eq.ScaledTopWeakConv}
\mathcal{L}^{\mathrm{top},N}_1 \Rightarrow W \mbox{ in } \mathcal{C}\left((\kappa_0,1)\right),
\end{equation}
where $W$ is as in the statement of the theorem.\\

We first verify that $\mathfrak{L}^{\mathrm{top}, N}$ satisfy the hypotheses of \cite[Theorem 5.1]{dimitrov2024tightness} with $p = \pp$, $\sigma = \sigmap$, $K = 1$, $K_N = 2$, $d_N = N$, $\hat{A}_N = \lfloor \kappa_0 N \rfloor$, $\hat{B}_N = N$, $\alpha = \kappa_0$, $\beta = 1$. Clearly,
$$d_N \rightarrow \infty, \hspace{2mm} \hat{A}_N/d_N \rightarrow \alpha, \hspace{2mm} \hat{B}_N/d_N \rightarrow \beta, \hspace{2mm} K_N \rightarrow K+ 1 \mbox{ as } N \rightarrow \infty,$$
verifying the first point in \cite[Theorem 5.1]{dimitrov2024tightness}. In addition, we note for any $m \in \mathbb{N}$ and $1 > \kappa_m > \kappa_{m-1} > \cdots > \kappa_1 > \kappa_0$, that for $j \in \llbracket 1, m \rrbracket$
\begin{equation}\label{Eq.DLETopToY}
\left|\sigmap^{-1} N^{-1/2} \cdot \left(L^{\mathrm{top},N}_1(\lfloor \kappa_j N \rfloor ) - \pp\kappa_j N \right) - Y^{j,N}_1\right| \leq \sigmap^{-1} N^{-1/2},
\end{equation}
where $(Y_1^{1, N}, \dots, Y_1^{m,N})$ is as in Definition \ref{Def.ScalingEdge}. Proposition \ref{Prop.FinitedimEdge} shows  $(Y_1^{1, N}, \dots, Y_1^{m,N})$ converge weakly to $\left(B_{\kappa_1-\kappa_0}, \dots, B_{\kappa_m- \kappa_0}\right)$, which implies that the sequence $\sigmap^{-1} N^{-1/2} \cdot \left(L^{\mathrm{top},N}_1(\lfloor \kappa N \rfloor ) - \pp \kappa N \right)$ is tight for each $\kappa \in (\kappa_0, 1)$, verifying the second point in \cite[Theorem 5.1]{dimitrov2024tightness}. Finally, Lemma \ref{Lem.InterlacingGibbs} establishes that $\mathfrak{L}$ from (\ref{Eq.SchurLE}) satisfies the interlacing Gibbs property from Definition \ref{DefSGP}. Since $\{L^{\mathrm{top},N}_i(s): i \in \llbracket 1, 2\rrbracket \mbox{ and } s \in \llbracket \hat{A}_N, \hat{B}_N \rrbracket \}$ is obtained by a deterministic vertical shift and restriction of $\mathfrak{L}$, it also satisfies the interlacing Gibbs property as a $\llbracket 1, 2 \rrbracket$-indexed ensemble on $\llbracket \hat{A}_N, \hat{B}_N \rrbracket$. This verifies the third point of \cite[Theorem 5.1]{dimitrov2024tightness}, from which we conclude that the sequence $\mathcal{L}^{\mathrm{top},N}_1$ is tight in $C((\kappa_0,1))$.

If $\mathcal{L}^{\infty}$ is any subsequential limit of $\mathcal{L}^{\mathrm{top},N}_1$, we know from (\ref{Eq.DLETopToY}) and Proposition \ref{Prop.FinitedimEdge} that $\mathcal{L}^{\infty}$ has the same finite-dimensional distribution as $W$. As finite-dimensional sets form a separating class in $C((\kappa_0, 1))$, see \cite[Example 1.3]{Bill}, we conclude $\mathcal{L}^{\infty} = W$. Overall, we see that the sequence $\mathcal{L}^{\mathrm{top},N}_1$ is tight and each subsequential limit agrees with $W$, which implies (\ref{Eq.ScaledTopWeakConv}).

%
\section{Weak convergence of the bottom curves}\label{Section8} The goal of this section is to prove Theorem \ref{Thm.Main2}. In Section \ref{Section8.1}, we establish the finite-dimensional convergence of the random vectors $(X^{j,N}_i: i \geq 2, j \in \llbracket 1, m \rrbracket)$ from Definition \ref{Def.ScalingBulk}; see Proposition \ref{Prop.FinitedimBulk}. In Section \ref{Section8.2} we introduce a sequence of geometric line ensembles $\mathfrak{L}^{\mathrm{bot}, N}$ in Definition \ref{Def.AuxLE}, state its weak convergence in Proposition \ref{Prop.ConvBottom}, and deduce Theorem \ref{Thm.Main2} from this result. In Section \ref{Section8.3}, we define an auxiliary family of line ensembles $\hat{\mathfrak{L}}^{\mathrm{bot}, N}$, and show their proximity to $\mathfrak{L}^{\mathrm{bot},N}$ in Proposition \ref{Prop.Proximity}. In Section \ref{Section8.4}, we prove Proposition \ref{Prop.ConvBottom} by combining the finite-dimensional convergence from Proposition \ref{Prop.FinitedimBulk}, the tightness criterion from \cite[Theorem 5.1]{dimitrov2024tightness} (applied to $\hat{\mathfrak{L}}^{\mathrm{bot}, N}$), and the proximity result from Proposition \ref{Prop.Proximity} to conclude tightness for $\mathfrak{L}^{\mathrm{bot}, N}$.

%
\subsection{Finite-dimensional convergence for bottom curves}\label{Section8.1} 

The goal of this section is to establish the following finite-dimensional convergence result.
\begin{proposition}\label{Prop.FinitedimBulk}
 Assume the same notation as in Definition \ref{Def.ScalingBulk}. For $i \geq 1$ and $j \in \llbracket 1, m \rrbracket$ set $\hat{X}_i^{j,N} = X_{i + 1}^{j,N}$. Then, the sequence of random vectors $\hat{X}^N = (\hat{X}_i^{j, N}: i\ge 1,  j \in \llbracket 1, m \rrbracket)$ converges in the finite-dimensional sense to $(\mathcal{A}_i(\fq t_j) - \fq^2 t_j^2: i\ge 1,  j \in \llbracket 1, m \rrbracket)$, where $\mathcal{A} = \{\mathcal{A}_i\}_{i \geq 1}$ is the Airy line ensemble from Definition \ref{Def.AiryLE}. 
\end{proposition}
\begin{proof} We adopt the same notation as in Lemma \ref{Lem.PrelimitKernelBulk}, with $\theta = \theta_0$, $R = R_0$, as in Lemma \ref{Lem.BigContour}. We also define the point processes $\hat{M}^N $ on $\mathbb{R}^2$ through
\begin{equation}\label{Eq.MTrunc}
\hat{M}^N (A) = \sum_{i \geq 1} \sum_{j = 1}^m {\bf 1}\{(t_j, \hat{X}_i^{j,N}) \in A\}.
\end{equation}
The structure of the proof is generally similar to that of Proposition \ref{Prop.FinitedimEdge} and for clarity is divided into four steps. In the first step, we assume that $\{X^{j,N}_r\}_{N \geq 1}$ is tight for each $r \geq 1, j \in \llbracket 1, m\rrbracket$, and that the measures $\hat{M}^N$ converge weakly. Under these assumptions, we prove the proposition by applying \cite[Proposition 2.19]{ED24a}. In the second step, we prove the weak convergence of $\hat{M}^N$, by leveraging the kernel convergence result from Proposition \ref{Prop.KernelConvBottom}. In the third step, we establish the tightness $\{X^{j,N}_r\}_{N \geq 1}$ by assuming the sequences $\{\hat{X}^{j,N}_1\}_{N \geq 1}$ are tight from above -- see (\ref{Eq.TailX2}) -- and invoking Lemma \ref{Lem.tightcri}. Finally, in the fourth step, we verify this upper tightness condition by estimating the upper-tails of the distributions of $\{\hat{X}_1^{j,N}\}_{N \geq 1}$.\\

{\bf \raggedleft Step 1.} We claim that 
\begin{equation}\label{Eq.TightX}
\mbox{ the sequence } \{\hat{X}_r^{j,N}\}_{N \geq 1} \mbox{ is tight for each $r \geq 1$, $j \in \llbracket 1, m \rrbracket$,}
\end{equation}
and
\begin{equation}\label{Eq.VaugeConvX}
\hat{M}^N \Rightarrow \hat{M},
\end{equation}
where $\hat{M}$ is a Pfaffian point process on $\mathbb{R}^2$ with reference measure $\mu_{\mathcal{T}} \times \Leb$ and correlation kernel
\begin{equation}\label{Eq.LimKerFD}
K^{\mathrm{Pf}}(s,x;t,y) = \begin{bmatrix}
    0 & \frac{f(s,x)}{f(t,y)} \cdot K^{\infty}\left(s,x;t,y\right)\\
    - \frac{f(t,y)}{f(s,x)} \cdot K^{\infty}\left(t,y;s,x\right) & 0
\end{bmatrix}.
\end{equation}
Here, the functions $f(s,x)$ and $K^{\infty}\left(s,x;t,y\right)$ are given by
\begin{equation}\label{Eq.LimKerFD2}
\begin{split}
&K^{\infty}\left(s,x;t,y\right) = K^{\mathrm{Airy}}\left(-\fq s,  x + \fq^2 s^2 ; - \fq t, y + \fq^2 t^2 \right), \hspace{2mm} f(s,x) = e^{2\fq^3s^3/3 + \fq s x},
\end{split}
\end{equation}
where we recall that $K^{\mathrm{Airy}}$ is the extended Airy kernel from Definition \ref{Def.AiryLE}. We prove (\ref{Eq.TightX}) and (\ref{Eq.VaugeConvX}) in the steps below. Here, we assume their validity and complete the proof of the proposition.\\

Applying \cite[Proposition 2.19]{ED24a} with $r = m$ and $X_i^{j,N} = \hat{X}_i^{j,N}$, we conclude that $\hat{X}^N = (\hat{X}_i^{j,N}: i \geq 1, j \in \llbracket 1, m \rrbracket)$ converges in the finite-dimensional sense to a vector $\hat{X}^{\infty} = (\hat{X}^{j, \infty}_i: i \geq 1, j \in \llbracket 1, m \rrbracket)$, and moreover, the random measure $\hat{M}^{\infty}$, defined as in (\ref{Eq.MTrunc}) for $N = \infty$, has the same distribution as $\hat{M}$. In view of \cite[Corollary 2.20]{ED24a}, to complete the proof of the proposition, it suffices to show that
\begin{equation}\label{Eq.MeasEqual}
\hat{M}  \overset{d}{=}  M,
\end{equation}
where the random measure $M$ is defined by 
\begin{equation}\label{Eq.MeasEqual2}
M(A) = \sum_{i \geq 1} \sum_{j = 1}^m {\bf 1}\{(t_j, \mathcal{A}_i(\fq t_j) - \fq^2 t_j^2) \in A\}.
\end{equation}

Using a kernel conjugation (see\cite[Proposition 5.8(4)]{DY25}), and \cite[Lemma 5.9]{DY25}, we know that $\hat{M}$ is a determinantal point process with reference measure $\mu_{\mathcal{T}} \times \Leb$ and correlation kernel 
 \begin{equation*}
 K^{\mathrm{Airy}}\left(-\fq s,  x + \fq^2 s^2 ; - \fq t, y + \fq^2 t^2 \right) = K^{\mathrm{Airy}}\left( \fq t, y + \fq^2 t^2; \fq s,  x + \fq^2 s^2 \right),
\end{equation*}
where the equality follows directly from (\ref{Eq.S1AiryKer}) upon changing variables $z \rightarrow -w$, $w \rightarrow -z$. By kernel transposition (see \cite[Proposition 2.13(4)]{ED24a}), we conclude $\hat{M}$ is a determinantal point process with reference measure $\mu_{\mathcal{T}} \times \Leb$ and correlation kernel
 \begin{equation}\label{Eq.LimKerFDDet}
K^{\mathrm{det}}(s,x;t,y) = K^{\mathrm{Airy}}\left(\fq s,  x + \fq^2 s^2; \fq t,  y + \fq^2 t^2 \right).
\end{equation}

On the other hand, if $\tilde{M}$ is as in (\ref{Eq.RMS1}) with $\mathsf{S} = \{\fq t_1, \dots, \fq t_m\}$, we have that $M = \tilde{M}\phi^{-1}$, where $\phi(s,x) = (\fq^{-1} s, x - s^2)$. From \cite[Proposition 2.13(5)]{ED24a}, we conclude that $M$ is determinantal with reference measure $\mu_{\mathcal{T}} \times \Leb$ and correlation kernel $K^{\mathrm{det}}(s,x;t,y) $ as in (\ref{Eq.LimKerFDDet}). As $M$ and $\hat{M}$ are both determinantal point processes with the same correlation kernel and reference measure, we conclude (\ref{Eq.MeasEqual}), cf. \cite[Proposition 2.13(3)]{ED24a}.\\

{\bf \raggedleft Step 2.} In this step, we prove (\ref{Eq.VaugeConvX}). We first note that $M^N$ from Lemma \ref{Lem.PrelimitKernelBulk} satisfies
\begin{equation}\label{Eq.MTrunc2}
M^N (A) = \hat{M}^N(A) + \sum_{j = 1}^m {\bf 1}\{(t_j, X_1^{j,N}) \in A \}.
\end{equation}
From Lemma \ref{Lem.PrelimitKernelBulk}, we know (for large $N$) that $M^N$ is a Pfaffian point process on $\mathbb{R}^2$ with correlation kernel $K^N$ as in (\ref{Eq:BulkKerDecomp}), and reference measure $\mu_{\mathcal{T},\nu(N)}$. Proposition \ref{Prop.KernelConvBottom} implies that for each fixed $s, t \in \mathcal{T}$, the kernels $K^{N}(s,\cdot;t,\cdot)$ converge uniformly over compact sets of $\mathbb{R}^2$ to $K^{\mathrm{Pf}}(s,\cdot; t, \cdot)$ from (\ref{Eq.LimKerFD}). Using that $K^{\mathrm{Pf}}(s,x;t,y)$ is continuous in $x,y$ for fixed $s,t \in \mathcal{T}$ and hence locally bounded, it follows from \cite[Proposition 5.14]{DY25} that $M^N$ converges weakly to $\hat{M}$. 

Next, from (\ref{Eq.XsBulk}), (\ref{Eq.DLETop}), and (\ref{Eq.ScaledDLETop}), we have
\begin{equation*}
X_1^{j,N} = \frac{\sigmap N^{1/6}}{\sigmaq \zc} \cdot  \mathcal{L}^{\mathrm{top},N}_1\left(N^{-1}\left(\lfloor \kappa N \rfloor + T_{t_j} \right)\right)  + \frac{1}{\sigmaq \zc N^{1/3}} \cdot C_{j,N},
\end{equation*}  
where 
\begin{equation*}
C_{j,N} = \pp \left( \lfloor \kappa N \rfloor + T_{t_j} \right)  + \lfloor (c-q)^{-1}qN\rfloor - \hq N - \pq T_{t_j} - 1 = N \cdot [\hp(\kappa) -\hq] + O(N^{2/3}),
\end{equation*}
and the constant in the big $O$ notation depends on $q,\kappa, c, \mathcal{T}$. From (\ref{Eq.ScaledTopWeakConv}) we know that 
$$\mathcal{L}^{\mathrm{top},N}_1\left(N^{-1}\left(\lfloor \kappa N \rfloor + T_{t_j} \right)\right)  \Rightarrow B_{\kappa - \kappa_0},$$
and from Definition \ref{Def.ParametersBulk} we have
$$\hp(\kappa) -\hq = \frac{q(c-q)(\sqrt{\kappa} - \sqrt{\kappa_0})^2}{(1-q^2)(1-qc)} > 0.$$
The last four equations together yield for each $j \in \llbracket 1, m \rrbracket$
\begin{equation}\label{Eq.XEscape}
X_1^{j,N} \Rightarrow \infty \mbox{ as } N \rightarrow \infty,
\end{equation}
which together with (\ref{Eq.MTrunc2}) and the weak convergence of $M^N$ to $\hat{M}$ give (\ref{Eq.VaugeConvX}).\\

{\bf \raggedleft Step 3.} We claim that for each $j \in \llbracket 1, m \rrbracket$
\begin{equation}\label{Eq.TailX2}
\lim_{a \rightarrow \infty} \limsup_{N \rightarrow \infty} \mathbb{P}(\hat{X}_1^{j,N} \geq a) = 0.
\end{equation}
We establish (\ref{Eq.TailX2}) in the next step. Here, we assume its validity and complete the proof of (\ref{Eq.TightX}). \\

Fix $r \in \mathbb{N}$ and $j \in \llbracket 1, m \rrbracket$. We aim to apply Lemma \ref{Lem.tightcri} for this choice of $r$ and $X_i^N = \hat{X}_i^{j,N}$ for $i \geq 1$. Note that (\ref{Eq.tightcriE01}) holds by the definition of $X^{j,N}_i$, and condition (3) in Lemma \ref{Lem.tightcri} is verified by (\ref{Eq.TailX2}). It remains to check conditions (1) and (2). 

Define the point processes $M^{j,N}, \hat{M}^{j,N}$ on $\mathbb{R}$ by
\begin{equation}\label{Eq.MjnX}
\hat{M}^{j, N}(A) = \sum_{i \geq 1} {\bf 1}\{\hat{X}_i^{j,N} \in A\}, \hspace{2mm} M^{j, N}(A) = \sum_{i \geq 1} {\bf 1}\{X_i^{j,N} \in A\} = \hat{M}^{j, N}(A) + {\bf 1}\{X_1^{j,N} \in A\}.
\end{equation}
From Lemma \ref{Lem.PrelimitKernelBulk} and \cite[Lemma 5.13]{DY25} we know (for large $N$) that $M^{j,N}$ is a Pfaffian point process on $\mathbb{R}$ with reference measure $\nu_{t_j}(N)$ and correlation kernel $K^{j,N}(x,y) = K^N(\kappa_j, x; \kappa_j,y)$, where $K^N$ is as in (\ref{Eq:BulkKerDecomp}). Proposition \ref{Prop.KernelConvBottom} implies that $K^{j,N}(x,y)$ converges uniformly over compact sets of $\mathbb{R}^2$ to the kernel
\begin{equation}\label{Eq.KerXLimSlice}
K^{j,\infty}(x,y) = K^{\mathrm{Pf}}(t_j,x;t_j,y),
\end{equation}
where $K^{\mathrm{Pf}}$ is as in (\ref{Eq.LimKerFD}). Furthermore, the measures $\nu_{t_j}(N)$ converge vaguely to the Lebesgue measure on $\mathbb{R}$. It then follows from \cite[Proposition 5.10]{DY25} that $M^{j,N}$ converge weakly to a Pfaffian point process $M^{j,\infty}$ with correlation kernel as in (\ref{Eq.KerXLimSlice}) and with reference measure $\Leb$. Combining the latter with (\ref{Eq.XEscape}), and (\ref{Eq.MjnX}) implies $\hat{M}^{j,N} \Rightarrow M^{j,\infty}$, which verifies condition (1) in Lemma \ref{Lem.tightcri}.

Arguing as in Step 1, we have that $M^{j,\infty}$ has the same distribution as the point process on $\mathbb{R}$
$$M^j(A) = \sum_{i \geq 1}{\bf 1}\{\mathcal{A}_i(\fq t_j) - \fq^2t_j^2 \in A\},$$
which we recognize as the Airy point process, shifted by $\fq^2t_j^2$. Since the Airy point process has almost surely infinitely many atoms, see \cite[(7.11)]{ED24a}, it follows that $\mathbb{P}(M^{j, \infty}(\mathbb{R}) = \infty) = 1$, verifying condition (2) in Lemma \ref{Lem.tightcri}.

In conclusion, $\{\hat{X}_i^{j,N} \}_{i \geq 1}$ satisfy the conditions of Lemma \ref{Lem.tightcri}, which implies (\ref{Eq.TightX}).\\

{\bf \raggedleft Step 4.} In this final step we fix $j \in \llbracket 1, m \rrbracket$ and prove (\ref{Eq.TailX2}). Using (\ref{Eq.MTrunc2}), and arguing as in (\ref{Eq.TailBoundMoment}) and (\ref{Eq.FactMomEdge}), we obtain
\begin{equation}\label{Eq.MomToTail}
\begin{split}
&\mathbb{P}\left(X^{j,N}_1 \geq a \right)  + \sum_{i \geq 1} \mathbb{P}\left(\hat{X}^{j,N}_i \geq a \right) = \sum_{i \geq 1} \mathbb{P}\left(X^{j,N}_i \geq a \right) \\
&  = \frac{1}{\sigmaq \zc N^{1/3}}\sum_{x \in \Lambda_{\kappa_j}(N), x \geq a_N} K^N_{12}(t_j, x; t_j,x),
\end{split}
\end{equation}
where $K^N_{12}$ is as in (\ref{Eq:BulkKerDecomp}) and $a_N = \min\{y \in \Lambda_{t_j}(N): y \geq a\}$. In addition, from (\ref{Eq.XEscape}) we have for each $a \in \mathbb{R}$
\begin{equation}\label{Eq.MomToTail2}
\lim_{N \rightarrow \infty} \mathbb{P}(X^{j,N}_1 \geq a) = 1.
\end{equation}
Combining (\ref{Eq.MomToTail}) with (\ref{Eq.MomToTail2}), we see that to prove (\ref{Eq.TailX2}), it suffices to show that 
\begin{equation}\label{Eq.FDBulkRed1}
\limsup_{a \rightarrow \infty} \limsup_{N \rightarrow \infty} \left|  \frac{1}{\sigmaq \zc N^{1/3}}\sum_{x \in \Lambda_{\kappa_j}(N), x \geq a_N} K^N_{12}(t_j, x; t_j,x) - 1\right| = 0.
\end{equation}

Our strategy is to adapt the argument from Step 3 of Proposition \ref{Prop.FinitedimEdge}. However, a direct use of the decomposition 
$$K^N_{12}(t_j, x; t_j,x) = I^N_{12}(t_j, x; t_j,x) + R^N_{12}(t_j, x; t_j,x),$$
as in (\ref{Eq:BulkKerDecomp}), results in a geometric series from $R^N_{12}$ that is not absolutely  convergent. To circumvent this, we derive an alternative expression for $K^N_{12}(t_j, x; t_j,x)$. Fix $r_{12}^z \in (1, q^{-1})$, $r_{12}^w  \in (c, r_{12}^z)$ and deform $\gamma_N$ to the circle $C_{r_{12}^w}$, and $\Gamma_N$ to $C_{r_{12}^z}$ in the definition of $I^N_{12}$ in (\ref{Eq.DefIN12Bulk}). In doing so, we cross the simple pole at $w = c$, which exactly cancels $R^N_{12}(t_j, x; t_j,x)$, yielding the identity:
\begin{equation*}
\begin{split}
&K^N_{12}(t,x;t,x) = \frac{1}{(2\pi \im)^{2}}\oint_{C_{r_{12}^z}} dz \oint_{C_{r_{12}^w}} dw F_{12}^N(z,w) H_{12}^N(z,w) ,
\end{split}
\end{equation*}
where $F_{12}^N$, $H_{12}^N$ are as in (\ref{Eq.DefIN12Bulk}) for $x = y$ and $s = t = t_j$. Using the last displayed equation, and evaluating the resulting geometric series gives
\begin{equation*}
\begin{split}
\frac{1}{\sigmaq \zc N^{1/3}}\sum_{x \in \Lambda_{\kappa_j}(N), x \geq a_N} K^N_{12}(t_j, x; t_j,x) =  \frac{1}{(2\pi \im)^{2}}\oint_{C_{r_{12}^z}} dz \oint_{C_{r_{12}^w}} dw \frac{F_{12}^N(z,w) H_{12}^N(z,w) }{1 - w/z},
\end{split} 
\end{equation*}
where $F_{12}^N$, $H_{12}^N$ are as in (\ref{Eq.DefIN12Bulk}) for $x = y = a_N$ and $s = t = t_j$. We may now deform the contours $C_{r_{12}^w}$ and $C_{r_{12}^z}$ back to $\gamma_N$ and $\Gamma_N$, as in Lemma \ref{Lem.PrelimitKernelBulk}, to obtain
\begin{equation}\label{Eq.K12NewDecomp}
\begin{split}
\frac{1}{\sigmaq \zc N^{1/3}}\sum_{x \in \Lambda_{\kappa_j}(N), x \geq a_N} K^N_{12}(t_j, x; t_j,x) =  V_N(a) + U_N(a) + 1,
\end{split} 
\end{equation}
where 
\begin{equation}\label{Eq.K12NewDecomp2}
\begin{split}
&V_N(a) = \frac{1}{2\pi \im} \oint_{\Gamma_N} dz \frac{F_{12}^N(z,c) (zc-1) (1-qz)^{\kappa N - \lfloor \kappa N \rfloor}}{z(z^2-1)(1-qc)^{\kappa N - \lfloor \kappa N \rfloor}} \cdot \frac{1}{1 - c/z}, \\
&U_N(a) = \frac{(\sigmaq \zc)^{-1} N^{-1/3}}{(2\pi \im)^{2}}\oint_{\Gamma_N} dz \oint_{\gamma_N} dw \frac{F_{12}^N(z,w) H_{12}^N(z,w) }{1 - w/z}.
\end{split} 
\end{equation}
We mention that the term $V_N(a)$ in (\ref{Eq.K12NewDecomp}) arose from crossing the simple pole at $w = c$ in the process of deforming $C_{r_{12}^w}$ to $\gamma_N$, and the $1$ comes from subsequently crossing the simple pole at $z = c$ when deforming $C_{r_{12}^z}$ to $\Gamma_N$. In view of (\ref{Eq.K12NewDecomp}), we see that to prove (\ref{Eq.FDBulkRed1}), it suffices to show
\begin{equation}\label{Eq.FDBulkRed2}
\limsup_{a \rightarrow \infty} \limsup_{N \rightarrow \infty} |U_N(a)| = 0, \mbox{ and } \limsup_{a \rightarrow \infty} \limsup_{N \rightarrow \infty} |V_N(a)| = 0.
\end{equation}
In the remainder of the step we verify (\ref{Eq.FDBulkRed2}) using the estimates from Section \ref{Section4.1}. In the inequalities below we will encounter various constants $B_i,b_i > 0$ with $B_i$ sufficiently large, and $b_i$ sufficiently small, depending on $q, \kappa, c, \mathcal{T}, \theta_0, R_0 $ -- we do not list this dependence explicitly. In addition, the inequalities will hold provided that $N$ is sufficiently large, depending on the same set of parameters, which we will also not mention further.\\

From (\ref{Eq.ContoursNestedBulk}), for some $B_1 > 0$ and all $z \in \Gamma_{N}$, $w \in \gamma_{N}$, and $x \geq 0$, we have
\begin{equation}\label{Eq.BulkFDB1}  
\left| \frac{1}{1 - w/z } \right| \leq B_1N^{1/3}, \hspace{2mm} \left| \frac{1}{1 - c/z } \right| \leq B_1, \hspace{2mm} \left|e^{-  \sigmaq \zc x N^{1/2} \log (z/c) +  \sigmaq \zc x N^{1/2} \log(w/c)  } \right| \leq B_1 e^{-b_1 x}.
\end{equation}
Let $\delta, \epsilon$ be as in the beginning of Section \ref{Section4.1}. Let $\Gamma_{N}(0)$ denote the part of $\Gamma_{N}$, within the disc $\{z: |z - c| \leq \delta \}$, and $\gamma_N(0)$ the part of $\gamma_N$, within the disc $\{z: |z - \zc| \leq N^{-1/12}\}$. In addition, let $U_N^0(a)$ be defined analogously to $U_N(a)$, but with $\Gamma_{N}, \gamma_N$ replaced with $\Gamma_{N}(0), \gamma_N(0)$.

Combining (\ref{Eq.BulkFDB1}) with (\ref{Eq.S1BoundZClose}), (\ref{Eq.S1BoundZFar}), the second line in (\ref{Eq.HijBound2}), and the bounded length of $\Gamma_{N}$, gives for some $B_2,b_2 > 0$ and all $a \geq 0$
\begin{equation*}
\left| V_N(a)  \right| \leq \exp \left(B_2 + N [\SFb(\zc) - \SFb(c)] -b_1 a_N \right) \leq \exp(B_2 -b_2 N - b_1 a_N),
\end{equation*} 
where in the last inequality we used that $\SFb(\zc) - \SFb(c) < 0$ from (\ref{Eq.DiffS}). This proves the second statement in (\ref{Eq.FDBulkRed2}).

Combining (\ref{Eq.BulkFDB1}) with (\ref{Eq.S1BoundZClose}), (\ref{Eq.S1BoundZFar}), (\ref{Eq.S1BoundWClose}), (\ref{Eq.S1BoundWFar}), the second line in (\ref{Eq.HijBound}), and the bounded lengths of $\Gamma_{N}, \gamma_N$, gives for some $B_3,b_3 > 0$ and all $a \geq 0$
\begin{equation}\label{Eq.TruncUNBulk}
\left| U_N(a) - U^0_N(a)   \right| \leq \exp \left(B_3 - b_3 N^{3/4}  - b_1 a_N \right)
\end{equation} 
Next, changing variables $z = \zc + \tilde{z} N^{-1/3}$, $w = \zc + \tilde{w} N^{-1/3}$ and applying (\ref{Eq.BulkFDB1}), (\ref{Eq.S1BoundZClose}), (\ref{Eq.S1BoundWClose}) and the second line in (\ref{Eq.HijBound}) gives for some $B_4,b_4 > 0$ and all $a \geq 0$ 
\begin{equation}\label{Eq.BulkFDB3} 
\left|U_N^0(a) \right| \leq B_4 e^{-b_2 a_N} \cdot \int_{\mathcal{C}^{\theta_{0}}_0[r_{0}]}|d\tilde{z}|  \int_{\mathcal{C}^{2\pi/3}_0}  e^{-b_4|\tilde{z}|^3 - b_4 |\tilde{w}|^3} |d\tilde{w}|,
\end{equation}
where we recall that $|d\tilde{z}|, |d\tilde{w}|$ denote integration with respect to arc-length, $r_{0} = \sec(\theta_{0})$ and the contours $\mathcal{C}_{a}^{\phi}[r]$ were defined above (\ref{Eq.I11Vanish}). Equations (\ref{Eq.TruncUNBulk}) and (\ref{Eq.BulkFDB3}) imply the first statement in (\ref{Eq.FDBulkRed2}). This concludes the proof of (\ref{Eq.FDBulkRed2}) and hence the proposition.
\end{proof}

%
\subsection{Proof of Theorem \ref{Thm.Main2}}\label{Section8.2} In this section we introduce a sequence of line ensembles $\mathfrak{L}^{\mathrm{bot},N}$ in Definition \ref{Def.AuxLE}, which are distributionally closely related to the line ensembles in Theorem \ref{Thm.Main2}. In Proposition \ref{Prop.ConvBottom}, whose proof is given in Section \ref{Section8.4}, we show that (upon appropriate rescaling) $\mathfrak{L}^{\mathrm{bot},N}$ converge weakly, and use this result to prove Theorem \ref{Thm.Main2}.

\begin{definition}\label{Def.AuxLE} Assume the same parameters as in Definition \ref{Def.ParametersBulk}, and let $\mathbb{P}_N$ be the Pfaffian Schur process from Definition \ref{Def.SchurProcess} with parameters as in (\ref{Eq.HomogeneousParameters}). If $(\lambda^1, \dots, \lambda^N)$ is distributed according to $\mathbb{P}_N$, we define the $\mathbb{N}$-indexed geometric line ensembles $\mathfrak{L}^N = \{L^N_i\}_{i \geq 1}$ on $\mathbb{Z}$ by
\begin{equation}\label{Eq.DLEBot}
L^{N}_i(s)  = \begin{cases} \lambda^{N - \lfloor \kappa N \rfloor - s + 1}_{i} - \lfloor  \hq N \rfloor  &\mbox{ if } i \geq 1, s + \lfloor \kappa N \rfloor  \in \llbracket 1, N \rrbracket,  \\
- \lfloor  \hq N \rfloor   &\mbox{ if } i \geq 1, s + \lfloor \kappa N \rfloor  \leq 0, \\
\lambda^1_{i} -  \lfloor \hq N \rfloor  &\mbox{ if } i \geq 1, s + \lfloor \kappa N \rfloor \geq N+1. 
\end{cases}
\end{equation}
As explained in Section \ref{Section6.1}, by linear interpolation we can view $\mathfrak{L}^{N}$ as random elements in $C(\mathbb{N} \times \mathbb{R})$. We also define the $\mathbb{N}$-indexed geometric line ensembles $\mathfrak{L}^{\mathrm{bot}, N} = \{L^{\mathrm{bot}, N}_i\}_{i \geq 1}$ on $\mathbb{Z}$ by setting $L_i^{\mathrm{bot}, N} = L^N_{i+1}$ for $i \geq 1$. Lastly, we define the scaled ensembles $\mathcal{L}^{N} = \{\mathcal{L}^{N}_i\}_{i \geq 1} \in C(\mathbb{N} \times \mathbb{R})$ and $\mathcal{L}^{\mathrm{bot},N} = \{\mathcal{L}^{\mathrm{bot},N}_i\}_{i \geq 1} \in C(\mathbb{N} \times \mathbb{R})$ by
\begin{equation}\label{Eq.ScaledDLEBot}
\begin{split}
&\mathcal{L}^{N}_i(t) = [\pq (1+ \pq)]^{-1/2} N^{-1/3} \cdot \left(L^{N}_i( tN^{2/3}) -  \pq t N^{2/3} \right) \mbox{ for } i \geq 1, t\in \mathbb{R} \mbox{, and } \\
& \mathcal{L}^{\mathrm{bot},N}_i(t) = \mathcal{L}^{N}_{i+1}(t) \mbox{ for } i \geq 1, t \in \mathbb{R}.
\end{split}
\end{equation}
\end{definition}
\begin{remark}\label{Rem.LClosetoX} We mention that if $\mathcal{T} = \{t_1, \dots, t_m\}$ and $X_i^{j,N}$ are as in Definition \ref{Def.ScalingBulk}, then for all large $N$, depending on $q,\kappa,c$ and $\mathcal{T}$, we have
\begin{equation}\label{Eq.LClosetoX}
\left| \mathcal{L}^{N}_i(\lfloor t_j N^{2/3} \rfloor ) - \frac{\sigmaq\zc}{\sqrt{\pq(1+\pq)}} \cdot X_{i}^{j,N} \right| \leq \frac{(i+\pq)\sigmaq\zc}{N^{1/3}\sqrt{\pq(1+\pq)}} \mbox{ for } i \geq 1, j \in \llbracket 1, m \rrbracket. 
\end{equation}
\end{remark}

We now state the main result we require for $\mathcal{L}^{\mathrm{bot},N}$, whose proof is given in Section \ref{Section8.4}.
\begin{proposition}\label{Prop.ConvBottom} Assume the same notation as in Definition \ref{Def.AuxLE}. The line ensembles $\mathcal{L}^{\mathrm{bot}, N}$ converge weakly to $\mathcal{L}^{\infty} = \{\mathcal{L}^{\infty}\}_{i \geq 1} \in C(\mathbb{N} \times \mathbb{R})$, defined by 
\begin{equation}\label{Eq.LimLEBulk}
\mathcal{L}^{\infty}_i(t) = (2\fq)^{-1/2} \cdot \left( \mathcal{A}_i(\fq t) - \fq^2 t^2 \right) \mbox{ for } i \geq 1, t \in \mathbb{R},
\end{equation}
where $\mathcal{A} = \{\mathcal{A}_i\}_{i \geq 1}$ is the Airy line ensemble from Definition \ref{Def.AiryLE}.
\end{proposition}

In the remainder of this section we establish Theorem \ref{Thm.Main2}.
\begin{proof}[Proof of Theorem \ref{Thm.Main2}] Combining Definition \ref{Def.BotCurveScaledLPP} and Proposition \ref{Prop.LPPandSchur}, we see that $\mathcal{U}^{\mathrm{bot},N}$ has the same law as $\mathcal{L}^{\mathrm{bot},N}$ from (\ref{Eq.ScaledDLEBot}). The result now follows from Proposition \ref{Prop.ConvBottom}.
\end{proof}

%
\subsection{An auxiliary line ensemble}\label{Section8.3} In view of (\ref{Eq.LClosetoX}) and Proposition \ref{Prop.FinitedimBulk}, we will be able to establish $\mathcal{L}^{\mathrm{bot}, N} \overset{f.d.}{\rightarrow} \mathcal{L}^{\infty}$ in Proposition \ref{Prop.ConvBottom}. Thus, the main difficulty in proving Proposition \ref{Prop.ConvBottom} lies in showing that $\mathcal{L}^{\mathrm{bot}, N}$ forms a tight sequence. In Step 1 of the proof of Theorem \ref{Thm.Main1} the tightness of the top curve followed from \cite[Theorem 5.1]{dimitrov2024tightness}; however, this result cannot be applied directly to $\mathfrak{L}^{\mathrm{bot}, N}$. The obstruction is that the interlacing Gibbs property satisfied by $\mathfrak{L}^{N}$ is lost when we restrict to the curves of index $i \geq 2$. To overcome this, we introduce auxiliary ensembles $\hat{\mathfrak{L}}^{\mathrm{bot}, N} = \{\hat{L}^{\mathrm{bot}, N}_i\}_{i = 1}^H$ in Definition \ref{Def.AuxLE2}. These ensembles are constructed so that (1) they satisfy the interlacing Gibbs property, and (2) they are close in distribution to $\{L^{\mathrm{bot}, N}_i \}_{i = 1}^H$, as shown in Proposition \ref{Prop.Proximity}. In Section \ref{Section8.4} we will prove the tightness of $\hat{\mathfrak{L}}^{\mathrm{bot}, N}$ by applying \cite[Theorem 5.1]{dimitrov2024tightness}, and then transfer this property to $\mathfrak{L}^{\mathrm{bot}, N}$ via Proposition \ref{Prop.Proximity}.

\begin{definition}\label{Def.AuxLE2} Assume the same notation as in Definition \ref{Def.AuxLE}. Fix $H \in \mathbb{N}$ and $d > 0$, and set $\hat{A}_N = \lfloor -dN^{2/3} \rfloor$, $\hat{B}_N = \lfloor  d N^{2/3} \rfloor$. We define the $\llbracket 1, H \rrbracket$-indexed geometric line ensemble $\hat{\mathfrak{L}}^{\mathrm{bot}, N} = \{L^{\mathrm{bot}, N}_i \}_{i = 1}^H$ on $\llbracket \hat{A}_N, \hat{B}_N \rrbracket$, such that for each measurable $A \subseteq C(\llbracket 1, H \rrbracket \times [\hat{A}_N, \hat{B}_N])$
\begin{equation}\label{Eq.AuxLELaw}
\mathbb{P}(\hat{\mathfrak{L}}^{\mathrm{bot}, N} \in A) = \mathbb{E}\left[ \mathbb{P}_{\ice, \mathrm{Geom}}^{\hat{A}_N,\hat{B}_N, \vec{x}, \vec{y}, \infty, g} \left( \mathfrak{Q} \in A \right) \right],
\end{equation}
where 
\begin{equation}\label{Eq.Boundary}
\begin{split}
&\vec{x} = \left(L^{\mathrm{bot}, N}_1(\hat{A}_N), \dots, L^{\mathrm{bot}, N}_{H}(\hat{A}_N) \right), \hspace{2mm} \vec{y} = \left(L^{\mathrm{bot}, N}_1(\hat{B}_N), \dots, L^{\mathrm{bot}, N}_{H}(\hat{B}_N) \right), \hspace{2mm} \\
& g(s) = L^{\mathrm{bot}, N}_{H+1}(s) \mbox{ for $s \in \llbracket \hat{A}_N, \hat{B}_N \rrbracket$},
\end{split}
\end{equation}
and $\mathfrak{Q}$ has law $\mathbb{P}_{\ice, \mathrm{Geom}}^{\hat{A}_N,\hat{B}_N, \vec{x}, \vec{y}, \infty, g}$ as in Definition \ref{DefAvoidingLawBer}.
\end{definition}
\begin{remark}\label{Rem.AuxLE} The ensemble $\hat{\mathfrak{L}}^{\mathrm{bot}, N}$ is obtained as follows. First sample $\mathfrak{L}^{\mathrm{bot}, N}$, and record the boundary data $(\vec{x}, \vec{y}, g)$ as in (\ref{Eq.Boundary}). Conditioned on this data, sample $\mathfrak{Q}$ according to the uniform measure on $\Omega_{\ice}(\hat{A}_N, \hat{B}_N, \vec{x}, \vec{y}, \infty, g)$, and set $\hat{\mathfrak{L}}^{\mathrm{bot}, N} = \mathfrak{Q}$. Since $\mathfrak{Q}$ satisfies the interlacing Gibbs property as a $\llbracket 1, H \rrbracket$-indexed line ensemble on $\llbracket \hat{A}_N, \hat{B}_N\rrbracket $, see \cite[Lemma 2.10]{dimitrov2024tightness}, the same holds for $\hat{\mathfrak{L}}^{\mathrm{bot}, N}$. We also mention that the law of $\hat{\mathfrak{L}}^{\mathrm{bot}, N}$ depends on $d$ and $H$, although this dependence is suppressed from the notation.
\end{remark}

We now turn to the main result we establish for $\hat{\mathfrak{L}}^{\mathrm{bot}, N}$.
\begin{proposition}\label{Prop.Proximity} Assume the notation from Definitions \ref{Def.AuxLE} and \ref{Def.AuxLE2} for some fixed $d > 0$ and $H \in \mathbb{N}$. Define the $\llbracket 1,H\rrbracket$-indexed geometric line ensemble $\tilde{\mathfrak{L}}^{\mathrm{bot},N} = \{\tilde{L}^{\mathrm{bot},N}_{i}\}_{i = 1}^H$ on $\llbracket \hat{A}_N, \hat{B}_N\rrbracket$, by
$$ \tilde{L}^{\mathrm{bot},N}_{i}(s) = L^{\mathrm{bot},N}_{i}(s) \mbox{ for } i \in \llbracket 1, H \rrbracket, s \in \llbracket \hat{A}_N, \hat{B}_N \rrbracket.$$
Then, for any sequence of measurable sets $A_N \in C(\llbracket 1, H \rrbracket \times [\hat{A}_N, \hat{B}_N])$, we have
\begin{equation}\label{Eq.Proximity}
\lim_{N \rightarrow \infty} \left| \mathbb{P}\left(\tilde{\mathfrak{L}}^{\mathrm{bot},N} \in A_N\right) - \mathbb{P}\left(\hat{\mathfrak{L}}^{\mathrm{bot},N} \in A_N \right)\right| = 0.
\end{equation}
\end{proposition}
\begin{proof} In the proof below, all constants will depend on $q, \kappa, c, d$ and $H$ implicitly, in addition to other parameters that are listed. We will not mention this further. For clarity, we split the proof into two steps.\\

{\bf \raggedleft Step 1.} Set $\tilde{n} = \hat{B}_N - \hat{A}_N$, and $V_N = \lfloor \pq \hat{A}_N \rfloor$ for brevity, and note that $\tilde{n} \sim 2d N^{2/3}$ as $N \rightarrow \infty$. Fix $\epsilon \in (0,1)$. We claim that there exist $M^{\mathrm{bot}}, W_1 > 0$, depending on $\epsilon$, such that for $N \geq W_1$ 
\begin{equation}\label{Eq.NoBigMax}
\mathbb{P}\left( L_{H+2}^{N}(s) - V_N - \pq (s - \hat{A}_N) > \tilde{n}^{1/2} M^{\mathrm{bot}} \mbox{ for some } s \in \llbracket \hat{A}_N, \hat{B}_N \rrbracket \right) < \epsilon.
\end{equation}
We will establish (\ref{Eq.NoBigMax}) in the second step. Here, we assume its validity and conclude the proof of the proposition.\\

From (\ref{Eq.LClosetoX}) and Proposition \ref{Prop.FinitedimBulk}, there exists $M^{\mathrm{side}} > 0$, depending on $\epsilon$, such that for all $N \geq 1$ and $i \in \llbracket 2, H + 1 \rrbracket$
\begin{equation}\label{Eq.EndsWellBehaved}
\mathbb{P}\left( \left| L_{i}^{N}(\hat{A}_N) - V_N \right| >   M^{\mathrm{side}}\tilde{n}^{1/2} \right) < \epsilon, \hspace{2mm}  \mathbb{P}\left( \left| L_{i}^{ N}(\hat{B}_N) - V_N - \pq \tilde{n} \right| >   M^{\mathrm{side}} \tilde{n}^{1/2} \right) < \epsilon.
\end{equation}
Let $M^{\mathrm{sep}}, N_2$ be as in Lemma \ref{GLELemma2}, with parameters $\epsilon, M^{\mathrm{side}}, M^{\mathrm{bot}}$ as above, $p = \pq$, $k = H + 1$. From (\ref{Eq.LClosetoX}) and (\ref{Eq.XEscape}), there exists $W_2 \geq W_1$, such that for $N \geq W_2$ we have $\tilde{n} \geq N_2$, and moreover,
\begin{equation}\label{Eq.TopIsHigh}
\mathbb{P}\left( L_{1}^{N}(\hat{A}_N) - V_N  <  M^{\mathrm{sep}} \tilde{n}^{1/2} \right) < \epsilon, \hspace{2mm}  \mathbb{P}\left(  L_{1}^{N}(\hat{B}_N) - V_N - \pq \tilde{n}  <  M^{\mathrm{sep}} \tilde{n}^{1/2} \right) < \epsilon.
\end{equation}
Lastly, let $G_N$ denote the set of triplets $(\vec{x}, \vec{y}, g) \in \mathfrak{W}_{H+1} \times \mathfrak{W}_{H+1} \times \Omega(\hat{A}_N, \hat{B}_N)$, such that 
\begin{equation}\label{Eq.BoundaryConditions}
\begin{split}
&x_1 - V_N \geq  M^{\mathrm{sep}} \tilde{n}^{1/2}, \hspace{2mm}  y_1 - V_N + \pq \tilde{n} \geq  M^{\mathrm{sep}} \tilde{n}^{1/2}, \hspace{2mm} \max_{i \in \llbracket 2, H + 1 \rrbracket }\left|x_i - V_N \right| \leq  M^{\mathrm{side}}\tilde{n}^{1/2}, \\
&\max_{i \in \llbracket 2, H + 1 \rrbracket }\left|y_i - V_N - \pq \tilde{n} \right| \leq M^{\mathrm{side}}\tilde{n}^{1/2}, \hspace{2mm} \max_{s \in \llbracket \hat{A}_N, \hat{B}_N \rrbracket} g(s) - V_N - \pq (s - \hat{A}_N) \leq  M^{\mathrm{bot}} \tilde{n}^{1/2}, \\
& \mathbb{P}\left(\vec{L}^N(\hat{A}_N) = \vec{x}, \vec{L}^N(\hat{B}_N) = \vec{y}, L^N_{H+2}\llbracket \hat{A}_N, \hat{B}_N \rrbracket = g \right) > 0,
\end{split}
\end{equation}
where $\vec{L}^N(s) = (L_1^N(s), \dots, L_{H+1}^N(s))$ for $s \in \llbracket \hat{A}_N, \hat{B}_N\rrbracket$, and $\Omega(\hat{A}_N, \hat{B}_N) = \sqcup_{x \leq y} \Omega(\hat{A}_N, \hat{B}_N,x,y)$ denotes the set of all increasing paths on $\llbracket \hat{A}_N, \hat{B}_N \rrbracket$ as in Definition \ref{DefDLE}. By Lemma \ref{GLELemma2}, for any $N \geq W_2$, $(\vec{x}, \vec{y}, g) \in G_N$, and any measurable $A_N \subseteq C(\llbracket 1, H \rrbracket \times [\hat{A}_N, \hat{B}_N])$, we have 
\begin{equation}\label{Eq.ApplyCoupling}
\left|\mathbb{P}_{\ice, \mathrm{Geom}}^{\hat{A}_N, \hat{B}_N, \vec{x}, \vec{y}, \infty, g} (\mathfrak{Q}^{\mathrm{bot}} \in A_N) - \mathbb{P}_{\ice, \mathrm{Geom}}^{\hat{A}_N, \hat{B}_N, \vec{u}, \vec{v}, \infty, g} (\mathfrak{Q} \in A_N) \right|  \leq \epsilon,
\end{equation}
where $u_i = x_{i+1}$, $v_i = y_{i+1}$, and in the first probability $\mathfrak{Q}^{\mathrm{bot}}(i,s) = \mathfrak{Q}(i+1, s)$ for $i \in \llbracket 1, H \rrbracket$ and $s \in \llbracket \hat{A}_N, \hat{B}_N\rrbracket$. We mention that to deduce (\ref{Eq.ApplyCoupling}) from (\ref{Eq.EqualEnsembles}), one shifts $\mathfrak{Q}$, $\tilde{\mathfrak{Q}}$ horizontally by $\hat{A}_N$ and vertically by $V_N$.  \\

Define the events
$$F_N = \left\{(\vec{L}^N(\hat{A}_N), \vec{L}^{N}(\hat{B}_N), L_{H+2}\llbracket \hat{A}_N, \hat{B}_N \rrbracket ) \in G_N \right\}.$$
Since $\mathfrak{L}^N$ satisfies the interlacing Gibbs property (by Lemma \ref{Lem.InterlacingGibbs}), and by the definitions of $\tilde{\mathfrak{L}}^{\mathrm{bot}, N}$, $\hat{\mathfrak{L}}^{\mathrm{bot}, N}$, we obtain for $N \geq W_2$
\begin{equation}\label{Eq.GibbsDecomp}
\begin{split}
&\left| \mathbb{P}\left(\{ \tilde{\mathfrak{L}}^{\mathrm{bot},N} \in A_N \} \cap F_N\right) - \mathbb{P}\left( \{ \hat{\mathfrak{L}}^{\mathrm{bot},N} \in A_N \} \cap F_N \right)\right|  \\
& \leq \sum_{(\vec{x}, \vec{y}, g) \in G_N} \left|\mathbb{P}_{\ice, \mathrm{Geom}}^{\hat{A}_N, \hat{B}_N, \vec{x}, \vec{y}, \infty, g} (\mathfrak{Q}^{\mathrm{bot}} \in A_N) - \mathbb{P}_{\ice, \mathrm{Geom}}^{\hat{A}_N, \hat{B}_N, \vec{u}, \vec{v}, \infty, g} (\mathfrak{Q} \in A_N) \right| \\
&\times  \mathbb{P}\left( \vec{L}^N(\hat{A}_N) = \vec{x}, \vec{L}^{N}(\hat{B}_N) = \vec{y}, L_{H+2}\llbracket \hat{A}_N, \hat{B}_N \rrbracket = g \right) \leq \epsilon, 
\end{split}
\end{equation}
where the last inequality follows from (\ref{Eq.ApplyCoupling}). From (\ref{Eq.NoBigMax}), (\ref{Eq.EndsWellBehaved}), and (\ref{Eq.TopIsHigh}), we also have
\begin{equation}\label{Eq.ProbOfGoodBound}
\mathbb{P}(F_N) \geq 1 - (2H+3) \epsilon.
\end{equation}
Combining (\ref{Eq.GibbsDecomp}) and (\ref{Eq.ProbOfGoodBound}), we conclude 
$$\limsup_{N \rightarrow \infty}  \left| \mathbb{P}\left( \tilde{\mathfrak{L}}^{\mathrm{bot},N} \in A_N \right) - \mathbb{P}\left(  \hat{\mathfrak{L}}^{\mathrm{bot},N} \in A_N \right)\right|  \leq (2H + 4) \epsilon.$$
As $\epsilon \in (0,1)$ was arbitrary, this proves (\ref{Eq.Proximity}).\\

{\bf \raggedleft Step 2.} In this step, we prove (\ref{Eq.NoBigMax}). We claim there exist constants $M^{\mathrm{bot}}, W_1 > 0$, depending on $\epsilon$, such that 
\begin{equation}\label{Eq.NoBigMax2}
\begin{aligned}
& \mathbb{P}\left(\max_{s \in \llbracket \hat{A}_N, 0 \rrbracket}\left(L_2^N(s)- V_N - \pq (s - \hat{A}_N)\right) >  M^{\mathrm{bot}} \tilde{n}^{1/2}\right)<\epsilon/2, \text { and } \\
&\mathbb{P}\left(\max_{s \in \llbracket 0, \hat{B}_N\rrbracket}\left(L_2^N(s) - V_N - \pq (s - \hat{A}_N)\right) > M^{\mathrm{bot}} \tilde{n}^{1/2} \right)<\epsilon/2.
\end{aligned}
\end{equation}
Since $L_2^N(s) \geq L^N_{H+2}(s)$, this implies (\ref{Eq.NoBigMax}). In the remainder we only establish the second line in (\ref{Eq.NoBigMax2}), the first one being handled analogously.\\

From (\ref{Eq.LClosetoX}) and Proposition \ref{Prop.FinitedimBulk}, we can find $R > 0$ large enough, depending on $\epsilon$, so that
\begin{equation}\label{eq:NoBigMaxE4}
\mathbb{P}(E(R))<\epsilon / 8 \text { and } \mathbb{P}(F(R))<\epsilon/8,
\end{equation}
where
$$E(R)=\left\{L_2^N(\hat{A}_N)- V_N   < -R (-\hat{A}_N)^{1/2} \right\}, \quad F(R)=\left\{L_2^N(0)- V_N + \pq \hat{A}_N \geq  R (-\hat{A}_N)^{1/2} \right\}.$$ 
Let $M > 0$ be as in Lemma \ref{GLELemma1} with $\epsilon$ equal to $\epsilon/8$ in the present setup and $p = \pq$, and fix any
\begin{equation}\label{Eq.TechIneq2}
M^{\mathrm{bot}} \geq 3R + 2M.
\end{equation}
Let $N_1$ be as in Lemma \ref{GLELemma1} for $\epsilon$ equal to $\epsilon/8$ in the present setup, $p = \pq$, $M^{\mathrm{side}}_1 = -R$ and $M^{\mathrm{side}}_2 = M^{\mathrm{bot}}$. We let $W_1 > 0$ be large, depending on $\epsilon, R$ and $M^{\mathrm{bot}}$, such that for $N \geq W_1$ we have $-\hat{A}_N \geq N_1$. The last paragraph specifies our choice of $M^{\mathrm{bot}}, W_1 $ for the remainder of the proof.\\

For $m \in \llbracket 0, \hat{B}_N \rrbracket$, we let $G_N^m$ denote the set of triplets $(\vec{x}, \vec{y}, g) \in \mathfrak{W}_{2} \times \mathfrak{W}_{2} \times \Omega(\hat{A}_N, m)$, such that 
\begin{equation}\label{Eq.NBMBoundary}
\begin{split}
&x_2 - V_N \geq  - R (-\hat{A}_N)^{1/2}, \hspace{2mm}  y_2 - V_N + \pq (m - \hat{A}_N) >  M^{\mathrm{bot}} \tilde{n}^{1/2} , \\
& \mathbb{P}\left(\vec{L}^N(\hat{A}_N) = \vec{x}, \vec{L}^N(m) = \vec{y}, L^N_{3}\llbracket \hat{A}_N, m \rrbracket = g \right) > 0,
\end{split}
\end{equation} 
where, similarly to Step 1, we have denoted $\vec{L}^N(s) = (L_1^N(s),L_{2}^N(s))$ for $s \in \llbracket \hat{A}_N, \hat{B}_N\rrbracket$, and by $\Omega(\hat{A}_N, m) = \sqcup_{x \leq y} \Omega(\hat{A}_N, m,x,y)$ the set of all increasing paths on $\llbracket \hat{A}_N, m \rrbracket$ as in Definition \ref{DefDLE}. In addition, for $m \in \llbracket 0, \hat{B}_N \rrbracket$ and $(\vec{x}, \vec{y}, g) \in G_N^m$, we define the events
$$F_{N}^m = \{L^N_2(s) -V_N - \pq (s - \hat{A}_N) \leq M^{\mathrm{bot}} \tilde{n}^{1/2} \mbox{ for } s \in \llbracket m+1, \hat{B}_N \rrbracket \},$$
$$E_N^m(\vec{x}, \vec{y}, g) = \{\vec{L}^N(\hat{A}_N) = \vec{x}, \vec{L}^N(m) = \vec{y}, L_{3}^N\llbracket \hat{A}_N, m \rrbracket = g  \}.$$
Note that $\{F_{N}^m \cap E_N^m(\vec{x}, \vec{y}, g): m \in \llbracket 0, \hat{B}_N \rrbracket,  (\vec{x}, \vec{y}, g) \in G_N^m \}$ are pairwise disjoint and also
\begin{equation}\label{Eq.Decomp1}
\begin{split}
&\mathbb{P} \left( \left\{ \max_{s \in \llbracket 0, \hat{B}_N\rrbracket}\left(L_2^N(s) - V_N - \pq (s - \hat{A}_N)\right) > M^{\mathrm{bot}} \tilde{n}^{1/2} \right\} \cap E(R)^c \cap F(R)^c \right) \\
&= \sum_{m \in \llbracket 0, \hat{B}_N \rrbracket} \sum_{(\vec{x}, \vec{y}, g) \in G_N^m} \mathbb{P}(F_N^m \cap E_N^m(\vec{x}, \vec{y}, g) \cap F(R)^c ).
\end{split}
\end{equation}

Using that $\mathfrak{L}^N$ satisfies the interlacing Gibbs property (by Lemma \ref{Lem.InterlacingGibbs}) and \cite[Lemma 2.11]{dimitrov2024tightness}, we have
\begin{equation*}
\begin{split}
&\mathbb{P}(F_N^m \cap E_N^m(\vec{x}, \vec{y}, g) \cap F(R)^c) \\
&= \mathbb{E}\left[\mathbf{1}_{F_N^m \cap E_N^m(\vec{x}, \vec{y}, g) } \cdot \mathbb{P}_{\ice,\mathrm{Geom} }^{\hat{A}_N, m, \vec{x}, \vec{y}, \infty, g}\left( Q_2(0) - V_N + \pq \hat{A}_N < R (-\hat{A}_N)^{1/2}\right)\right]
\end{split}
\end{equation*}
From Lemma \ref{GLELemma1}, with parameters as above, we have for $(\vec{x}, \vec{y}, g) \in G_N^m$ that
\begin{equation*}
\begin{split}
\mathbb{P}_{\ice,\mathrm{Geom} }^{\hat{A}_N, m, \vec{x}, \vec{y}, \infty, g}\left( Q_2(0) - V_N + \pq \hat{A}_N < (-\hat{A}_N + m)^{1/2} \cdot [(1-u_{N,m}) M^{\mathrm{side}}_1  + u_{N,m} M^{\mathrm{side}}_2 - M ]\right) < \epsilon/8,
\end{split}
\end{equation*}
where $u_{N,m} = -\hat{A}_N/(-\hat{A}_N + m)$. We mention that to obtain the last inequality from (\ref{Eq.GLELemma1}), one needs to set $n = -\hat{A}_N + m$, translate horizontally by $\hat{A}_N$ and vertically by $V_N$. Notice that by construction we have for $m \in \llbracket 0, \hat{B}_N \rrbracket$ that $u_{N,m} = -\hat{A}_N/(-\hat{A}_N + m) \in [1/2,1]$ and so
\begin{equation*}
\begin{split}
&(-\hat{A}_N + m)^{1/2} \cdot [(1-u_{N,m}) M^{\mathrm{side}}_1  + u_{N,m} M^{\mathrm{side}}_2 - M ] \geq  R(-\hat{A}_N)^{1/2} ,
\end{split}
\end{equation*}
where we used $M^{\mathrm{side}}_1 = - R$, and $M^{\mathrm{side}}_2 = M^{\mathrm{bot}} \geq 3R + 2M$ from (\ref{Eq.TechIneq2}). Combining the last three displayed equations, we obtain 
\begin{equation}\label{Eq.UBF1}
\mathbb{P}(F_N^m \cap E_N^m(\vec{x}, \vec{y}, g) \cap F(R)^c) < (\epsilon/8) \cdot \mathbb{P}(F_N^m \cap E_N^m(\vec{x}, \vec{y}, g)).
\end{equation}

Combining (\ref{eq:NoBigMaxE4}), (\ref{Eq.Decomp1}) and (\ref{Eq.UBF1}) gives for $N \geq W_1$
\begin{equation*}
\begin{split}
&\mathbb{P} \left( \max_{s \in \llbracket 0, \hat{B}_N\rrbracket}\left(L_2^N(s) - V_N - \pq (s - \hat{A}_N)\right) > M^{\mathrm{bot}} \tilde{n}^{1/2} \right) \leq \mathbb{P}(F(R)) + \mathbb{P}(E(R)) \\
&+ \sum_{m \in \llbracket 0, \hat{B}_N \rrbracket} \sum_{(\vec{x}, \vec{y}, g) \in G_N^m}  (\epsilon/8) \cdot \mathbb{P}(F_N^m \cap E_N^m(\vec{x}, \vec{y}, g)) \leq 3\epsilon/8.
\end{split}
\end{equation*}
The last inequality implies the second line in (\ref{Eq.NoBigMax2}).
\end{proof}

%
\subsection{Proof of Proposition \ref{Prop.ConvBottom}}\label{Section8.4} We adopt the same notation as in the statement of the proposition. From (\ref{Eq.LClosetoX}) and Proposition \ref{Prop.FinitedimBulk}, we deduce that for any $m \in \mathbb{N}$, $i_1, \dots, i_m \in \mathbb{N}$, and $t_1 < t_2 < \cdots < t_m$, we have
\begin{equation}\label{Eq.FDLEBulk}
\left(\mathcal{L}^{\mathrm{bot}, N}_{i_1}(\lfloor t_1 N^{2/3} \rfloor ), \dots, \mathcal{L}^{\mathrm{bot}, N}_{i_m}(\lfloor t_m N^{2/3} \rfloor ) \right) \overset{f.d.}{\rightarrow} \left(\mathcal{L}^{\infty}_{i_1}(t_1), \dots, \mathcal{L}^{\infty}_{i_m}(t_m) \right),
\end{equation}
where we also used that $\sigmaq\zc [\pq(1+\pq)]^{-1/2} = (2\fq)^{-1/2}$, which follows directly from (\ref{Eq.ParBulk1}). We now claim that 
\begin{equation}\label{Eq.BulkIsTight}
\mbox{ the sequence $\mathcal{L}^{\mathrm{bot},N}$ is tight.}
\end{equation} 
Note that (\ref{Eq.FDLEBulk}) implies all subsequential limits of $\mathcal{L}^{\mathrm{bot},N}$ have the same finite-dimensional distribution as $\mathcal{L}^{\infty}$. As finite-dimensional sets form a separating class (see \cite[Lemma 3.1]{DM21}), $\mathcal{L}^{\infty}$ is the only possible subsequential limit of $\mathcal{L}^{\mathrm{bot},N}$. Together with (\ref{Eq.BulkIsTight}), this yields the proposition. Thus, it remains to establish (\ref{Eq.BulkIsTight}).\\

In view of (\ref{Eq.FDLEBulk}) and the ``if '' part of \cite[Lemma 2.4]{DEA21}, to prove (\ref{Eq.BulkIsTight}), it suffices to show that for each $a < b$, $\epsilon > 0$ and $k \geq 1$, we have
\begin{equation}\label{Eq.BulkIsTight2}
\lim_{\delta \rightarrow 0+ }\limsup_{N \rightarrow \infty} \mathbb{P} \left( \sup_{x,y \in [a,b], |x-y| \leq \delta} \left| \mathcal{L}^{\mathrm{bot},N}_k(x)- \mathcal{L}^{\mathrm{bot},N}_k(y) \right| \geq \epsilon  \right) = 0.
\end{equation} 
Fix $M \geq k+1$ and $d > 0$ large enough so that $[a,b] \subset (-d,d)$, and let $\hat{\mathfrak{L}}^{\mathrm{bot}, N}$ be as in Definition \ref{Def.AuxLE2}. Extend $\hat{\mathfrak{L}}^{\mathrm{bot}, N}$ to a geometric line ensemble on $\mathbb{Z}$ by setting, for $i \in \llbracket 1, M \rrbracket$,
$$\hat{\mathfrak{L}}^{\mathrm{bot}, N}(i,s) = \hat{\mathfrak{L}}^{\mathrm{bot}, N}(i,\hat{B}_N) \mbox{ for } s \geq \hat{B}_N + 1, \mbox{ and } \hat{\mathfrak{L}}^{\mathrm{bot}, N}(i,s) = \hat{\mathfrak{L}}^{\mathrm{bot}, N}(i,\hat{A}_N) \mbox{ for } s \leq \hat{A}_N - 1.$$
We next verify that $\hat{\mathfrak{L}}^{\mathrm{bot}, N}$ satisfies the conditions of \cite[Theorem 5.1]{dimitrov2024tightness} with $K = M-1$, $K_N = M$, $d_N = N^{2/3}$, $p = \pq$, $\alpha = -d$, $\beta = d$, and $\hat{A}_N, \hat{B}_N$ as in the present setup. Indeed, one readily observes that 
$$d_N \rightarrow \infty, \hspace{2mm} \hat{A}_N/d_N \rightarrow \alpha, \hspace{2mm} \hat{B}_N/d_N \rightarrow \beta, \hspace{2mm} K_N \rightarrow K+ 1 \mbox{ as $N \rightarrow \infty$},$$
verifying the first point in \cite[Theorem 5.1]{dimitrov2024tightness}. To verify the second point, we seek to show that 
$$[\pq (1 + \pq)]^{-1/2} N^{-1/3} \cdot \left(  \hat{L}^{\mathrm{bot}, N}_i(\lfloor t N^{2/3} \rfloor) - \pq t N^{2/3} \right),$$
are tight for each $t \in (-d,d)$ and $i \in \llbracket 1, M \rrbracket$. However, from (\ref{Eq.FDLEBulk}) and Proposition \ref{Prop.Proximity}, the latter variables converge weakly to $\mathcal{L}^{\infty}_i(t)$, and are hence tight. Lastly, as mentioned in Remark \ref{Rem.AuxLE}, we have that $\hat{\mathfrak{L}}^{\mathrm{bot}, N}$ satisfies the interlacing Gibbs property as a $\llbracket 1, M \rrbracket$-indexed geometric line ensemble on $\llbracket \hat{A}_N, \hat{B}_N\rrbracket$, verifying the third point in \cite[Theorem 5.1]{dimitrov2024tightness}.

From \cite[Theorem 5.1]{dimitrov2024tightness}, we conclude that $\hat{\mathcal{L}}^{\mathrm{bot}, N} = \{\hat{\mathcal{L}}^{\mathrm{bot}, N}_{i}\}_{i = 1}^{M-1} \in C(\llbracket 1, M-1 \rrbracket \times (-d,d))$, defined through 
$$\hat{\mathcal{L}}^{\mathrm{bot}, N}_i(t) = [\pq (1 + \pq)]^{-1/2} N^{-1/3} (\hat{L}^{\mathrm{bot},N}_i(tN^{2/3}) - \pq t N^{2/3}) \mbox{ for } i \in \llbracket 1, M-1 \rrbracket, t \in (-d,d),$$
is tight. Applying the ``only if'' part of \cite[Lemma 2.4]{DEA21} and Proposition \ref{Prop.Proximity} gives (\ref{Eq.BulkIsTight2}).

\subsection*{Acknowledgments} The authors would like to thank the American Institute of Mathematics and the organizers, Leonid Petrov and Axel Saenz Rodriguez, of the AIM Workshop \emph{All roads to the KPZ universality class}, where part of this project was carried out. We are also grateful to Vincent Zhang for writing the code used to generate some of the figures in the text. ED was partially supported by the Simons Award TSM-00014004 by the Simons Foundation International.

\bibliographystyle{alpha}
\bibliography{PD}

\end{document}